\documentclass[12pt]{article}

\usepackage{amsmath,amssymb,amsthm,latexsym}

\input epsf
\usepackage{epsfig}

\usepackage{color}
\usepackage{comment}
\allowdisplaybreaks

\newtheorem{theorem}{Theorem}[section]
\newtheorem{thm}[theorem]{Theorem}
\newtheorem{corollary}[theorem]{Corollary}
\newtheorem{cor}[theorem]{Corollary}
\newtheorem{lemma}[theorem]{Lemma}
\newtheorem{proposition}[theorem]{Proposition}
\newtheorem{prop}[theorem]{Proposition}
\newtheorem{definition}[theorem]{Definition}
\newtheorem{remark}[theorem]{Remark}
\newtheorem{example}[theorem]{Example}
\numberwithin{equation}{section}

\def\be{\begin{equation}}
\def\ee{\end{equation}}
\def\bes{\begin{equation*}}
\def\ees{\end{equation*}}

\def\vp{{\varphi}}

\def\wt{\widetilde}
\def\wh{\widehat}

\def\eps{\varepsilon}
\def\lam{{\lambda}}
\def\ol{\overline}

\def\qed{{\hfill $\square$ \bigskip}}

\def\esssup{{\mathop{\rm ess \; sup \, }}}
\def\essinf{{\mathop{\rm ess \; inf \, }}}
\def\essosc{{\mathop{\rm ess \; osc \, }}}

\def\sE {{\cal E}} \def\sF {{\cal F}}\def\sL {{\cal L}}
\def\sN {{\cal N}}
\def\sS {{\cal S}}
\def\bE {{\mathbb E}} \def\bN {{\mathbb N}}
\def\bP {{\mathbb P}} \def\bR {{\mathbb R}} \def\bZ {{\mathbb Z}}

\def\sms{\smallskip}

\def\EHR{\mathrm{EHR}}
\def\PHR{\mathrm{PHR}}
\def\PHI{\mathrm{PHI}}
\def\UHK{\mathrm{UHK}}

\def\LHK{\mathrm{LHK}}
\def\HK{\mathrm{HK}}

\def\CSJ{\mathrm{CSJ}}
\def\SCSJ{\mathrm{SCSJ}}

\def\PI{\mathrm{PI}}
\def\FK{\mathrm{FK}}
\def\VD{\mathrm{VD}}
\def\RVD{\mathrm{RVD}}

\def\PHI{\mathrm{PHI}}

\def\UHKD{\mathrm{UHKD}}
\definecolor{dred}{rgb}{0.8, 0.0, 0.0}

\def\NDL{\mathrm{NDL}}
\def\UJS{\mathrm{UJS}}

\def\EHI{\mathrm{EHI}}

\def\E{\mathrm{E}}
\def\J{\mathrm{J}}
\def\<{\langle}
\def\>{\rangle}

\def\T{\mathrm{Tail}}

\begin{document}

\title{\bf Stability of parabolic Harnack inequalities
for symmetric non-local Dirichlet forms}

\author{{\bf Zhen-Qing Chen}\footnote{Research partially supported by
Simons Foundation grant 520542, a Victor Klee Faculty Fellowship at UW, and
NNSFC 11731009.},
\quad
{\bf Takashi Kumagai}\footnote{Research
partially supported by the Grant-in-Aid for Scientific Research (A) 25247007 and 17H01093.}
\quad and \quad  {\bf Jian Wang}\footnote{Research partially supported by
the National Natural Science Foundation of China (No.\ 11831014), the Program for Probability and Statistics: Theory and Application (No.\ IRTL1704), the Program for Innovative Research Team in Science and Technology in Fujian Province University (IRTSTFJ),
and the JSPS postdoctoral fellowship (26$\cdot$04021).
}
\vspace{1cm}\\
{\it In memory of Kazumasa Kuwada.}}
\date{}
\maketitle

\begin{abstract} In this paper,
we establish stability of parabolic Harnack inequalities
for symmetric non-local Dirichlet forms  on metric
measure spaces under general volume doubling condition.
We obtain
their stable equivalent characterizations in terms of the jumping
kernels, variants of cutoff Sobolev inequalities, and  Poincar\'e
inequalities. In particular, we establish
the connection between parabolic Harnack inequalities and
two-sided heat kernel estimates,
as well as with the H\"older regularity of parabolic functions
for symmetric non-local Dirichlet forms.
\end{abstract}



\section{Introduction and Main Results} \label{sec:intro}

Harnack inequalities are inequalities that control the growth of  non-negative harmonic functions and caloric functions
(solutions of heat equations) on domains.
The inequalities were first proved for harmonic functions
for Laplacian
in the plane by Carl Gustav Axel von Harnack, and later became fundamental in the theory of harmonic
analysis, partial differential equations and probability.
One of the most significant implications
of the inequalities
is that (at least for the cases of local operators/diffusions) they imply H\"older continuity of harmonic/caloric functions.
We refer readers to \cite{K} for the history and the basic introduction of Harnack inequalities.

Because of their fundamental importance,
there has been a long history of research on Harnack inequalities.
Harnack inequalities and H\"older regularities for harmonic functions are important
components of the celebrated
De Giorgi-Nash-Moser theory in harmonic analysis and partial differential equations.
In early 90's,
equivalent characterizations for parabolic Harnack inequalities (that is, Harnack inequalities for caloric functions) were obtained
by Grigor'yan \cite{Gr} and Saloff-Coste \cite{Sa0} for Brownian motions (or equivalently, Laplace-Beltrami operators) on complete Riemannian manifolds.
They showed
that parabolic Harnack inequalities are equivalent to
doubling condition of the volume measures plus Poincar\'e inequalities, which are also equivalent to the two-sided Gaussian-type heat
kernel estimates.
An important consequence of this equivalence is that the parabolic Harnack inequalities are stable under transformations of the
Riemannian manifolds by
quasi-isometry.
This result was later extended to symmetric diffusions on  metric measure spaces by Sturm \cite{St} and to random walks on graphs by
Delmotte \cite{De}. It has been further extended to symmetric anomalous diffusions on metric measure spaces  including fractals in \cite{BBK1}.

\medskip

In this paper, we consider the stability of parabolic Harnack inequalities for symmetric non-local
Dirichlet forms (or equivalent, symmetric jump processes)
on metric measure spaces. Let $(M, d, \mu)$ be a metric measure space where $d$ is a metric and $\mu$
is a Radon measure (see Section \ref{setup} for a precise setting).
We consider a symmetric regular \emph{Dirichlet form} $(\sE, \sF)$
on $L^2(M; \mu)$ of pure jump type; that is,
\begin{equation}\label{eq:DFregjump}
\sE(f,g)=\int_{M\times M\setminus \textrm{diag}}(f(x)-f(y))(g(x)-g(y))\,J(dx,dy),\quad f,g\in \sF,
\end{equation}
where \textrm{diag} denotes the diagonal set $\{(x, x): x\in M\}$
and $J(\cdot,\cdot)$ is a symmetric jumping
 measure on $M\times M\setminus \textrm{diag}$.
Let $X$ be Hunt process corresponding to $(\sE, \sF)$.
An important example of the jumping kernel $J$ is
$J(dx,dy)=\frac{c(x,y)}{d(x,y)^{d+\alpha}}\,\mu(dx)\,\mu(dy)$, where
$c(x,y)$ is a symmetric function bounded between two positive constants and $\alpha> 0$.
The corresponding process is called
a symmetric $\alpha$-stable-like process.
When $M=\bR^d$, or more general,  an Ahlfors $d$-regular space, $\mu$
is the Hausdorff measure on $M$
and $\alpha \in (0, 2)$,  various properties of the
  symmetric $\alpha$-stable-like processes including two-sided heat kernel estimates and parabolic
	Harnack inequalities
	have been studied in \cite{CK1}. In particular, when $M= \bR^d$, $\mu$ is the Lebesgue measure
	on $\bR^d$ and $c(x,y)$ is a constant function, this corresponds simply to a
	rotationally symmetric $\alpha$-stable L\'evy process. However, on some metric measure spaces $M$
	such as the Sierpinski gasket and the Sierpinski
	carpet, the index $\alpha$ can be larger
 than $2$; see Example \ref{E:5-1SG}.

\medskip

Let $\phi$ be a strictly increasing continuous function on $[0,\infty)$ with $\phi(0)=0$.

\begin{definition}\label{theorem:HI-PHI} \rm
We say that the {\it parabolic Harnack inequality} $\PHI(\phi)$ holds for the process $X$, if there exist constants $0<C_1<C_2<C_3<C_4$,  $0<C_5<1$ and  $C_6>0$ such that for every $x_0 \in M $, $t_0\ge 0$, $R>0$ and for
every non-negative function $u=u(t,x)$ on $[0,\infty)\times M$ that is
caloric (or space-time harmonic) in
 cylinder $Q(t_0, x_0,C_4\phi(R),R):=(t_0, t_0+C_4\phi(R))\times B(x_0,R)$,
\be\label{e:phidef}
  \esssup_{Q_- }u\le C_6 \,\essinf_{Q_+}u,
 \ee where $Q_-:=(t_0+C_1\phi(R),t_0+C_2\phi(R))\times B(x_0,C_5R)$ and $Q_+:=(t_0+C_3\phi(R), t_0+C_4\phi(R))\times B(x_0,C_5R)$.
\end{definition}
We call the function $\phi$ the scale function for $\PHI (\phi)$.
The $\PHI(\phi )$ results obtained in \cite{Gr, Sa, St, De} are for $\phi (r)=r^2$.
It is proved in \cite{CK1} that symmetric $\alpha$-stable-like processes with $\alpha\in (0,2)$ enjoy $\PHI(\phi)$ for $\phi(r)=r^\alpha$. In \cite{CK2}, $\PHI (\phi)$ is obtained for symmetric jump processes of mixed types
on metric measure spaces with variable scale $\phi$.

Here is the question we consider in this paper.

\begin{itemize}
\item[{\bf(Q)}]  Suppose $(\sE, \sF)$ and $(\wh \sE , \sF)$
are regular Dirichlet forms on $L^2(M; \mu)$ of the form
\eqref{eq:DFregjump},
whose corresponding jumping measures and processes are $J$, $\wh J$ and $X$, $\wh X$, respectively.
Suppose further there exist constants $c_1,c_2>0$ such that $c_1J(A,B)\le \wh J (A,B)\le c_2 J(A,B)$ for all
$A,B\subset M$ with $A\cap B=\emptyset$.
If $\PHI(\phi)$ holds for $X$, does $\PHI(\phi)$ also hold for the process $\wh X$?
\end{itemize}

Assume the metric measure space $(M, d, \mu)$ satisfies the volume doubling and reversed volume doubling condition;
see Definition \ref{D:1.4} for a precise definition.
In Theorem \ref{T:PHI}, the main result of this paper,  we will not only answer the question affirmatively
but also give
an equivalent characterization of $\PHI(\phi)$ that
is stable under such perturbations:
\begin{equation}\label{e:1.3}
\PHI (\phi) \Longleftrightarrow
 \PI(\phi) + \J_{\phi,\le} + \CSJ(\phi) + \UJS;
\end{equation}
see \eqref{eq:PIn}, \eqref{jsigm}, \eqref{e:csj1} and \eqref{ujs} below
for related notations and definitions.
Moreover, Theorem \ref{T:PHI} also gives
 the precise relations among
the parabolic Harnack inequality $\PHI (\phi)$,
 the H\"older regularity ${\rm PHR} (\phi)$ of caloric functions,
and the elliptic H\"older regularity (${\rm EHR}$) of harmonic functions:
\begin{equation*}
 \PHI (\phi)\Longleftrightarrow {\rm PHR} (\phi) +
 \E_\phi + \UJS
\Longleftrightarrow {\rm EHR}   +  \E_\phi + \UJS;
\end{equation*}
see \eqref{eq:ephiaq}, \eqref{e:phr} and \eqref{e:ehr}
for definitions.
As we will see from Examples \ref{E:1.2}-\ref{E:1.3},  characterization \eqref{e:1.3} also gives us
an effective tool to establish $\PHI (\phi)$ for a class of symmetric jump processes.

To our knowledge, there has been no literature on the equivalence of parabolic Harnack inequalities
for non-local Dirichlet forms on general metric measure spaces despite of the importance of parabolic Harnack inequalities. We note that when the underlying space is a graph satisfying the Ahlfors regular condition, some equivalence conditions for
$\PHI(\phi)$ with $\phi(r)=r^\alpha$ for $\alpha\in (0,2)$ are obtained in
Barlow, Bass and Kumagai \cite{BBK2}.
In some general metric measure spaces including certain
 fractals mentioned above,
it is known that $\PHI(\phi)$ may hold for $\phi(r)=r^\alpha$ with
$\alpha\ge 2$ (see, for instance, \cite[Section 6.1]{CKW1}).
In this paper, we establish the stability of $\PHI(\phi)$ for
a large class of scale functions $\phi$ including
those $\phi (r)=r^\alpha$ with $\alpha \geq 2$.
We also emphasize that our metric measure spaces are only assumed to
satisfy general volume doubling and reverse volume doubling properties;
see Definition \ref{thm:defvdrvd} for definitions.
These make
 the study of stability of $\PHI(\phi)$
extremely challenging.

The characterization \eqref{e:1.3} in particular implies that $\PHI (\phi)$ is invariant under time change of the symmetric jump process
$X$ by a positive continuous additive functional  $A_t = \int_0^t q(X_s) \,ds$ for some measurable function $q$ that is bounded between
two positive constants. This is because the time-changed process $(Y_t)_{t\ge0}=(X_{\tau_t})_{t\ge0}$ is an
$m$-symmetric
jump process on $M$ having the same jumping kernel $J(x, y)$, where $\tau_t=\inf\{s>0: A_s>t\}$ and
$m(dx):=q(x) \,\mu (dx)$.
Clearly the right hand side of \eqref{e:1.3} holds for $(M, d, \mu)$ and $J$ if and only if it holds for
$(M, d, m)$ and $J$.

We point out that the characterization \eqref{e:1.3} of $\PHI (\phi)$ is  new even in the Euclidean space case.
Suppose
that $(M, d)$ is the Euclidean space $\bR^d$, $\mu$
is a measure on $\bR^d$ that is comparable to the Lebesgue measure,
and $\phi (r)=\int_{\alpha_1}^{\alpha_2} r^\beta \,\nu (d\beta)$ or $\phi (r)=1/ \int_{\alpha_1}^{\alpha_2} r^{-\beta} \,\nu (d\beta)$,
where  $0<\alpha_1 <\alpha_2<2$ and $\nu$ is a probability measure on $[\alpha_1, \alpha_2]$.
Then, as a special case  of  \cite[Rermark 1.7]{CKW1}, $\CSJ (\phi)$ is implied by
 $\J_{\phi, \leq }$.
 In this case,  our result \eqref{e:1.3}
says that
\begin{equation}\label{e:1.4}
\PHI (\phi) \Longleftrightarrow
 \PI(\phi) + \J_{\phi,\le}   + \UJS.
 \end{equation}
In  \cite[Theorem 1.6]{BBK2} , \eqref{e:1.4} is proved
for continuous time random walks on
graphs that satisfy
the Ahlfors $d$-regular condition with  $\phi (r)=r^\alpha$ for $\alpha \in (0, 2)$.
 We will illustrate the utility of \eqref{e:1.4} in  Examples \ref{E:1.2} and \ref{E:1.3} below.

 \medskip

Parabolic Harnack inequalities are closely related to heat kernel estimates.
In the  recent paper \cite{CKW1}, we obtained stability of two-sided heat kernel estimates
and upper bound heat kernel estimates for symmetric jump processes of
mixed types
on general metric measure spaces (see Section \ref{HKE-sect} for a brief survey of the
results of \cite{CKW1}).
There are also recent work on the stability of two-sided heat kernel estimates for
stable-like jumps processes
with Ahlfors $d$-set condition in the framework of metric measure spaces \cite{GHH} and
in the framework of infinite connected locally finite graphs \cite{MS}.
In contrast to the cases of local operators/diffusions,
parabolic Harnack inequalities are
no longer equivalent to (in fact weaker than) the two-sided heat kernel estimates.
In fact Corollary \ref{C:1.25} of this paper asserts
$$ \HK (\phi)\Longleftrightarrow\PHI(\phi)+ \J_{\phi,\ge};
$$
see \eqref{jsigm} and \eqref{HKjum} for definitions.
This discrepancy is caused by the heavy tail of the jumping kernel.
This heavy tail phenomenon is also one of main sources of difficulties in analyzing non-local operators/jump processes.

\medskip

Due to the above difficulties and differences, obtaining the stability of $\PHI(\phi)$ for non-local operators/jump processes requires new ideas. Our approach contains the following two key ingredients, and both of them are highly non-trivial:
\begin{itemize}
\item[(i)] We
make full use of the probabilistic properties of jump process $X$ (in particular the L\'evy system of $X$ that describes how the process $X$ jumps) to connect $\PHI(\phi)$ with the properties of the associated heat kernel and jumping kernel.
For instance, $\UJS$
yielded by
a probabilistic consideration
and motivated by \cite{BBK2} in a graph setting plays one of key roles for the characterization of $\PHI(\phi)$ in the present
framework; see the main result of this paper, Theorem \ref{T:PHI}.

\item[(ii)] We adopt
some PDE's techniques from the recent study of fractional $p$-Laplacian operators in \cite{CKP1} to
derive some useful
properties of the process $X$. We emphasis that, to get the stability of $\PHI(\phi)$ in our general framework we should use
  cutoff Sobolev inequalities
$\CSJ (\phi)$ for non-local Dirichlet forms, instead of the fractional Poincar\'e inequalities or Sobolev inequalities in the existing literature (e.g.\ see \cite{CKP1, DK, K2}), since the latter two functional inequalities require some regularity of state space and non-local operators.
See the equivalence condition (7)
in Theorem \ref{T:PHI}.
\end{itemize}

\smallskip

 The following example, partly motivated by \cite{BBK2} in a graph setting,  illustrates the power of \eqref{e:1.3} characterizing $\PHI (\phi)$ even in the Euclidean space case. In this example, although for each fixed $x\in \bR^d$, the jumping kernel $J(x, y)$ vanishes outside a double cone in $\bR^d$ with apex at $x$,
  $\PHI(\phi)$ holds nevertheless.
 It also clearly indicates that only $\PHI(\phi)$ can not imply $\HK(\phi)$ without additional condition  on the lower bound of the jumping kernel $J$.

 \begin{example}\label{E:1.2}
{\bf($\PHI(\phi)$ holds but  $\HK(\phi)$ fails.)}\quad \rm
Let $M=\bR^d$, $\mu$ be the
Lebesgue measure on $\bR^d$,   $0<\alpha <2$ and $\phi (r)= r^\alpha$.
For $0<\theta< \pi/2$
 and $v\in \bR^d$ with $|v|=1$, define $A=\{h\in \bR^d: |(h/|h|,v)|\ge
\cos\theta\}$ and
\[
J(x,y)={\bf1}_A(x-y)|x-y|^{-d-\alpha}.
\]
 We can apply \eqref{e:1.4} to show that $\PHI(\phi)$ holds. Clearly   $\J_{\phi,\ge}$ does not hold and so $\HK(\phi)$ fails.
   In particular,
   caloric functions of the corresponding  symmetric jump process are jointly H\"older continuous. See Section \ref{Sectin-ex} for details.
\end{example}

When $d=2$, the above example is a special case of the following example
where the direction of the cones can vary but  $\PHI(\phi)$ still holds.
Note that, the non-degenerate part of the jumping kernel in the example below can not only be of stable-like but also be of mixed  stable-like type considered in \cite{CK2}.

  \begin{example}\label{E:1.3}  \rm
Let $M=\bR^2$, $\mu$ be a measure on $\bR^2$
that is comparable to the Lebesgue measure,
and $\phi (r)=1/\int_{\alpha_1}^{\alpha_2} r^{-\beta}\, \nu (d\beta)$, where $0<\alpha_1<\alpha_2 <2$ and
$\nu$ is a probability measure on $[\alpha_1, \alpha_2]$.
 Fix $\theta \in (0, \pi/2 )$
and a positive constant $c_\theta >0$.
Let $\xi$ be an increasing function  on $\bR$ so that $\xi (0) \in [0, 2\pi)$, $\xi (x+c_\theta ) = \xi (x) + 2\pi$   for every $x\in \bR$, and
\begin{equation}\label{eq:xixisin}
 \xi (x+r)-\xi (x) \in [\sin^{-1}r,  \,
\theta\,] \qquad \mbox{ for  any }  x\in \bR \hbox{ and }
r\in [ 0,  r_\theta]
\end{equation}
 with some  $r_\theta \in (0,  \sin \theta )$
 that  is small enough.
Define  $v(x)=(\cos (\xi (x_1)), \sin (\xi(x_1)))$ for $x=(x_1,x_2)\in \bR^2$, and a two-sided cone $\Gamma_{\theta}(x)$ with
apex angle $2 \theta$ and vertex at $x$ by
$$ \Gamma_{\theta}(x)=x+\left\{h\in \bR^d:  |\<  h/|h|,v(x) \>| \ge
\cos \theta \right\}.
$$
Note that $v(x)$ is a unit vector that depends on the first coordinate $x_1$ of $x$ and is $c_\theta$-periodic in $x_1$.

Let $J(x, y)$ be any measurable symmetric  kernel on $\bR^d \times \bR^d$ with the following property; there exists a constant $C\geq 1$ such that for all $x, y\in \bR^d$,
\begin{equation}\label{e:1.5}
 C^{-1} \frac{ {\bf 1}_{\Xi(x)}(y)+ {\bf 1}_{\Xi(y)}(x)}{ |x-y|^d \phi (|x-y|)}
 \leq  J(x,y)\leq   C \frac{{\bf 1}_{\Xi(x)}(y)+ {\bf 1}_{\Xi(y)}(x)}{ |x-y|^d \phi (|x-y|)},\end{equation}
 where $\Xi(x):=\Gamma_{\theta}(x)\cup \ol{B(x, 1)}$.
   We can again apply \eqref{e:1.4} to show that $\PHI (\phi)$ holds.
Consequently, bounded
 caloric functions for the corresponding symmetric jump process are jointly H\"older continuous.
See Section \ref{Sectin-ex} for details.
\end{example}

In addition, we show in Example \ref{E:5.2} that the trace process of Brownian motion on the Sierpinski
gasket on one side of the big triangle enjoys a bounded version of $\PHI (r^\alpha)$ for some
$\alpha \in (1, 2)$ but
$\HK (r^\alpha)$ fails on any bounded time interval $(0, T_0]$.

\medskip

Finally, we should mention that, even though non-local operators appear naturally in the study of stochastic processes with jumps,
there are huge amount of interests among analysts to study Harnack inequalities and related
properties for non-local operators; see \cite{CS, CKP1, CKP2, DK, K, K2, Sil} and the references therein.
Combining probabilistic methods with analytic methods in the study of heat kernel estimates
and parabolic Harnack inequalities for non-local operators proves to be quite powerful and fruitful,
as is the  case for this paper and for \cite{CKW1}.

\medskip

In the following, we give the framework of this paper in details and
present the main results of this paper. We also recall some theorems from \cite{CKW1}
that will be used in this paper.

\subsection{Setting}\label{setup}

Let $(M,d)$ be a locally compact separable metric space, and   $\mu$
a positive Radon measure on $M$ with full support.
A triple $(M, d, \mu)$ is called a \emph{metric measure space},
and we denote by
$\langle  \cdot , \cdot \rangle$ the inner product in $L^2(M; \mu)$.
For simplicity, we assume that $\mu (M)=\infty$ throughout the paper.
(See Remark \ref{rem:boun} below for further comments.)
Let us emphasize that we do not assume $M$ to be connected nor
 $(M, d)$ to be geodesic.

Let $(\sE, \sF)$ be a regular Dirichlet form on $L^2(M; \mu)$ given in \eqref{eq:DFregjump}.
We assume throughout this paper that, for each $x\in M$,  there is a kernel $J(x, dy)$ so that
\[
J(dx, dy) = J(x, dy)\, \mu (dx).
\]
In this paper, we will abuse notation and always take
the quasi-continuous version for an element of $\sF$ (note that since $(\sE,\sF)$ is regular,
each function in $\sF$ admits a quasi-continuous version).
Denote by  $\sL$ the (negative definite) $L^2$-\emph{generator} of $(\sE, \sF)$.
Let $\{P_t\}$  be
the associated \emph{semigroup} on $L^2(M; \mu)$. \sms Associated with the regular
Dirichlet form $(\sE, \sF)$ on $L^2(M; \mu)$ is an $\mu$-symmetric
\emph{Hunt process} $X=\{X_t, t \ge 0, \bP^x, x \in M\setminus
\sN\}$, where $\sN$ is a properly exceptional set for $(\sE, \sF)$ in that
 $\mu(\sN)=0$ and
$\bP^x( X_t \in \sN \hbox{ for some } t>0 )=0$ for all $x \in M\setminus\sN$. This Hunt process is unique up to a
properly exceptional set (see \cite[Theorem 4.2.8]{FOT}).
A more precise version of $\{P_t\}$ with better regularity properties can be obtained as follows:
for any bounded Borel measurable function $f$ on $M$,
$$ P_t f(x) = \bE^x f(X_t), \quad x \in M_0:=M \setminus \sN. $$
The \emph{heat kernel} associated with $\{P_t\}$ (if it exists) is a
measurable function
$p(t, x,y):M_0 \times M_0 \to (0,\infty)$ for every $t>0$, such that
\begin{align*}
\bE^x f(X_t) &= P_tf(x) = \int p(t, x,y) f(y) \, \mu(dy), \quad  x \in M_0,
f \in L^\infty(M;\mu), \\
p(t, x,y) &= p(t, y,x) \quad \hbox{ for all } t>0,\, x ,y \in M_0, \\
p(s+t , x,z) &= \int p(s, x,y) p(t, y,z) \,\mu(dy)
\quad \hbox{ for all } s, t>0 \hbox{ and }  x,z \in M_0.
\end{align*}
We call $p(t, x,y)$ the {\em
heat kernel} on $(M, d, \mu, \sE)$. Note that we can extend $p(t,x,y)$ to all
$x$, $y\in M$ by setting $p(t,x,y)=0$ if $x$ or $y$ is outside
$M_0$.

The goal of this paper is to present stable characterizations of
parabolic Harnack inequality for the symmetric
jump process $X$. To state our results precisely and show the relations between heat kernel estimates and parabolic Harnack inequalities, we need a number of
definitions and also recall the stable characterizations of
two-sided estimates and upper bound estimates for heat
kernels from \cite{CKW1}.

\begin{definition}\label{D:1.4}
{\rm Denote by $B(x,r)$ the ball in $(M,d)$ centered at $x$ with radius $r$, and set
\[
V(x,r) = \mu(B(x,r)).
\]
(i) We say that $(M,d,\mu)$ satisfies the {\it volume doubling property} ($\VD$) if
there exists a constant ${C_\mu}\ge 1$ such that
for all $x \in M$ and $r > 0$,
\be \label{e:vd}
 V(x,2r) \le {C_\mu} V(x,r).
 \ee
(ii) We say that $(M,d,\mu)$ satisfies the {\it reverse volume doubling property} ($\RVD$)
if there exist positive constants
$d_1$ and ${c_\mu}$ such that
for all $x \in M$ and $0<r\le R$,
\be \label{e:rvd}
\frac{V(x,R)}{V(x,r)}\ge {c_\mu}  \Big(\frac Rr\Big)^{d_1}.\ee
 } \label{thm:defvdrvd}\end{definition}

VD condition \eqref{e:vd} is equivalent to the following: there exist $d_2, {\wt C_\mu}>0$ so that
\be \label{e:vd2}
\frac{V(x,R)}{V(x,r)} \leq   {\wt C_\mu}  \Big(\frac Rr\Big)^{d_2} \quad \hbox{for all } x\in M \hbox{ and }
0<r\le R.
\ee
RVD condition \eqref{e:rvd} is equivalent to
the existence of positive constants ${l_\mu}$ and  ${\wt c_\mu}>1$
so that
\be \label{e:rvd2}
 V(x,{l_\mu}  r) \geq  {\wt c_\mu} V(x,r) \quad \hbox{for all } x\in M \hbox{ and } r>0.
 \ee
It is  known that $\VD$ implies
$\RVD$
if $M$ is connected and unbounded
(see, for example \cite[Proposition 5.1 and Corollary 5.3]{GH}).

\sms

Let $\bR_+ :=[0,\infty)$ and $\phi: \bR_+\to \bR_+$ be a strictly increasing continuous
function  with $\phi (0)=0$, $\phi(1)=1$ that satisfies the following: there exist $c_1,c_2>0$ and $\beta_2\ge \beta_1>0$ such that
\be\label{polycon}
 c_1 \Big(\frac Rr\Big)^{\beta_1} \leq
\frac{\phi (R)}{\phi (r)}  \ \leq \ c_2 \Big(\frac
Rr\Big)^{\beta_2}
\quad \hbox{for all }
0<r \le R.
\ee

\begin{definition}{\rm We say $\J_\phi$ holds if for any $x,y\in M$
there exists a non-negative symmetric
function $J(x, y)$ so that for $\mu\times \mu $-almost  all $x, y \in M$,
\begin{equation}\label{e:1.2}
J(dx,dy)=J(x, y)\,\mu(dx)\, \mu (dy),
\end{equation} and
\begin{equation}\label{jsigm}
 \frac{c_1}{V(x,d(x, y)) \phi (d(x, y))}\le J(x, y) \le \frac{c_2}{V(x,d(x, y)) \phi (d(x, y))}
 \end{equation}
 for some constants $c_2\geq c_1>0$.
We say that $\J_{\phi,\le}$ (resp. $\J_{\phi,\ge}$) if
\eqref{e:1.2} holds and
the upper bound (resp. lower bound) in \eqref{jsigm} holds.}
\end{definition}

For the non-local Dirichlet form $(\sE, \sF)$, we
define the carr\'e du-Champ operator $\Gamma (f, g)$
for $f, g\in \sF$ by
$$
\Gamma (f, g) (dx) = \int_{y\in M} (f(x)-f(y))(g(x)-g(y))\,J(d x,d y)  .
$$

 \subsection{Heat kernel estimates}\label{HKE-sect}
 The following $\CSJ (\phi)$ and $\SCSJ (\phi)$ conditions that control
the energy of cutoff functions are first introduced in \cite{CKW1}.
See \cite[Remark 1.6]{CKW1} for  background on these conditions.
Recall that $\phi$ is a strictly increasing continuous
function on $\bR_+$ satisfying $\phi (0)=0$, $\phi(1)=1$  and \eqref{polycon}.

\begin{definition} \rm
\begin{itemize}
\item[(i)] Let $U \subset V$ be open sets in $M$ with
$U \subset \ol U \subset V$.
We say a non-negative bounded measurable function $\vp$ is a {\it cutoff function for $U \subset V$},
if $\vp  = 1$ on $U$,  $\vp=0$ on $V^c$ and $0\leq \vp \leq 1$ on $M$.
\item[(ii)]
We say that $\CSJ(\phi)$ holds if there exist constants $C_0\in (0,1]$ and $C_1, C_2>0$
such that for every
$0<r\le R$, almost all $x\in M$ and any $f\in \sF$, there exists
a cutoff function $\vp\in \sF_b:=\sF\cap L^\infty(M,\mu)$
for $B(x,R) \subset B(x,R+r)$
so that
\be \label{e:csj1} \begin{split}
 \int_{B(x,R+(1+C_0)r)} f^2 \, d\Gamma (\vp,\vp)
\le &C_1 \int_{U\times U^*}(f(x)-f(y))^2\,J(dx,dy) \\
&+ \frac{C_2}{\phi(r)}  \int_{B(x,R+(1+C_0)r)} f^2  \,d\mu,
\end{split}
\ee
where $U=B(x,R+r)\setminus B(x,R)$ and $U^*=B(x,R+(1+C_0)r)\setminus B(x,R-C_0r)$.

\item[(iii)]
We say that $\SCSJ(\phi)$ holds if there exist constants $C_0\in (0,1]$ and $C_1, C_2>0$
such that for every
$0<r\le R$ and almost all $x\in M$, there exists
a cutoff function $\vp\in \sF_b$ for $B(x,R) \subset B(x,R+r)$ so that \eqref{e:csj1} holds for any $f\in \sF$.

\end{itemize}

\end{definition}

Clearly  $\SCSJ(\phi)  \Longrightarrow \CSJ(\phi)$.

\begin{remark}\rm \label{R:csjrem}
As is pointed out  in
\cite[Remark 1.7]{CKW1}, under $\VD$, \eqref{polycon} and $\J_{\phi,\le}$, $\SCSJ(\phi)$ always holds if $\beta_2<2$, where $\beta_2$ is the exponent in \eqref{polycon}. In particular, $\SCSJ(\phi)$  holds for $\phi (r)=r^\alpha$ always
when $0<\alpha< 2$.
\end{remark}

\medskip

We next introduce the Faber-Krahn inequality.
For any open set $D \subset M$, $\sF_D$ is defined to be the
$\| \cdot \|_{\sE_1}$-closure
in $\sF$ of  $\sF\cap C_c(D)$,
where
$\|\cdot\|_{\sE_1}^2=\sE(\cdot, \cdot)+\|\cdot\|_2^2$
Here $C_c(D)$ is the space of continuous functions on $M$ with compact support in $D$.
Define
$$
 \lam_1(D)
= \inf \left\{ \sE(f,f):  \,  f \in \sF_D \hbox{ with }  \|f\|_2 =1 \right\},
$$
the bottom of the Dirichlet spectrum of $-\sL$ on $D$.

\begin{definition}\label{Faber-Krahnww}
{\rm $(M,d,\mu,\sE)$ satisfies the {\em Faber-Krahn
inequality} $\FK(\phi)$, if there exist positive constants $C$ and
$\nu$ such that for any ball $B(x,r)$ and any open set $D \subset
B(x,r)$, \be \label{e:fki}
 \lam_1 (D) \ge \frac{C}{\phi(r)} (V(x,r)/\mu(D))^{\nu}.
\ee
} \end{definition}

\medskip

For a set
$A\subset M$, define the exit time $\tau_A = \inf\{ t >0 : X_t \in A^c \}$.

\begin{definition}{\rm We say that $\E_\phi$ holds if
there is a constant $c_1>1$ such that for all $r>0$ and  all $x\in M_0$,
\be\label{eq:ephiaq}
c_1^{-1}\phi(r)\le \bE^x [ \tau_{B(x,r)} ] \le c_1\phi(r).
\ee
We say that $\E_{\phi,\le}$ (resp. $\E_{\phi,\ge}$) holds
if the upper bound (resp. lower bound) in the inequality above holds. }
\end{definition}

\begin{definition}\label{D:1.11}   \rm \begin{description}
\item{\rm (i)} We say that $\HK(\phi)$ holds if there exists a heat kernel
$p(t, x,y)$
of the semigroup $\{P_t\}$ for $(\sE,\sF)$,
 which has the
following estimates for all $t>0$ and all $x,y\in M_0$,
\begin{equation}\label{HKjum}\begin{split}
c_1\Big(\frac 1{V(x,\phi^{-1}(t))} &\wedge
\frac{t}{V(x,d(x,y))\phi(d(x,y))}\Big)\\
&\le p(t, x,y) \\
&\le c_2\Big(\frac 1{V(x,\phi^{-1}(t))}\wedge
\frac{t}{V(x,d(x,y))\phi(d(x,y))}\Big),
\end{split}\end{equation} where $c_1, c_2>0$ are constants independent of $x,y\in M_0$ and $t>0$.
Here $\phi^{-1}(t)$ is the inverse function of the strictly increasing function
$t\mapsto \phi (t)$.

\item{(ii)} We say $\UHK(\phi)$ (resp. $\LHK(\phi)$) holds if the
upper bound (resp. the lower bound) in \eqref{HKjum} holds.

\item{(iii)}  We say $\UHKD(\phi)$ holds if there is a constant $c>0$ such that
$$
p(t, x,x)\le \frac c{V(x,\phi^{-1}(t))} \quad
\hbox{for all } t>0 \hbox{ and } x\in M_0 .
$$
\end{description}
\end{definition}

\medskip

It is pointed out in \cite[Remark 1.12]{CKW1} that
$$
 \frac 1{V(y,\phi^{-1}(t))}\wedge\frac{t}{V(y,d(x,y))\phi(d(x,y))}
\asymp  \frac 1{V(x,\phi^{-1}(t))}\wedge
\frac{t}{V(x,d(x,y))\phi(d(x,y))} .
$$
We may thus replace
$V(x, \phi^{-1}(t))$ and $V(x,d(x,y))$ by $V(y, \phi^{-1}(t))$ and $V(y,d(x,y))$
in \eqref{HKjum} by modifying the values of
$c_1$ and $c_2$. On the other hand, it follows from \cite[Theorem 1.13 and Lemma 5.6]{CKW1}  that  if $\HK(\phi)$ holds, then the heat kernel $p(t, x,y)$ is H\"{o}lder continuous on $(x,y)$ for every $t>0$, and so \eqref{HKjum} holds for all $x,y\in M$.

We say
{\it $(\sE, \sF)$ is conservative}   if its associated Hunt process $X$
has infinite lifetime.
This is equivalent to $P_t 1 =1$ a.e. on $M_0$ for every $t>0$.

\medskip

The following are the main results of \cite{CKW1},
which will be used later in this paper.

\begin{thm} {\rm (\cite[Theorem 1.13]{CKW1})} \label{T:main}
Assume that the metric measure space $(M, d , \mu)$ satisfies $\VD$ and $\RVD$, and $\phi$ satisfies \eqref{polycon}.
Then the following are equivalent: \\
$(1)$ $\HK(\phi)$. \\
$(2)$ $\J_\phi$ and $\E_\phi$.  \\
$(3)$ $\J_\phi$ and $\SCSJ(\phi)$.\\
$(4)$ $\J_\phi$ and $\CSJ(\phi)$.
\end{thm}

\begin{thm}  {\rm (\cite[Theorem 1.15]{CKW1})} \label{T:main-1}
Assume that the metric measure space $(M, d, \mu)$ satisfies $\VD$ and $\RVD$, and $\phi$ satisfies \eqref{polycon}.
Then the following are equivalent: \\
$(1)$ $\UHK(\phi)$ and
$(\sE, \sF)$ is conservative. \\
$(2)$ $\UHKD(\phi)$, $\J_{\phi,\le}$ and $\E_\phi$.\\
$(3)$ $\FK(\phi)$, $\J_{\phi,\le}$ and $\SCSJ(\phi)$.\\
$(4)$ $\FK(\phi)$, $\J_{\phi,\le}$ and $\CSJ(\phi)$.
\end{thm}

As a consequence of \cite[Proposition 3.1(ii)]{CKW1}
(recalled in Proposition  \ref{P:3.1} of this paper),
$\LHK (\phi)$ implies that $X$ has infinite lifetime.
As is remarked in \cite{CKW1},   $\UHK(\phi)$ alone
does not imply the conservativeness of the associated Dirichlet form $(\sE, \sF)$.

\subsection{Parabolic Harnack inequalities}

We first give probabilistic definitions of harmonic and
caloric functions
in the general context of metric measure spaces.

Let $Z:=\{V_s,X_s\}_{s\ge0}$ be the space-time process corresponding to $X$ where $V_s=V_0-s$.
The filtration generated by $Z$ satisfying the usual conditions will be denoted by $\{\widetilde{\mathcal{F}}_s;s\ge0\}$. The law of the space-time process $s\mapsto Z_s$ starting from $(t,x)$ will be denoted by $\bP^{(t,x)}$. For every open subset $D$ of $[0,\infty)\times M$, define
$\tau_D=\inf\{s>0:Z_s\notin D\}.$

Recall that
a set $A\subset [0, \infty)\times M$ is said to be nearly Borel measurable
if for any probability measure
$\mu_0$ on $[0, \infty)\times M$, there are
Borel measurable subsets $A_1$, $A_2$ of $[0, \infty)\times M$ so that
$A_1\subset A\subset A_2$ and that
$\bP^{\mu_0} (Z_t\in A_2\setminus A_1
\hbox{ for some } t\geq 0)=0$. The collection of all nearly Borel measurable subsets
of $[0, \infty)\times M$ forms a $\sigma$-field, which is called nearly Borel measurable
$\sigma$-field.

\begin{definition}\label{pro-har}  \rm \begin{description}
\item{(i)}
We say that a nearly  Borel measurable function $u(t,x)$ on
$[0,\infty)\times M$ is
\emph{caloric} (or \emph{space-time harmonic}) on
$D=(a,b)\times B(x_0,r)$ for the Markov process $X$
if there is a properly
exceptional set $\mathcal{N}_u$ of the Markov process $X$ so that for every
relatively compact open subset $U$ of $D$,
$u(t,x)=\bE^{(t,x)}u(Z_{\tau_{U}})$ for every $(t,x)\in
U\cap([0,\infty)\times (M\backslash \mathcal{N}_u)).$

\item{\rm (ii)} A nearly Borel measurable function $u$ on $M$ is said to be \emph{subharmonic}  (resp. \emph{harmonic, superharmonic})
in ${D}$ (with respect to the process $X$)
if for any relatively compact subset $U\subset D$,
$t\mapsto   u (X_{t\wedge \tau_U}) $ is a uniformly integrable submartingale
(resp. martingale, supermartingale) under $\bP^x$ for q.e. $x\in U$.
\end{description}
\end{definition}

\begin{remark}\rm  Concerning the definition of the space-time process $Z:=\{V_s,X_s\}_{s\ge0}$, the time evolves as $V_s=V_0+s$ in \cite[p.\! 37]{CK1} and \cite[p.\! 307]{CK2}, which is opposed to $V_s=V_0-s$ in the present paper (as well as
in \cite{BBK2,CKK1, CKK2}). The advantage of using time backwards (i.e., $V_s=V_0-s$ for all $s>0$) is due to that $u(t,x)=P_tf(x)$ is an example of caloric function. Indeed, for every $(t,x)\in [0,\infty)\times M$ and bounded measurable function $f$, let $u(t,x)=P_tf(x)$.
We have by the Markov property of $X$ that for any $(t_0,x_0)\in [0,\infty)\times M$ and $0<s<t$,
\begin{align*}\bE^{(t_0,x_0)}(u(Z_t)| \widetilde{\mathcal{F}}_s))&=\bE^{x_0}(u(t_0-t,X_t)|{\mathcal{F}}_s)=\bE^{X_s}u(t_0-t,X_{t-s})\\
&=\bE^{X_s}P_{t_0-t} f(X_{t-s})=P_{t_0-s}f(X_s)=u(Z_s),\end{align*} which implies that $u(t,x)$ on
$[0,\infty)\times M$ is caloric.  Similarly, if the heat kernel $p(t,x,y)$ exists, then we can prove that $(t,x)\mapsto p(t,x,y_0)$ is caloric on $(0,\infty)\times M$ for any fixed $y_0\in M$. (Note that, in contrast with the present paper, $(t,x)\mapsto p(t_0-t,x,y_0)$ is caloric on $[0,t_0)\times M$ in the time forwards case, see \cite[Lemma 4.5]{CK1}.) This causes the corresponding difference of the definition for the parabolic H\"older regularity (see Definition \ref{PER} (iii) below) between \cite{CK1} and the present paper,
but they are equivalent under a time-reversal.
\end{remark}

\begin{definition}\label{PER}
\rm \begin{description}
\item{(i)} We say that the \emph{parabolic Harnack inequality} $\PHI^+(\phi)$ holds for
Markov the process $X$
if Definition \ref{theorem:HI-PHI}
holds for some constants $C_1>0$, $C_k =kC_1$ for $k=2, 3, 4$, $0<C_5<1$ and  $C_6>0$.

\item{(ii)} We say that the  \emph{elliptic Harnack inequality} ($\EHI$) holds for the
Markov process $X$ if there exist constants $c>0$ and $\delta\in(0,1)$ such that
for every $x_0 \in M $, $r>0$ and for
 every non-negative function $u$ on $M$ that is harmonic in $B(x_0,r)$,
$$
\esssup_{ B(x_0,\delta r)} u\le c\, \essinf_{ B(x_0,\delta r)} u.
$$
\item{(iii)} We say that the
\emph{parabolic H\"older regularity}
$\PHR(\phi)$
holds for the Markov process $X$ if there exist  constants $c>0$,
$\theta\in (0, 1]$ and $\eps \in (0, 1)$ such that for every $x_0 \in M $, $t_0\ge0$, $r>0$ and for
 every bounded measurable function $u=u(t,x)$ that is  caloric in
$Q(t_0,x_0,\phi(r), r)$,    there is a properly exceptional set ${\cal N}_u\supset {\cal N}$ so that
\be\label{e:phr}
|u(s,x)-u(t, y)|\le c\left( \frac{\phi^{-1}(|s-t|)+d(x, y)}{r} \right)^\theta  \esssup_{ [t_0, t_0+\phi (r)] \times M}|u|
\ee
 for  every
 $s, t\in (t_0+\phi(r)-\phi(\eps r), t_0+\phi (r))$ and $x, y \in B(x_0, \eps r)\setminus {\cal N}_u$.

\item{(vi)} We say that the \emph{elliptic H\"older regularity}
($\EHR$)
holds for the process $X$, if there exist constants $c>0$,
$\theta\in (0, 1]$ and $\eps \in (0, 1)$ such that  for every $x_0 \in M $, $r>0$ and for
 every bounded measurable function $u$ on $M$ that is harmonic in $B(x_0, r)$, there is a properly exceptional set ${\cal N}_u\supset {\cal N}$
so that
\be\label{e:ehr}
|u(x)-u(y)| \leq c   \left( \frac{d(x, y)}{r}\right)^\theta\esssup_{M}  |u| \ee
for any $x,$ $y \in B(x_0, \eps r)\setminus {\cal N}_u.$

\end{description}
\end{definition}

Clearly  $\PHI^+(\phi) \Longrightarrow \PHI (\phi)  \Longrightarrow \EHI$ and $\PHR(\phi) \Longrightarrow \EHR$.

\begin{remark}\label{R:phrehr}
\rm
\begin{itemize}
\item[(i)]
$\PHI(\phi)$ in Definition \ref{theorem:HI-PHI}
is called a weak parabolic Harnack inequality in \cite{BGK2},
in the sense that \eqref{e:phidef} holds for some $C_1, \cdots, C_5$.
It is called a parabolic Harnack inequality in \cite{BGK2}
if \eqref{e:phidef} holds for any choice of positive constants $C_4>C_3>C_2>C_1>0$, $0<C_5<1$
with $C_6=C_6(C_1, \dots, C_5)<\infty$. Since our underlying metric measure space may not be geodesic,
one can not expect to deduce  parabolic Harnack inequality from weak  parabolic Harnack inequality.
See \cite{BGK2} for related discussion on diffusions.

\item[(ii)]
We will show
in Proposition \ref{phi-phi} that under $\VD$, $\RVD$ and \eqref{polycon},
$\PHI^+ (\phi)$ and $\PHI (\phi)$ are equivalent.

\item[(iii)] Clearly, $\PHI (\phi)$ holds if and only if the desired property holds for every bounded
 caloric function on cylinder $Q(t_0, x_0, C_4\phi (R), R)$.
Same  for $\PHI ^+ (\phi)$ and $\EHI$.

\item[(iv)] Note that in the definition of $\PHR(\phi)$ (resp. $\EHR$) if the inequality \eqref{e:phr} (resp. \eqref{e:ehr}) holds for some $\varepsilon\in(0,1)$, then it holds for all $\eps \in(0,1)$ (with possibly different constant $c$). We take $\EHR$ for example. For every $x_0 \in M $ and $r>0$, let $u$ be a bounded function on $M$ such that it is harmonic in $B(x_0, r)$. Then, for any $\eps' \in (0, 1)$
    and $x\in B(x_0,\eps'r)\setminus {\cal N}_u$, $u$ is harmonic on $B(x,(1-\eps')r)$. Applying \eqref{e:ehr} for $u$ on
    $B(x,(1-\eps')r))$, we find that for any $y\in B(x_0,\eps'r)\setminus {\cal N}_u$ with $d(x,y)\le (1-\eps')\eps r$, $$|u(x)-u(y)|\le c \left( \frac{d(x, y)}{r}\right)^\theta\esssup_{z\in M}  |u(z)|.$$ This implies that for any $x,y\in B(x_0,\eps'r)\setminus {\cal N}_u$, \eqref{e:ehr} holds with $c'=c\vee \frac{2}{[(1-\eps')\eps]^\theta}.$

\end{itemize}
\end{remark}

\medskip

Below we discuss  stability of parabolic Harnack inequalities. This requires further
definitions.

\begin{definition}\label{ndl1}
\rm We say that \emph{a
near diagonal lower bounded estimate for Dirichlet heat kernel} $\NDL(\phi)$ holds, i.e. there exist $\eps\in (0,1)$ and $c_1>0$ such that for any $x_0\in M$, $r>0$,
$0<t\le \phi(\eps r)$ and $B=B(x_0,r)$,
\be\label{NDLdf1} p^{B}(t, x
,y )\ge \frac{c_1}{V(x_0, \phi^{-1}(t))},\quad x ,y\in B(x_0,
\eps\phi^{-1}(t)) \cap M_0.
\ee
\end{definition}
Under $\VD$, we may replace $V(x_0, \phi^{-1}(t))$ in the definition
by either $V(x, \phi^{-1}(t))$ or $V(y, \phi^{-1}(t))$.
Under \eqref{polycon}, we also may replace $ \phi(\eps r)$ and
$\eps\phi^{-1}(t)$ in the definition above by $\eps\phi(r)$ and
$\phi^{-1}(\eps t)$, respectively.

\medskip

The following inequality was introduced in \cite{BBK2} in the setting of graphs.
See \cite{CKK1} for the general setting of metric measure spaces.

\begin{definition}\label{thm:defUJS} {\rm
We say that $\UJS$ holds if
there is a symmetric function $J(x, y)$ so that
$J(x, dy)=J(x, y)\,\mu (dy)$, and
there is a constant $c>0$ such that for $\mu$-a.e. $x, y\in M$ with $x\not= y$,
\begin{equation}\label{ujs}
J(x,y)\le  \frac{c}{V(x,r)}\int_{B(x,r)}J(z,y)\,\mu(dz)
\quad\hbox{for every }
0<r\le   d(x,y) /2.
\end{equation}
} \end{definition}

Note that $\UJS$ is implied by the following
pointwise comparability condition of the
jump kernel $J(x,y)$: there is a constant $c>0$ such that
$J(x,y)\le cJ(z,y)$ for
$\mu$-a.e. $x, y,z\in M$ with $x\not= y$ and $0<d(x,z)\le   d(x,y) /2$.
 Some sufficient conditions for  $\UJS$ can be found in
\cite[Lemma 2.1 and Example 2.2]{CKK2}.

\begin{definition} {\rm
We say that
the {\em $($weak$)$ Poincar\'e inequality}
$\PI(\phi)$
holds if there exist constants $C>0 $ and $\kappa\ge1$ such that
for any  ball $B_r=B(x,r)$ with $x\in M$  and for any $f \in \sF_b$,
\begin{equation}\label{eq:PIn}
\int_{B_r} (f-\ol f_{B_r})^2\, d\mu \le C \phi(r)\int_{{B_{\kappa r}}\times {B_{\kappa r}}} (f(y)-f(x))^2\,J(dx,dy),
\end{equation}
where $\ol f_{B_r}= \frac{1}{\mu({B_r})}\int_{B_r} f\,d\mu$ is the average value of $f$ on ${B_r}$.} \end{definition}
 If the integral on the right hand side of \eqref{eq:PIn} is over ${B_{r}}\times {B_r}$ (i.e. $\kappa=1$), then it is called strong Poincar\'e inequality.
If the metric is geodesic, it is known that
(weak) Poincar\'e inequality implies
strong Poincar\'e inequality (see for instance \cite[Section 5.3]{Sa}), but in general they are not the same. In this paper, we only use weak Poincar\'e inequality.
Note also that the left hand side of \eqref{eq:PIn} is
equal to $\inf_{a \in \bR} \int_{B_r} (f-a)^2\, d\mu$.

The following is the main result of this paper.

\begin{theorem}\label{T:PHI}
Suppose that the metric measure space  $(M, d,  \mu)$ satisfies $\VD$ and $\RVD$, and $\phi$ satisfies \eqref{polycon}.
Then the following are equivalent:\\
$(1)$ $\PHI(\phi)$. \\
$(2)$ $\PHI^+ (\phi)$. \\
$(3)$ $\UHK(\phi)$, $\NDL(\phi)$ and $\UJS$.\\
$(4)$ $\NDL(\phi)$ and $\UJS$.\\
$(5)$ {\rm PHR}$(\phi)$, $\E_\phi$ and $\UJS$. \\
$(6)$ {\rm EHR}, $\E_\phi$ and $\UJS$. \\
$(7)$ $\PI(\phi)$, $\J_{\phi,\le}$, $\CSJ(\phi)$ and $\UJS$.
\end{theorem}

We note that any of  the conditions above implies the conservativeness of the process $\{X_t\}$;
see Proposition  \ref{P:3.1} and \cite[Lemma 4.22]{CKW1},
 Proposition \ref{ndlb} and Proposition
\ref{phi-ephi-2}.

As a corollary of Theorem \ref{T:main} and Theorem
\ref{T:PHI} (noting that $\J_\phi$ implies $\UJS$), we have
the following.

\begin{cor}\label{C:1.25} Suppose that the metric measure space  $(M, d,  \mu)$ satisfies $\VD$ and $\RVD$, and $\phi$ satisfies \eqref{polycon}. Then
$$ \HK (\phi)\Longleftrightarrow\PHI(\phi)+ \J_{\phi,\ge}.$$
\end{cor}

\begin{remark}\label{rem:boun}\rm
In this paper, the metric measure space $(M,d,\mu)$ is assumed to be unbounded.
This condition can be relaxed. In fact, if all the corresponding conditions on $(M, d, \mu)$
are imposed only for a finite range of radius (that is, assumed to hold for all $r\in(0,\bar R)$ for some
 $\bar R\in (0, {\rm diam}\, M]$), then with a minor adjustment of the proofs,
 all the results of this paper continue to hold but
with a localized version, for instance,  with the statement of $\PHI(\phi)$ changed to hold for all $r\in (0,\bar R)$, and  those
of $\UHK(\phi)$
and $\HK(\phi)$ changed to hold for $t\in (0,\phi(\bar R))$ and all $x, y\in M$.
In particular, all results of this paper hold on bounded metric measure spaces with the aforementioned modification.
We plan to spell out the details in a future publication.
We note that for the heat kernel estimates for stable-like with Ahlfors $d$-set condition, \cite{GHH} considers both bounded and unbounded cases.
\end{remark}

In this paper, we concentrate on the stability of parabolic Harnack
inequalities.
Stability of elliptic Harnack  inequalities  and the connection to
the H\"older regularity
of harmonic functions for symmetric non-local Dirichlet forms are studied in a
separate paper \cite{CKW3}.

\medskip

In addition to the papers mentioned above,
for other related work on Harnack inequalities and H\"older regularities
for harmonic functions of non-local operators, we mention \cite{BL1, ChZ, LS, Kom, MK, SU, SV} and the
references therein. We emphasize this is only a partial list of the vast literature
on the subject.

\medskip

The rest of the paper is organized as follows.
  The proof of Theorem \ref{T:PHI} is given in Section \ref{pf of 1-21}.
  In  Section \ref{section2}, we present some
 preliminary results.  Various consequences of parabolic Harnack inequalities are given in Section \ref{sectHarn}.
The proof of $(1) \Longleftrightarrow (2)\Longleftrightarrow (3)\Longleftrightarrow (4)$ is given in Subsection \ref{suffphi71},
the proof of $(1) \Longleftrightarrow (5)\Longleftrightarrow (6)$ is given in Subsection \ref{sec7-2}, while
$(1)\Longleftrightarrow (7)$ is shown in Subsection \ref{sec7-3}.
Figure \ref{diagfig} illustrates
implications of various conditions and flow of our proofs.

\begin{figure}[t]
~~~~~~~~~~~\centerline{\epsfig{file=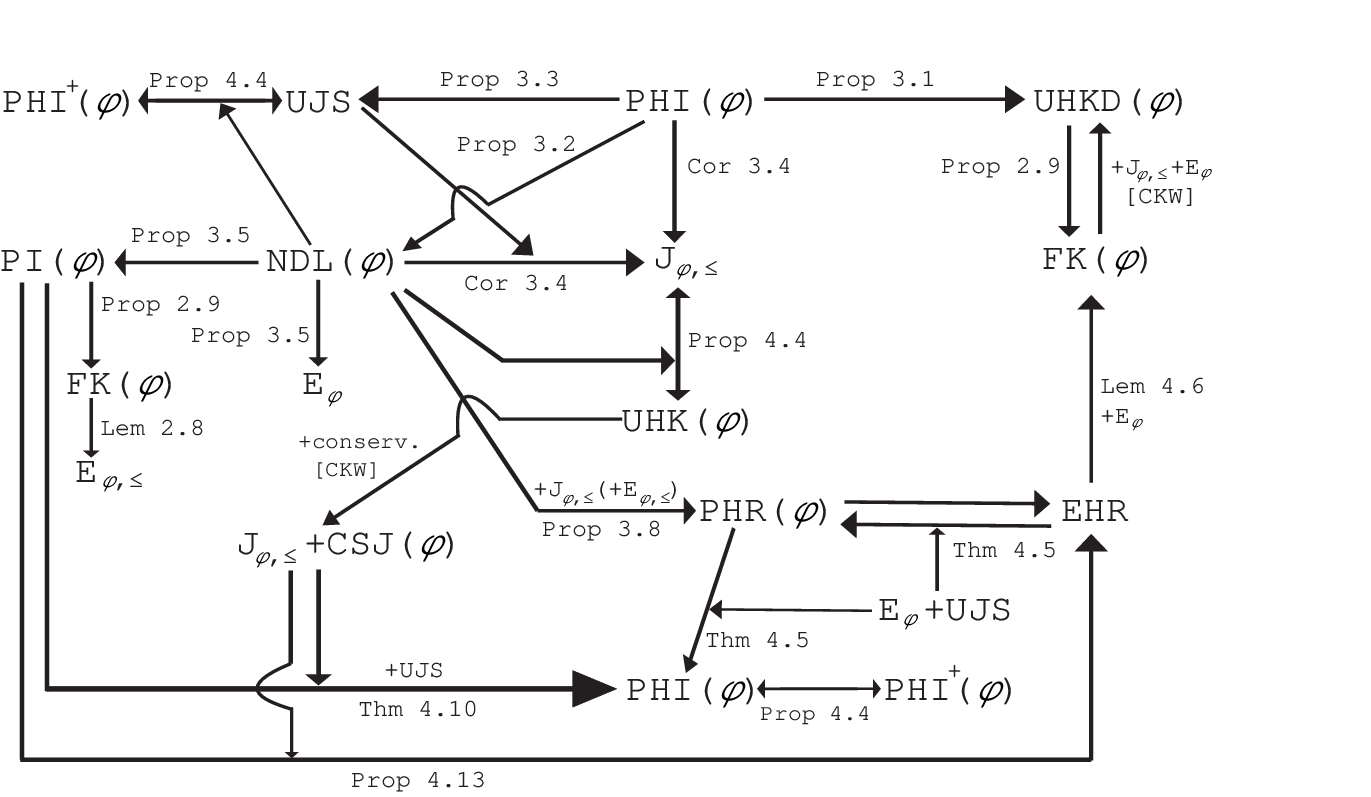, height=3.7in}}
\caption{diagram}\label{diagfig}
\end{figure}

\bigskip

Throughout this paper, we will use $c$, with or without subscripts,
to denote strictly positive finite constants whose values are
insignificant and may change from line to line. For functions $f$ and $g$ defined on a set $D$,
we write $f\asymp g$ if there exists a constant $c\geq 1$ such that $c^{-1}f(x)\le g(x)\le
c\,f(x)$ for all $x\in D$. For $p\in [1,
\infty]$, we will use $\| f\|_p$ to denote the $L^p$-norm in $L^p(M;\mu)$. For any $D\subset M$, denote by $C(D)$ (resp. $C_c(D)$) the set of continuous functions (resp. continuous functions with compact support) on $D$.
In this paper, we omit some of the proofs that are similar to those in literature.

\section{Preliminaries}\label{section2}

In this section we present some preliminary results that will be used in the sequel.

We first recall the analytic characterization of  harmonic and subharmonic functions.
Let $D$ be an open subset of $M$.
Recall that a function $f$ is said to be locally in $\sF_{D}$, denoted as $f\in \sF_{D}^{loc}$, if for every relatively compact subset $U$ of ${D}$, there is a function $g\in \sF_{D}$ such that $f=g$ $m$-a.e. on $U$.
The following is established in    \cite{Chen}.
\begin{lemma}{\rm (\cite[Lemma 2.6]{Chen})}
Let ${D}$ be an open subset of $M$. Suppose $u$ is a function in $\sF_{D}^{loc}$ that is locally bounded on ${D}$
and satisfies that
\begin{equation}\label{con-1}
\int_{U\times V^c} |u(y)|\,J(dx,dy)<\infty
\end{equation} for any relatively compact open sets $U$ and $V$ of $M$ with $\bar{U}\subset V \subset \bar{V} \subset {D}$.
 Then  for every $v\in C_c({D})\cap \sF$, the expression
$$
\int (u(x)-u(y))(v(x)-v(y))\,J(dx,dy)
$$ is well defined and finite; it will still be denoted as $\sE(u,v)$.
\end{lemma}
As noted in \cite[(2.3)]{Chen}, since $(\sE,\sF)$ is a regular Dirichlet form on $L^2(M; \mu)$, for any relatively compact open sets $U$ and $V$ with $\bar{U}\subset V$, there is a function $\psi\in \sF\cap C_c(M)$ such that $\psi=1$ on $U$ and $\psi=0$ on $V^c$. Consequently,
$$\int_{U\times V^c}\,J(dx,dy)=\int_{U\times V^c}(\psi(x)-\psi(y))^2\,J(dx,dy)\le \sE(\psi,\psi)<\infty,$$
so each bounded function $u$ satisfies \eqref{con-1}.

We say that a nearly Borel measurable function $u$ on $M$
is \emph{$\sE$-subharmonic}  (resp. \emph{$\sE$-harmonic, $\sE$-superharmonic})
in ${D}$ if $u\in\sF_{D}^{loc}$ that is locally bounded on $D$, satisfies
\eqref{con-1}
for any relatively compact open sets $U$ and $V$ of $M$ with $\bar{U}\subset V \subset \bar{V} \subset {D}$, and that
\begin{equation*}\label{an-har}
\sE(u,\varphi)\le 0 \quad (\textrm{resp.}\ =0, \ge0)
\quad \hbox{for any } 0\le\varphi\in\sF\cap C_c(D).
\end{equation*}

The following is established in \cite[Theorem 2.11  and Lemma 2.3]{Chen} first for harmonic functions,
and then extended in \cite[Theorem 2.9]{ChK} to subharmonic functions.

\begin{theorem}\label{equ-har} Let ${D}$ be an open subset of $M$, and   $u$ be  a bounded function.
Then $u$ is $\sE$-harmonic $($resp.  $\sE$-subharmonic$)$ in ${D}$ if and only if $u$ is  harmonic
 $($resp. subharmonic$)$
 in ${D}$.
 \end{theorem}

We next recall four results from \cite{CKW1}.  Lemma \ref{intelem} is essentially given in \cite[Lemma 2.1]{CK2}.

\begin{lemma} {\rm (\cite[Lemma 2.1]{CKW1})}   \label{intelem}
Assume that $\VD$,
 \eqref{polycon}
and $\J_{\phi,\le}$ hold. Then there exists a constant $c_1>0$ such that
\[
 \int_{B(x,r)^c}
J(x,y)\,\mu (dy)\le \frac{c_1}{\phi (r)}
\quad\mbox{for every }
x\in M \hbox{ and } r>0.
\]
\end{lemma}

\begin{proposition} {\rm (\cite[Proposition 3.1(ii)]{CKW1})} \label{P:3.1}
Suppose that $\VD$ holds.
Then either $\LHK(\phi)$ or  $\NDL(\phi)$ implies  $\zeta
=\infty$ a.s., where $\zeta$ denotes the lifetime of the process $X$.
\end{proposition}

For a Borel measurable function $u$ on $M$, following
\cite{CKP1}, we define its
\emph{nonlocal tail} $\T(u; x_0,r)$ in the ball
$B(x_0,r)$ by
\begin{equation}\label{def-T}
\T\, (u; x_0,r):=\phi(r)\int_{B(x_0,r)^c}\frac{|u(z)|}{V(x_0,d(x_0,z))\phi(d(x_0,z))}\,\mu(dz).
\end{equation}
In the following, for any $x\in M$ and $r>0$, set
$B_r (x) = B(x,r)$.

\begin{lemma} {\rm (\cite[Lemma 4.8]{CKW1})} \label{oppo}
Suppose $\VD$, \eqref{polycon}, $\FK(\phi)$,
$\CSJ(\phi)$ and $\J_{\phi,\le}$ hold. Let $x_0 \in M$, $R, {r_1},{r_2}>0$  with
 ${r_1}\in [R/2 ,R]$ and ${r_1}+{r_2}\le R$, and $u$ be an $\sE$-subharmonic
function in $B_R(x_0)$.
For $\theta>0$, set $v :=
(u-\theta)_+$.
We have
\begin{equation*}\begin{split}
 \int_{B_{{r_1}}(x_0)}v^2\,d\mu\le&
\frac{ c_1 }{ \theta^{2\nu} V(x_0,R)^\nu} \left(\int_{B_{{r_1}+{r_2}}(x_0)}u^2\,d\mu\right)^{1+\nu}\\
&\times \left(1+\frac{{r_1}}{{r_2}}\right)^{\beta_2} \left[ 1+
\left(1+\frac{{r_1}}{{r_2}}\right)^{d_2+\beta_2-\beta_1}
\frac{\T\,(u; x_0,R/2)}{\theta} \right],\end{split}
\end{equation*}
where $\nu$ is the constant in
$\FK(\phi)$, $d_2$ is the constant in
\eqref{e:vd2},
$\beta_1, \beta_2$ are the constants in \eqref{polycon}, and $c_1$ is a constant independent of $\theta, x_0, R, {r_1}$
and ${r_2}$.
\end{lemma}

\begin{proposition} {\rm (\cite[Proposition 4.10]{CKW1})} \label{P:mvi2g}
{\bf($L^2$-mean value inequality)}\,  Assume $\VD$, \eqref{polycon}, $\FK(\phi)$, $\CSJ(\phi)$ and
$\J_{\phi,\le}$ hold. For any $x_0\in M$ and $r>0$,
let $u$ be a bounded
$\sE$-subharmonic
in $B_r(x_0)$.  Then
there is a constant $C_0>0$ independent of
$x_0$ and $r$ so that
\be \label{e:mvi2g-1}
 \esssup_{B_{r/2}(x_0)}u\le  C_0\left(  \left(\frac{1}{V(x_0,r)}\int_{B_r(x_0)} u^2\,d\mu \right)^{1/2}+ \T\, (u; x_0,r/2)\right).
\ee
\end{proposition}

The  following three results are proved in \cite{CKW1}.

\begin{proposition} {\rm (\cite[Proposition 4.14]{CKW1})} \label{P:exit}
Assume $\VD$, \eqref{polycon}, $\FK(\phi)$, $\J_{\phi,\le}$ and $\CSJ(\phi)$ hold. Then, $\E_{\phi}$ holds.
\end{proposition}

\begin{lemma} {\rm (\cite[Lemma 4.15]{CKW1})}  \label{upper-e}
Assume that $\VD$, \eqref{polycon} and $\FK(\phi)$ hold. Then, $\E_{\phi,\le} $ holds.\end{lemma}

\begin{proposition} {\rm (\cite[Proposition 7.6]{CKW1})} \label{pi-e-pre}
Assume that $\VD$, $\RVD$ and \eqref{polycon} are satisfied.
Then either $\PI(\phi)$ or $\UHKD(\phi)$ implies $\FK(\phi)$.
\end{proposition}

We also record the following
elementary iteration lemma,
 see, e.g., \cite[Lemma 7.1]{G} or \cite[Lemma 4.9]{CKW1}.

\begin{lemma}\label{L:it} Let $\beta>0$ and let $\{A_j\}$ be a sequence of real positive numbers such that
$A_{j+1}\le c_0b^jA_j^{1+\beta}$ for every $j\geq 0$ with $c_0>0$ and $b>1$. If
$A_0\le c_0^{-1/\beta}b^{-1/\beta^2}$,
 then we have $ A_j\le b^{-j/\beta}A_0$ for $j\geq 1$,
 which in particular yields $\lim_{j\to\infty}{ A_j}=0.$
\end{lemma}

The following formula,
often called the L\'{e}vy system formula,
will be used many times in this paper. See, for example \cite[Appendix A]{CK2}
for a proof.

\begin{lemma}\label{Levy-sys}
Let $f$ be a non-negative measurable function on ${\mathbb R}_+
\times M \times M$ that vanishes along the diagonal. Then for every
$t\geq 0 $, $x\in M_0$ and  stopping time $T$ $($with
respect to the filtration of $\{X_t\}$$)$,
$$
{\mathbb E}^x \left[\sum_{s\le T} f(s,X_{s-}, X_s) \right]={\mathbb
E}^x \left[ \int_0^T  \int_M f(s,X_s, y)\, J(X_s,dy)
\,ds \right].
 $$
\end{lemma}

\section{Consequences of Harnack inequalities}\label{sectHarn}

\subsection{Consequences of $\PHI(\phi)$}\label{phi-section-1}

 In this subsection (together with some of the results from next subsection), we prove that
 $\PHI(\phi)$ implies $\UHK(\phi)$, $\NDL(\phi)$ and $\UJS$. Without further mention, throughout the proof we will assume that $\mu$ and $\phi$ satisfy $\VD$ and \eqref{polycon}, respectively.
Noting that $V(y, r)>0$ for every $y\in M$ and $r>0$ (since $\mu$ has full support),
we have from \eqref{e:vd2} that for all $x,y\in M$ and $0<r\le R$,
\be\label{eq:vdeno}
\frac{V(x,R)}{V(y,r)}\le
\frac{V(y, d(x, y)+R)}{V(y,r)} \leq {\wt C_\mu}\Big(\frac {d(x,y)+R}r\Big)^{d_2}.
\ee

 \begin{proposition}\label{uodest}  Under $\VD$ and \eqref{polycon}, $\PHI(\phi)$ implies $\UHKD(\phi)$.
 \end{proposition}

 \begin{proof}
 Let $C_i$ $(i=1,\ldots, 6)$ be the constants taken from the definition of $\PHI(\phi)$.
 For any $x_0\in M$, $r>0$, $t= C_4\phi(r)$ and any $0\le f\in L^2(M; \mu)\cap L^1(M;\mu)$,
 applying $\PHI(\phi)$ to the caloric function $v(s,x):=P_s f(x)$ in $Q(0,x_0,t,r)$, we have for $x,y\in B(x_0,C_5r)\setminus {\cal N}_v$,
 \begin{equation*}\label{uodest-1-1}
 P_{(C_1+C_2)\phi(r)/2}f(x)\le  C_6 P_{(C_3+C_4)\phi(r)/2}f(y),
 \end{equation*}
 where ${\cal N}_v$ is the properly exceptional set associated with $v$.  Then,
 $$V(x_0,C_5r)P_{(C_1+C_2)\phi(r)/2}f(x)\le
 C_6\int_{B(x_0,C_5r)}P_{(C_3+C_4)\phi(r)/2}f(y)\,\mu(dy)\le C_6\int f(y)\,\mu(dy).$$ Therefore, there is a constant $c_1>0$ such that for almost all $x\in M$ and $t>0$,
 \begin{equation}\label{uodest-1} P_{t}f(x)\le \frac{c_1}{V\left(x,{\phi^{-1}(t)}\right)}
 \|f\|_{1},\end{equation} where we have used $\VD$ and
 \eqref{polycon} in the inequality above. In particular, the semigroup $\{P_t\}$ is locally ultracontractive.
According to \cite[Proposition 7.7]{CKW1}
 (see also \cite[Theorem 3.1]{BBCK} and \cite[Theorem 2.12]{GT}), there exists a properly exceptional set $\mathcal{N}\subset M$ such that, the semigroup $\{P_t\}$ possesses the heat kernel $p(t,x,y)$ with domain $(0,\infty)\times (M\setminus \mathcal{N} )\times (M\setminus \mathcal{N})$.

By \eqref{uodest-1} again, for almost all $x$, $y\in M$,
$$p(t,x,y)\le\frac{c_1}{V\left(x,{\phi^{-1}(t)}\right)}.$$ In the following, for any $x\in M$ and $t>0$, define $$ \varphi(x,t)=\inf_{0<r\le \phi^{-1}(t)}\frac{1}{\mu(B(x,r))}\int_{B(x,r)}\frac{1}{V\left(z,{\phi^{-1}(t)}\right)}\,\mu(dz).$$
On the one hand, by \eqref{eq:vdeno} from $\VD$, there is a constant $c_2>1$ such that for all $x\in M$ and $t>0$,
$$\frac{1}{c_2V\left(x,{\phi^{-1}(t)}\right)}\le  \varphi(x,t)\le  \frac{c_2}{V\left(x,{\phi^{-1}(t)}\right)}.$$On the other hand,
for any $t>0$, $x\mapsto\varphi(x,t)$ is an upper semi-continuous
function on $M$. Indeed, for any $x\in M$,
\begin{align*}\limsup_{y\to x}\varphi(y,t)=&\lim_{s\to0}\sup_{0<d(y,x)\le s}\inf_{0<r\le \phi^{-1}(t)}\frac{1}{\mu(B(y,r))}\int_{B(y,r)}\frac{1}{V\left(z,{\phi^{-1}(t)}\right)}\,\mu(dz)\\
\le&\inf_{0<r\le \phi^{-1}(t)}\lim_{s\to0}\sup_{0<d(y,x)\le s}\frac{1}{\mu(B(y,r))}\int_{B(y,r)}\frac{1}{V\left(z,{\phi^{-1}(t)}\right)}\,\mu(dz)\\
=&\inf_{0<r\le \phi^{-1}(t)}\frac{1}{\mu(B(x,r))}\int_{B(x,r)}\frac{1}{V\left(z,{\phi^{-1}(t)}\right)}\,\mu(dz)\\
=&\varphi(x,t).\end{align*} Combining all the conclusions above with
\cite[Proposition 7.7]{CKW1}
again, we have
$$
p(t,x,y)\le \frac{c_3}{V\left(x,{\phi^{-1}(t)}\right)}
\quad \hbox{for all } (x,y)\in (M\setminus \mathcal{N} )\times (M\setminus \mathcal{N}).
$$
This proves $\UHKD(\phi)$.
\qed
 \end{proof}

 A key consequence of $\PHI(\phi)$ is a near-diagonal lower bound estimate for $p^{D}(t,x,y)$.
For the cases of diffusions, similar fact was proved in \cite[Section 4.3.4]{BGK2},
but there is a gap in the middle of Page 1129. (Indeed, the proof uses
$B(x_0, R+\rho)=\cup_{x\in B(x_0,R)}B(x,\rho)$, which is not true in general
unless the metric is geodesic.) Our proof below fixes the issue (see step (ii) in the proof)
and proves $\NDL(\phi)$ in the framework of general metric spaces.

 \begin{proposition}\label{ndlb}
 Assume $\VD$, \eqref{polycon} and  $\PHI(\phi)$ hold. Then   $\NDL(\phi)$   holds.
 Consequently, $X=\{X_t\}$ is conservative.
 \end{proposition}

 \begin{proof}
 Note that by $\VD$ and
  Proposition \ref{P:3.1},
 $\NDL(\phi)$ implies the conservativeness of the process $X$. We only need to verify that $\NDL(\phi)$ holds.
Below we will prove $\NDL(\phi)$ with $ \phi(\eps r)$ and $\eps\phi^{-1}(t)$ replaced by  $\eps\phi(r)$ and $\phi^{-1}(\eps t)$ in the definition.

 (i) For any open ball $B:=B(x_0,r)$ with $x_0\in M_0$ and $r>0$, it
 follows from \eqref{uodest-1} and $\VD$ that for any $t>0$
 $$\|P^B_tf\|_\infty\le \frac{c_1}{V\left(x_0,{\phi^{-1}(t)}\right)}
 \|f\|_{1}.$$ Then, by \cite[Theorem 3.1]{BBCK}, the Dirichlet
 semigroup $\{P^B_t\}$ has the heat kernel $p^B(t,x,y)$ defined on
 $(0,\infty)\times (B\setminus \sN_1)\times (B\setminus \sN_1)$ such
 that
 $$p^B(t,x,y)\le \frac{c_1}{V(x_0, \phi^{-1}(t))},\quad x,y\in B \setminus \sN_1,
 $$
 where $\sN_1\subset B$ is a properly exceptional set of
 the killing process $\{X^B_t\}$
 such that $\sN_1\supset \sN\cap B$;
 moreover, there is an $\sE^B$-nest
 $\{F_k\}$ consisting of an increasing sequence of compact sets of $B$ so
 that $\sN_1 =B \setminus\cup_{k=1}^\infty F_k$ and that for every
 $t>0$, $y\in B\setminus \cal{N}$ and $k\ge1$, $x\mapsto p^B(t,x,y)$ is
 continuous on each $F_k$ (i.e. for every $t>0$ and $y\in B\setminus
 \sN_1$, the function $x\mapsto p^B(t,x,y)$ is quasi-continuous on
 $B$).

(ii) Choose an $\wh x_0 \in B(x_0, C_5r)\setminus \sN_1$, where
$C_5\in (0, 1)$ is the constant in $\PHI (\phi)$. Define
$$
\wh B=\{y\in B\setminus \sN_1:
p^B(t, \wh x_0, y)>0 \hbox{ for some } t>0 \}.
$$
We will show that for every $x, y\in \wh B$, there is some $t>0$ so that $p^B(t, x, y)>0$,
and that
\begin{equation}\label{e:3.3}
p^B(t, x, y)=0 \quad \hbox{on } (0, \infty) \times \wh B
\times (B\setminus (\wh B\cup \sN_1)).
\end{equation}
To prove these, first noting that
since $\bP^x (\lim_{t\downarrow 0} X^B_t=X^B_0=x)=1$ implies
$\bP^x (\tau_B>0)=1$, we must have
$p^B(t, \wh x_0, \wh x_0)=\int_B p^B(t/2, \wh x_0, y)^2 \,\mu (dy)>0$ for some $t>0$.
Thus $\wh x_0\in \wh B$. By $\PHI(\phi)$ applied to the caloric function
$(s, y)\mapsto p^B(s, y, \wh x_0)=p^B(s, \wh x_0, y)$, we see that if $x\in \wh B$,
then there are constants $r_x>0$ and $s_x>0$
so that
\begin{equation}\label{e:3.2}
p^B(s, \wh x_0, z)>0 \quad\hbox{ for every } z\in B(x, r_x)\setminus \sN_1 \hbox{ and }
s\geq s_x.
\end{equation}
Hence, there is an open subset $U$ of $B$ containing $\wh x_0$ so that $\wh B=U\setminus \sN_1$. Similarly,
for every $x, y\in \wh B$, by $\PHI (\phi)$, there are constants $r_0>0$ and $s_0>0$
so that
$$
p^B(s, x, z)>0 \quad \hbox{and} \quad p^B(s, y, z)>0 \quad \hbox{for every } z\in B(\wh x_0, r_0)\setminus \sN_1
\hbox{ and } s\geq s_0.
$$
In particular, it follows that for every $s, t\geq s_0$,
\begin{equation}\label{efef}
p^B(t+s, x, y)\geq \int_{B(\wh x_0, r_0)} p^B(s, x, z)p^B(t, z, y) \,\mu (dz) >0.
\end{equation}
For $x\in \wh B$, define
$$
\wh B_x = \{y\in B\setminus \sN_1: p^B(t, x, y)>0 \hbox{ for some } t>0 \}.
$$
Then $\wh B\subset \wh B_x$. We claim  $\wh B = \wh B_x$.
Were $\wh B \varsubsetneq \wh B_x$, take
$y\in \wh B_x \setminus \wh B$. By $\PHI (\phi)$ applied to the caloric function
$(s, z)\mapsto p^B(s, z, y)=p^B(s, y, z)$, there are constants $r_x>0$ and $s_x>0$
so that $p^B(s, y, z)>0$ for every $z\in B(x, r_x)\setminus \sN_1$ and $s\geq s_x$,
and \eqref{e:3.2} holds. Hence, for every $t, s \geq s_x$, we have
$$
p^B(t+s, \wh x_0, y)\geq \int_{B(x, r_x)} p^B(t, \wh x_0, z) p^B (s, z, y)\, \mu (dz)>0,
$$
which implies that $y\in \wh B$. This contradiction shows that $\wh B_x =\wh B$ for every
$x\in \wh B$.
We have thus established that for every $x, y\in \wh B$, there is some $t>0$ so that
$p^B(t, x, y)>0$, and that \eqref{e:3.3} holds.
Consequently, for every $t>0$ and $x, y\in \wh B=U\setminus \sN_1$,
\begin{equation}\label{e:3.4}
p^U(t, x, y)= p^B(t, x, y)- \bE_x \left[p^B(t-\tau_U, X^B_{\tau_U}, y); t<\tau_U \right]
=p^B(t, x, y)
\end{equation}

Observe that by the symmetry of $p^B(t, x, y)$, \eqref{e:3.3} implies that
$$
\int_{B\setminus U} P^B_t {\bf1}_U (x) \,\mu (dx) =
\int_{U\times (B\setminus U)}
p^B(t, x, y)\,\mu (dx)\,\mu (dy) =0;
$$
in other words, for every $t>0$,
\begin{equation}\label{e:3.5}
P^B_t {\bf1}_U =0  \quad \mu \hbox{-a.e. on }  B\setminus U .
\end{equation}
Let $\lambda_0>0$ be the bottom of the generator  $\sL^U$
associated with $\{P_t^U\}$ and $\psi\geq 0$ the corresponding eigenfunction with $\| \psi\|_{L^2(U; \mu)}=1$.
Note that $\psi =0$ on $B\setminus U$. In view of \eqref{e:3.4} and \eqref{e:3.5}, we have for every $t>0$ and
$x\in B\setminus \sN_1$,
$$
P^B_t \psi (x) = P^U_t \psi (x) = e^{-\lambda_0 t} \psi (x).
$$
Since
$$ e^{-\lambda_0 t} \| \psi \|_{L^\infty (B;\mu)} = \| P^B_t \psi \|_{L^\infty (B;\mu)}
\leq  \mu (B)  \| \psi \|_{L^\infty (B;\mu)} \, \sup_{x,y\in B\setminus \sN_1} p^B(t,x,y),
$$
we have
\begin{equation}\label{e:ndlp-1}
\sup_{x,y\in B\setminus \sN_1 } p^B(t,x,y)\ge
\frac{1}{\mu(B)}e^{-\lambda_0t}.
\end{equation}
We claim that $\psi >0$ on $\wh B$.  Noticing that
\begin{equation}
\label{e:ndlp-2}
v(t,x):=P_t^B\psi(x)=e^{-\lambda_0 t}\psi(x)
\end{equation}
is a caloric function on $(0, \infty) \times B$ and $\psi >0$ has
unit $L^2(B; \mu)$-norm, by $\PHI (\phi)$, there are some $y_0\in \wh B$
and $r_0>0$  so that $B(y_0, r_0)\setminus \sN_1
\subset \wh B$, and  $\psi>0$ on $B(y_0, r_0)$. On the other hand, for every $x\in \wh B$, by \eqref{efef}
(and so $p^B(s, x, y_0)>0$ for some $s>0$) and $\PHI (\phi)$ again,
there are constants $s_0>0$ and  $r_1\in (0, r_0]$ so that
$p^B(t, x, z)>0$ for every $t\geq s_0$ and $z\in B(y_0, r_1)\setminus \sN_1$.
It follows then
$$
\psi (x)= e^{\lambda_0 t} P^B_t \psi (x) \geq e^{\lambda_0 t} \int_{B(y_0, r_1)}
p^B (t, x, z)  \psi (z) \,\mu (dz) >0.
$$
The claim that $\psi >0$ on $\wh B$ is proved. In particular, $\psi (\wh x_0)>0$.

 (iii) Let $C_i$ $(i=1,\ldots, 6)$ be the constants in the definition of $\PHI(\phi)$. Applying $\PHI(\phi)$ to the function $v(t,x)=e^{-\lambda_0 t} \psi (x) $ in the cylinder $Q(0,  x_0, C_4\phi(r),r)$, we get that
 $$
 v(t_-,\wh x_0)\le C_6 v(t_+,\wh x_0),
 $$
 where $t_-=\frac{C_1+C_2}{2}\phi(r)$ and $t_+=\frac{C_3+C_4}{2}\phi(r)$.
It follows from \eqref{e:ndlp-2} that
$$
e^{-\lambda_0 t_-}\psi(\wh x_0)\le C_6 e^{-\lambda_0 t_+}
\psi(\wh x_0).
$$
Since $\psi(\wh x_0)>0$, we arrive at
$$
\lambda_0 \le \frac{\log C_6}{t_+-t_-}\le \frac{1}{\phi(\kappa r)},
$$
where $\kappa>0$ is chosen so that
$$\frac{(C_3+C_4)-(C_1+C_2)}{2}\phi({r}/{2})\ge
\phi(\kappa r)\log C_6
$$for all $r>0$. This along with \eqref{e:ndlp-1} further yields that for
all $t>0$,
$$\esssup_{x,y\in B}p^B(t,x,y)\ge\frac{1}{\mu(B)}e^{-\frac{t}{\phi(\kappa
r)}}.$$

Following the
 arguments between (4.52) and (4.60)
 in \cite[1130--1131]{BGK2} line by line with small modifications,
 we obtain that
 there is a constant $c'>0$ such that for all $x,$ $y\in
 B(x_0, C_5r)\setminus \mathcal{N}_1$ and $t\in (t_0+C_3\phi(r),
 t_0+C_4\phi(r))$ with $t_0=(C_3-C_1)\phi(r)$,
 \begin{equation}\label{e:lowerprof} p^{B}(t, x,y)\ge \frac{c'}{V(x_0,r)}.\end{equation} Note that,
 in order to get \eqref{e:lowerprof} we should change \cite[(4.57)]{BGK2} into $$\esssup_{x\in B'}p^B(s,x,z)\le C_6p^B(t,y,z),\quad y,z\in B':= B(x_0,C_5r)\setminus \mathcal{N}_1.$$
Furthermore, using \eqref{e:lowerprof} instead of
\cite[(4.60)]{BGK2}, one can verify that $\NDL(\phi)$ holds for this
case by the almost same argument between (4.60) and (4.63) in
 \cite[1131--1132]{BGK2}.
 \qed\end{proof}

We next prove that $\PHI(\phi)$ implies $\UJS$.

\begin{proposition}\label{ujs-jle} Under $\VD$ and \eqref{polycon},
$\PHI(\phi)$ implies $\UJS$.
\end{proposition}

\begin{proof}
 (i) Since $(\sE,\sF)$ is a regular Dirichlet form on $L^2(M; \mu)$, for any relatively compact open sets $U$ and $V$ with $\bar{U}\subset V$, there is a function $\psi\in \sF\cap C_c(M)$ such that $\psi=1$ on $U$ and $\psi=0$ on $V^c$. Consequently,
\begin{equation}\label{ooo} \int_{U\times V^c}\,J(dx,dy)=\int_{U\times V^c}(\psi(x)-\psi(y))^2\,J(dx,dy)\le
\sE(\psi,\psi)<\infty.\end{equation}
Since $U$ and $V$ are arbitrarily, we get that for almost all $x\in M$ and each $r>0$,
\begin{equation}\label{faruj-1}J(x, B(x,r)^c)<\infty .\end{equation}

(ii) Let
$D$ be an open set of $M$, and $f(t,z)$ be a bounded and non-negative function on $(0,\infty)\times D^c$.   Then
$$u(t,z):= \begin{cases}
 \bE^z \left[f(t-\tau_D,X_{\tau_D}); \tau_D\le t \right], &  t>0,z\in M_0, \\
0, &  t>0,z\in \mathcal{N}
\end{cases}$$
is non-negative on $(0,\infty)\times M$ and caloric in $(0,\infty)\times D$.
In the proof below, the constants $C_i$ $(i=1, \ldots,6)$ are taken from the definition of $\PHI(\phi)$. For any $x, y\in M_0$ and $0<r \le \frac{1}{2 }d(x,y)$. For any  $0<\varepsilon<r$  and $0<h<(C_1+C_2)\phi(r)/2$, define
$$
f_h(t,z)={\bf 1}_{((C_1+C_2)\phi(r)/2-h,(C_1+C_2)\phi(r)/2)}(t){\bf1}_{ B(y,\varepsilon)}(z), \quad  t>0, z\in M.
$$
For $t\geq (C_1+C_2)\phi(r)/2$, define
\begin{align*}
u_h(t,z)=&\bE^z \left[f_h(t-{\tau_{B(x,r)}},X_{\tau_{B(x,r)}}); {\tau_{B(x,r)}}\le t \right]\\
=&\bP^z\Big(X_{\tau_{B(x,r)}}\in B(y,\varepsilon),
t-(C_1+C_2)\phi(r)/2<\tau_{B(x,r)}< t-(C_1+C_2)\phi(r)/2+h\Big) \end{align*}
if  $ z\in M_0$, and $u_h(t,z)=0$ if $z\in \mathcal{N}.$

According to Lemma \ref{Levy-sys},
for any $z\in B(x,r)\cap M_0$ and $t\geq (C_1+C_2)\phi(r)/2$,
\begin{align*}
u_h(t,z)&=\bE^z\left[ \int_0^{\tau_{B(x,r)}}\,dv\int_{B(y,\varepsilon)}{\bf 1}_{(t-(C_1+C_2)\phi(r)/2, t-(C_1+C_2)\phi(r)/2+h)}(v)\,J(X_v, du)\right]  \\
&=\int_{t-(C_1+C_2)\phi(r)/2}^{t-(C_1+C_2)\phi(r)/2+h} \bE^z \left[ {\bf1}_{(0,\tau_{B(x,r)})}(v) \int_{B(y,\varepsilon)} J(X_v,du)\right] \,dv\\
&= \int_{t-(C_1+C_2)\phi(r)/2}^{t-(C_1+C_2)\phi(r)/2+h} P_v^{B(x,r)}H(z)\,dv,
\end{align*}
where   $H(z):=\int_{B(y,\varepsilon)}\,J(z,du)$.

Applying $\PHI(\phi)$ to $u_h$ in $Q(0,x,C_4\phi(r),r)$, we obtain that for any  $x_0\in B(x,{\eps_1})\setminus (\mathcal{N}_{u_h}\cup \mathcal{N})$ with ${\eps_1}\le C_5r$,
$$u_h((C_1+C_2)\phi(r)/2,x_0)\le C_6u_h((C_3+C_4)\phi(r)/2,x).$$ Now, by the definition of $u_h$ and Proposition \ref{uodest},
\begin{align*}
u_h((C_3+C_4)\phi(r)/2,x)&=\int_{B(x,r)}p^{B(x,r)}\left(\frac{(C_3+C_4)-(C_1+C_2)}{2}\phi(r),x,z\right)\\
&\qquad \qquad \times u_h((C_1+C_2)\phi(r)/2,z)\,\mu(dz)\\
&\le \frac{c_1}{V(x,r)}\int_{B(x,r)} u_h((C_1+C_2)\phi(r)/2,z)\,\mu(dz).
\end{align*}
Combining both inequalities above and integrating by $\frac 1{V(x,{\eps_1})}\int_{B(x,{\eps_1})}\cdots \mu(dx_0)$, we have
\begin{equation}\label{eq:inteobe}\begin{split}
&\frac 1{V(x,{\eps_1})}\int_{B(x,{\eps_1})}u_h((C_1+C_2)\phi(r)/2,x_0)\,\mu(dx_0)\\
&\le \frac{c_2}{V(x,r)}\int_{B(x,r)} u_h((C_1+C_2)\phi(r)/2,z)\,\mu(dz).\end{split}
\end{equation}

According to \eqref{ooo}, $H\in L^1(B(x,r);\mu)$. Then, as $h\to0$,
\begin{equation*}\begin{split}\label{l1}
\Big|\int_{B(x,{\eps_1})}&\Big(\frac 1hu_h((C_1+C_2)\phi(r)/2,z)-H(z)\Big)\,\mu(dz)\Big|\\
\le & \,
\frac 1h\int_0^h\int_{B(x,{\eps_1})}\Big|P_v^{B(x,r)}H(z)-H(z)\Big|\,\mu(dz)\,dv\\
\leq & \,  \frac 1h\int_0^h\|(P_v^{B(x,r)}H-H) \|_{L^1 ({B(x,r)}; \mu)} \,dv\to   \,  0,
\end{split}\end{equation*}
thanks to the continuity of the semigroup $\{P^{B(x,r)}_t\}$ in $L^1({B(x,r)}; \mu)$.
Similarly, we have
\[
\lim_{h\to 0}\Big|\int_{B(x,r)}\Big(\frac 1hu_h((C_1+C_2)\phi(r)/2,z)-H(z)\Big)\mu(dz)\Big|=0.
\]
Thus dividing both sides of \eqref{eq:inteobe} by $h$ and taking $h\to 0$, we have
\begin{equation*}\label{ujs-gene}
\frac 1{V(x,{\eps_1})}\int_{B(x,{\eps_1})}\int_{B(y,\varepsilon)} J(z,du)\,\mu(dz)
\le \frac{c_2}{V(x,r)}\int_{B(x,r)} \int_{B(y,\varepsilon)}J(z,du)\,\mu(dz).
\end{equation*}

Letting $\eps_1 \to0$, by \eqref{ooo}, \eqref{faruj-1} and the Lebesgue differentiation theorem (e.g.\ see \cite[Theorem 1.8]{H}), we find that
for $\mu$-a.e $x\in M$,
\begin{align*}
J(x,B(y,\varepsilon))\le &\frac{c_2}{V(x,r)}\int_{B(x,r)} \int_{B(y,\varepsilon)}J(z,du)\,\mu(dz)=\frac{c_2}{V(x,r)}\int_{B(y,\varepsilon)}\int_{B(x,r)} J(z,du)\,\mu(dz).
\end{align*}
The above inequality implies that $J(x, dy)$ is absolutely continuous with respect to the measure $\mu(dy)$.
So there is a non-negative function $J(x, y)$ so that $J(x, dy)=J(x, y)\, \mu (dy)$.
Since $J(dx, dy)$ is a symmetric measure, we may modify the values of $J(x, y)$ so that it is symmetric in $(x, y)$
for $\mu$-a.e. $x, y\in M$.
Dividing the above by $V(y, \eps)$ and then sending $\eps \to 0$, we have by the Lebesgue differentiation theorem
again that for $\mu$-a.e. $x, y\in M$ and $0<r< \frac{1}{2}d(x, y)$, we have
$$
 J(x,y)\le \frac{c_2}{V(x,r)}\int_{B(x,r)}  J(z, y)\,\mu(dz),
$$
proving $\UJS$. \qed
\end{proof}

\begin{corollary}\label{ujs-ndl}
If  $\VD$, \eqref{polycon}, $\UJS$ and $\NDL(\phi)$ are satisfied, then $\J_{\phi,\le}$ holds. In particular,
  $\J_{\phi,\le}$ holds under
   $\VD$, \eqref{polycon} and $\PHI(\phi)$.
\end{corollary}

\begin{proof}
For any $x\in M_0$ and $r,t>0$,  by Lemma \ref{Levy-sys},
\begin{align*}
1 &\ge\bP^x(X_{\tau_{B(x,r)}}\notin B(x,r), \tau_{B(x, r)}\leq t \hbox{ and }
 \tau_{B(x, r)} \hbox{ is a jumping time})\\
& =\int_0^t\int_{B(x,r)} p^{B(x,r)}(s,x,y) J(y, B(x,r)^c)\,\mu(dy)\,ds.
\end{align*}
By using $\NDL(\phi)$ and taking $t=\phi(\eps r)$ (where $\eps\in(0,1)$ is the constant in the definition of $\NDL(\phi)$), we obtain that for any $x\in M_0$ and $r>0$,
\begin{align*}
 1&\ge\int_{t/2}^t \int_{B(x,\eps \phi^{-1}(t/2))} p^{B(x,r)}(s,x,y) J(y, B(x,r)^c)\,\mu(dy)\,ds\\
 &\ge \frac{t}{2}\essinf_{s\in[t/2,t], y\in B(x,\eps\phi^{-1}(t/2))}p^{B(x,r)}(s,x,y)\int_{B(x,\eps\phi^{-1}(t/2))}J(y, B(x,r)^c)\,\mu(dy)\\
 &\ge \frac{c_1 t}{ V(x,\phi^{-1}(t))} \int_{B(x,\eps\phi^{-1}(t/2))}J(y, B(x,r)^c)\,\mu(dy).
\end{align*} Thus, by $\VD$ and \eqref{polycon},  there are constants $c_2, c_3>1$ such that
\begin{equation}\label{faruj}\int_{B(x,r)} J(y, B(x,c_2r)^c)\,\mu(dy)\le \frac{c_3V(x,r)}{\phi(r)}.\end{equation}

For fixed $x,y\in M$, set $r=\frac{d(x,y)}{1+c_2}\le \frac{d(x,y)}{2}.$ Then, by \eqref{ujs} and \eqref{faruj},
\begin{align*}
J(x,y)\le&  \frac{c_4}{V(x,r)}\int_{B(x,r)}J(z,y)\,\mu(dz)\\
\le&\frac{c_4^2}{V(x,r)V(y,r)}\int_{B(x,r)}\int_{B(y,r)} J(z,u)\,\mu(du)\,\mu(dz)\\
\le&\frac{c_4^2}{V(x,r)V(y,r)}\int_{B(x,r)}\int_{B(x,c_2r)^c} J(z,u)\,\mu(du)\,\mu(dz)\\
\le& \frac{c_5}{V(x,r)V(x,r)}\int_{B(x,r)} J(z,B(x,c_2r)^c)\,\mu(dz)
\le\frac{c_6}{V(x,r)\phi(r)},
\end{align*} which completes the proof, thanks to $\VD$ and \eqref{polycon} again.
\qed\end{proof}

We note that by Proposition \ref{uodest}, Corollary \ref{ujs-ndl},
Proposition \ref{pi-e} in the next subsection and
Theorem \ref{T:main-1}, we have $\PHI(\phi)\Longrightarrow \UHK(\phi)$.

\subsection{Consequences of $\NDL(\phi)$}

In this subsection, we present some consequences of $\NDL(\phi)$.
Since $\PHI(\phi)$ implies $\NDL(\phi)$ by
Proposition \ref{ndlb}, this subsection
can be regarded as a continuation of Subsection \ref{phi-section-1}.

\begin{proposition}
\label{pi-e} Assume that  $\VD$, \eqref{polycon}, and $\NDL(\phi)$ hold.
Then
\begin{itemize}
\item[{\rm (i)}]  $\PI(\phi)$ holds.
If furthermore $\RVD$ is satisfied, then $\FK(\phi)$  also holds.
\item[{\rm (ii)}] $\E_{\phi,\ge}$ holds. If in addition $\RVD$ is satisfied, then we have $\E_{\phi,\leq }$ and so $\E_\phi$.
\end{itemize}
In particular, if $\VD$, $\RVD$, \eqref{polycon} and $\PHI (\phi)$ hold, then so do   {\rm (i)}    and    {\rm (ii)}.
\end{proposition}

\begin{proof}
(i)  The main idea of the proof is due to \cite[Theorem 5.1]{KS}, which is concerned with second order (degenerate) elliptic operators. See also
 the proof of \cite[Theorem 5.5.2]{Sa} for related arguments.
For any $x_0\in M$ and $r>0$, let $B=B(x_0,r)$.
Define a bilinear form $(\bar \sE, \bar \sF)$ on $L^2(B; \mu)$ by
\begin{align*}
\bar \sE (u, v)&= \int_{B\times B}  (u(x)-u(y))(v(x)-v(y)) J(x, y) \,\mu (dx) \,\mu (dy), \\
\bar \sF &= \left\{ u\in L^2(B; \mu): \bar \sE (u, u)<\infty \right\}.
\end{align*}
One can easily check by using Fatou's lemma that $(\bar \sE, \bar \sF)$ is closable and
is a Dirichlet form on $L^2(B; \mu)$.  Let $\{\bar P_t\}$ be the $L^2$-semigroup associated
with $(\bar \sE, \bar \sF)$.
Let $\bar \sF_B$ be the closure of $\bar \sF \cap C_c (B)$.
Then $(\bar \sE, \bar \sF_B)$ is a regular Dirichlet form on $L^2(B; \mu)$, whose associated
semigroup will be denoted as $\{\bar P^B_t\}$. By \cite[Theorem 5.2.17]{CF},
$(\bar \sE, \bar \sF_B)$ is the resurrected Dirichlet form of $(\sE,  \sF_B)$.
In other words, if we denote by $\bar X^B=\{\bar X^B_t\}$ the Hunt process associated with the regular Dirichlet form
$(\bar \sE, \bar \sF_B)$ on $L^2(B; \mu)$, then $\bar X^B$ is the resurrection of $X^B=\{X^B_t\}$ in $B$,
and so $\bar X^B$ can be obtained from $X^B$ by creation through a Feynman-Kac transform.
Consequently, $\bar X^B  $ has a transition density function ${\bar p}^B(t, x, y)$ with respect to $\mu$
and ${\bar p}^B(t, x, y)\geq p^B(t, x, y)$ for every $t>0$ and $x, y\in B\cap M_0$.
This together with  $\NDL(\phi)$ implies that there exist $\varepsilon\in(0,1)$ and $c_1>0$ such that for all $x_0\in M$ and $x,$ $y \in B(x_0, \eps^2 r)\cap M_0$,
$$
\bar p^B(\phi (\eps r), x, y) \geq  p^B (\phi (\eps r), x, y) \geq \frac{c_1}{V(x_0, r)}.
$$
On the other hand, we know from \cite[Section 6.2]{CF},  $(\bar \sE, \bar \sF)$
is the active reflected Dirichlet space for $(\bar \sE, \bar \sF_B)$.
Although $(\bar \sE, \bar \sF)$ may not be regular as a Dirichlet form on $L^2(B; \mu)$,
by Silverstein \cite[Theorem 20.1]{Si}, there is a locally compact separable
metric space $\wt B$ (called regularizing space)
so that  $(\bar \sE, \bar \sF)$
is  regular   on $L^2(\wt B; \wt \mu)$ and $B$ is
intrinsically
open in $\wt B$.
Here $\wt \mu$ is an extension of $\mu$ to $\wt B$ by setting
$\wt \mu (\wt B\setminus B)=0$.
Let $\wt X=\{\wt X_t\}$ denote the Hunt process on $\wt B$ associated with the regular Dirichlet form
$(\bar \sE, \bar \sF)$   on $L^2(\wt B; \mu)$.
 Then the part process $\wt X^B=\{\wt X^B_t\}$ of $\wt X$ killed upon leaving $B$ has the same distribution as $\bar X^B$.
 Now for $f\in  \bar \sF$, by the basic property of Dirichlet form (see, for example, \cite[(1.1.4)]{CF}),
 \begin{align*}
 \bar{\sE} (f, f)&\geq \frac{1}{\phi (\eps r)} \int_B f(x) (f-\bar P_{\phi(\eps r)}f)(x) \, \mu (dx) \\
&\geq  \frac{1}{2 \phi (\eps r)} \bE^{\wt \mu} \left[ (f(\wt X_{\phi (\eps r)}) -f(\wt X_0) )^2 \right] \\
&\geq  \frac{1}{2 \phi (\eps r)} \bE^{\wt \mu} \left[ ( f (\wt X_{\phi (\eps r)}) -f(\wt X_0) )^2 ;
\phi (\eps r) < \tau_B \right] \\
&= \frac{1}{2 \phi (\eps r)} \int_{B\times B} \bar p^B(\phi (\eps r), x, y) (f(x)-f(y))^2 \,\mu (dx)\, \mu (dy) \\
&\geq  \frac{c_2}{V(x_0, r) \phi (r)} \int_{B(x_0, \eps^2r) } \int_{B(x_0, \eps^2 r)}    (f(x)-f(y))^2\, \mu (dx)\, \mu (dy) \\
&\geq  \frac{c_3 }{  \phi (r)} \int_{B(x_0, \eps^2 r) }     (f(x)- \bar f_{B(x_0, \eps^2 r)}  )^2\, \mu (dx).
\end{align*}
 Recall that  $\bar{f}_D:=\frac{1}{\mu(D)}\int_D f\,d\mu$ for any open set $D$ of $M$.
In the last two inequalities above we have used $\VD$, \eqref{polycon} and the fact that
$$\int_{B(x_0, \eps^2 r)}\left( f(x)- \bar{f}_{B(x_0, \eps^2 r)}\right)^2\,\mu(dx)
=\inf_{a\in \bR}\int_{ B(x_0, \eps^2 r)}\left( f(x)- a\right)^2\,\mu(dx).
$$
This establishes $\PI (\phi)$.

That $\PI(\phi)$ implies $\FK(\phi)$ under additional assumption $\RVD$ is given in Proposition \ref{pi-e-pre}. (Note that, under additional assumption $\RVD$, $\FK(\phi)$ is also a direct consequence of $\PHI(\phi)$, thanks to Propositions \ref{uodest} and \ref{pi-e-pre}.)

(ii) By $\VD$, \eqref{polycon} and $\NDL(\phi)$, for some $\eps\in(0,1)$,
\begin{align*}\bP^x(\tau_{B(x,r)}\ge \phi(\eps r))=&\int_{B(x,r)} p^{B(x,r)}(\phi(\eps r),x,y)\,\mu(dy)\\
\ge& \int_{B(x,\eps ^2 r)} p^{B(x,r)}(\phi(\eps r),x,y)\,\mu(dy)\ge  c_6,\end{align*} and thus $\bE^{x_0}\tau_{B(x_0,r)}\ge c_6 \phi(r).$ This proves $\E_{\phi,\ge}$.

Next, we assume that $\RVD$ is satisfied.  Let $B=B(x_0, r)$ with $x_0\in M_0$ and $r>0$,
and $B'=B(x_0,r/(2{l_\mu}))$, where ${l_\mu}>1$ is the constant in \eqref{e:rvd2}. Then, $\VD$, \eqref{polycon} and $\NDL(\phi)$ give us that for $t= \phi(r/\eps)$ with some $\eps\in(0,1)$,
$$p(t, x,y)\ge \frac{c_1}{V(x_0,r)}, \quad x,y\in B\backslash \mathcal{N}.$$
Fix $y_0\in M$ with $(1+2{l_\mu})r/(2{l_\mu}(1+{l_\mu}))< d(x_0,y_0)< (1+2{l_\mu})r/(2(1+{l_\mu}))$
(such a point $y_0$ indeed exists due to $\RVD$),
then for any $x\in B'\backslash \mathcal{N}$,
\begin{align*}\bP^x(X_t\notin B')\ge& \bP^x(X_t\in B(y_0, {r}/({2(1+{l_\mu})})))=\int_{B(y_0,r/(2(1+{l_\mu})))} p(t, x,y)\,\mu(dy)\\
\ge& \frac{c_2 V(y_0,r/(2(1+{l_\mu})))}{V(x_0,r)}\ge c_3,\end{align*}
where $\VD$ is used in the last inequality.
So, we have $\bP^x(\tau_{B'}>t)\le \bP^x(X_t\in B')\le 1-c_3$ for all $x\in B'\backslash \mathcal{N}$. Hence, by the Markov property,
$\bP^x(\tau_{B'}>kt)\le (1-c_3)^k$, and thus $\bE^x \tau_{B'}\le c_4t$. Since $\bE^{x_0}\tau_{B(x_0,r/(2{l_\mu}))}=\bE^{x_0}\tau_{B'}$, replacing $r/(2{l_\mu})$ by $r$ gives us that
$\bE^{x_0}\tau_{B(x_0,r)}\le c_5 \phi(r),$ where \eqref{polycon} is used in the inequality above. Therefore, $\E_\phi$ holds.  (Note that, by Lemma \ref{upper-e}, under $\VD$ and \eqref{polycon}, $\FK(\phi)$ implies $\E_{\phi,\le}$. Then, $\E_{\phi,\le}$ can be also deduced from $\PHI(\phi)$ directly under additional assumption $\RVD$, thanks to Propositions \ref{uodest} and \ref{pi-e-pre}.)
\qed \end{proof}

Combining all the conclusions of this and previous subsections, we can obtain the following main result in this section.
\begin{theorem} Assume that $\mu$ and $\phi$ satisfy $\VD$, $\RVD$ and \eqref{polycon} respectively. Then the following  hold
\begin{align*}
\PHI(\phi)&\,\,\Longrightarrow \UHKD(\phi)+\NDL(\phi)+\UJS +\E_\phi+  \,\J_{\phi,\le}\\
&\,\Longleftrightarrow \UHKD(\phi)+\NDL(\phi)+\UJS\\
&\,\Longleftrightarrow \UHK(\phi) + \NDL(\phi) +\UJS.
\end{align*}
\end{theorem}
\begin{proof} Note that by Corollary \ref{ujs-ndl}, $\NDL(\phi)+\UJS\Longrightarrow \J_{\phi,\le}$; and that by Proposition \ref{pi-e}, $\NDL(\phi)$ implies $\E_\phi$.
According to  Theorem \ref{T:main-1},
$\UHK(\phi)\,\,+$ conservativeness $\Longleftrightarrow \UHKD(\phi)+ \J_{\phi, \le}+ \E_\phi$.
Then the required assertion now follows from all the previous propositions.
 (Here we note that both $\PHI(\phi)$ and $\NDL(\phi)$
imply the conservativeness of the process $\{X_t\}$, see Proposition \ref{P:3.1} and
Proposition \ref{ndlb}.)
 \qed
 \end{proof}

\subsection{ H\"older regularity}\label{consehi}

Another consequence of $\NDL (\phi)$ is that, it along with $\E_{\phi,\le}$ and $\J_{\phi,\le}$
implies the joint H\"older regularity of
bounded caloric functions. In other words, $\NDL (\phi)+\E_{\phi,\le}+\J_{\phi,\le}$ imply $\PHR(\phi)$ and $\EHR$.
For our purpose, in the following lemma we use the definition of
$\NDL(\phi)$ with $\eps\phi(r)$ and $\phi^{-1}(\eps t)$ replaced by
$ \phi(\eps r)$ and $\eps\phi^{-1}(t)$, respectively.

\begin{lemma}\label{phi-l-2}  Suppose that $\VD$, \eqref{polycon} and $\NDL(\phi)$ hold. For every $0<\delta\le \eps
$ $($where $\eps$ is the constant in the definition of $\NDL(\phi)$$)$, there exists a constant $C_1>0$ such that for every $r>0$, $x\in M_0$, $t\ge \delta\phi( r)$ and any compact set $A\subset [t-\delta \phi(r),t-\delta\phi( r)/2]\times B(x, \phi^{-1}(\eps \delta \phi( r)/2))$,
\be \label{e:6.13}
\bP^{(t,x)} (\sigma_A< \tau_{[t-
\delta \phi (r), t]
\times B(x,r)})\ge C_1\frac{m\otimes \mu(A)}{V(x,r)\phi(r)},
\ee
 where $m\otimes \mu$ is a product of the Lebesgue measure on $\bR_+$ and $\mu$ on $M$.  \end{lemma}

\begin{proof}
The proof is almost the same as that for \cite[Lemma 4.9(i)]{CKK2}.
Let $\tau_r=\tau_{[t-\delta \phi(r), t]\times B(x,r)}$ and $A_s=\{y\in M: (s,y)\in A\}.$ For any $t,r>0$ and $x\in M_0$,  \begin{align*}
\delta \phi(r)\bP^{(t,x)}(\sigma_A<\tau_r)\ge &\int_0^{\delta \phi(r)}\bP^{(t,x)}\left( \int_0^{\tau_r}{\bf 1}_A(t-s, X_s)\,ds>0\right)\,du\\
\ge& \int_0^{\delta \phi(r)}\bP^{(t,x)}\left( \int_0^{\tau_r}{\bf 1}_A(t-s, X_s)\,ds>u\right)\,du\\
=&\bE^{(t,x)} \left[\int_0^{\tau_r} {\bf 1}_A(t-s, X_s)\,ds\right].
\end{align*}
Note that, for any $t\ge \delta \phi(r)$,
 \begin{align*}
\bE^{(t,x)} \left[\int_0^{\tau_r} {\bf 1}_A(t-s, X_s)\,ds\right]
=&\int_{\delta\phi( r)/2}^{\delta \phi(r)}\bP^{(t,x)} \left(\big(t-s, X_s^{B(x,r)}\big)\in A\right)\,ds \\
=&\int_{ \delta \phi(r)/2}^{\delta\phi( r)}\bP^{x} \left(X_s^{B(x,r)}\in A_{t-s}\right)\,ds\\
=&\int_{\delta\phi( r)/2}^{\delta \phi(r)}\,ds\int_{A_{t-s}} p^{B(x,r)} (s, x,y)\,\mu(dy).
\end{align*} By $\VD$, \eqref{polycon} and $\NDL(\phi)$, for any $s\in [\delta \phi(r)/2, \delta \phi(r)]$ and $y\in B(x,\phi^{-1}(\eps \delta \phi( r)/2))\setminus\mathcal{N}$,
$$
p^{B(x,r)}(s, x,y)\ge \frac{c_1}{V(x,r)}.
$$
Thus,
$$
\bE^{(t,x)} \left[\int_0^{\tau_r} {\bf 1}_A(t-s, X_s)\,ds\right]\ge \frac{c_1}{V(x,r)}\int_{ \delta \phi(r )/2}^{\delta \phi(r)}\,ds\int_{A_{t-s}} \,\mu(dy)=\frac{c_1m\otimes \mu(A)}{V(x,r)}.$$ Combining all the conclusions above, we obtain the desired assertion. \qed
\end{proof}

\begin{proposition}\label{ndl-holder}
Assume that $\VD$, \eqref{polycon}, $\NDL(\phi)$, $\E_{\phi,\le}$ and $\J_{\phi,\le}$ hold. For every $\delta \in (0, 1)$,
there exist positive constants
$C>0$ and $\gamma\in(0,1]$, where $\gamma$ is independent of $\delta$,
 so that  for any bounded caloric function $u$ in
$Q(t_0,x_0,\phi(r),r)$,
there is a properly exceptional set ${\cal N}_u\supset {\cal N}$ such that
$$|u(s,x)-u(t, y)|\le C\left( \frac{\phi^{-1}(|s-t|)+d(x, y)}{r} \right)^\gamma  \esssup_{ [t_0, t_0+\phi (r)] \times M}\,|u|
$$
for  every $s, t\in (t_0+\phi(r)-\phi(\delta r), t_0+\phi (r))$ and $x, y \in B(x_0, \delta r)\setminus {\cal N}_u$.
 In other words, under $\VD$ and \eqref{polycon}, $\NDL(\phi)+\E_{\phi,\le}+\J_{\phi,\le}$ imply {\rm PHR}$(\phi)$ and {\rm EHR}.
 \end{proposition}

\begin{proof} With estimate \eqref{e:6.13}, the result can be proved in exactly the same way as that for \cite[Theorem 4.14]{CK1}.
We omit the details here.
We note that in the present paper the time evolves as $V_s=V_0-s$, which is opposed to $V_s=V_0+s$ as in \cite[p.\! 37]{CK1}, so here we should reserve the time interval in the statement.  (The statement of \cite[Theorem 3.1]{CKK1} should be corrected in the same way.)
\qed
\end{proof}

\medskip

The following
consequence of  H\"older regularities will be used in Subsection \ref{sec7-2}.

\begin{lemma}\label{green-inequality}
Suppose {\rm EHR} holds. Let $D\subset M$ be an open set with
 $\esssup_{y\in D\cap M_0} \bE^y\tau_{D}<\infty$. Fix a function $f\in B_b({D})$ and set $u=G^{D} f$. Then for any $B(x_0,r)\subset {D}$ and $0<{r_1}\le r$,
$$
 {\rm osc}_{B(x_0,{r_1})\cap M_0} u\le 2
 \sup_{y\in B(x_0, r)\cap M_0} |f (y)|
\, \sup_{y\in B(x_0,r) \cap M_0}\bE^y \tau_{B(x_0, r)}  + c\left( {{r_1}}/{r}\right)^{\theta}
\sup_{z\in D\cap M_0} | u (z) | ,
$$
where $c>0$ and $\theta\in(0,1]$ only depend on  the constants in {\rm EHR}.
\end{lemma}

\begin{proof}   Note that for any $x\in {D}\cap M_0$,
$$
G^{D} |f|(x)= \bE^x\left[ \int_0^{\tau_{D}} | f(X_t)| \,dt\right] \le
\sup_{y\in D\cap M_0}  |f (y)|
 \, \bE^x\tau_{D}.$$
Consequently, for any $r_1 \in (0, r)$,
\begin{align*}
{\rm osc}_{B(x,{r_1})\cap M_0}  G^{B(x_0, r)} f & \le 2 \sup_{y\in B(x_0,r_1)\cap M_0}  G^{B(x_0, r)}| f| (y) \\
& \le  2 \, \sup_{y\in B(x_0, r)\cap M_0} | f(y)| \, \sup_{y\in B(x_0,r_1) \cap M_0}
\bE^y \tau_{B(x, r)} .
\end{align*}
Since  $  G^{D}  f (y)-G^{B(x_0,r)}f (y) = \bE^y [ G^D f(X_{\tau_{B(x_0, r)}})] = \bE^y [ u(X_{\tau_{B(x_0, r)}})]
$ is harmonic in $B(x_0, r)$, and $u=0$ outside $D$,
we have by $\EHR$ and Remark \ref{R:phrehr}(ii) that
\begin{align*}
&{\rm osc}_{B(x,{r_1}) \cap M_0} u \\
&\leq  {\rm osc}_{B(x_0,{r_1})\cap M_0}  G^{B(x_0, r)} f  + {\rm osc}_{B(x_0,{r_1})\cap M_0}
 ( G^{D}  f -G^{B(x_0,r)}f )\\
&\leq   2 \, \sup_{y\in B(x, r)\cap M_0} | f(y)| \, \sup_{y\in B(x_0,r_1) \cap M_0} \bE^y \tau_{B(x_0, r)}
+ c (r_1/r)^\theta \sup_{y\in D\cap M_0} | \bE^y [ u(X_{\tau_{B(x, r)}})]|\\
&\le  2  \sup_{y\in B(x, r)\cap M_0} |f (y)|  \, \sup_{y\in B(x,r) \cap M_0}\bE^y \tau_{B(x, r)}
+ c\left( {{r_1}}/{r}\right)^{\theta}
\sup_{z\in D \cap M_0} | u (z) |.
\end{align*}
  This proves the lemma.  \qed
  \end{proof}

\section{Equivalences of $\PHI(\phi)$}\label{pf of 1-21}

We have already given some part of the proof
of Theorem \ref{T:PHI}
in Section \ref{sectHarn}. In this section, we
will complete the proof.
In Subsection \ref{suffphi71}, we prove $(1)\Longleftrightarrow (2)\Longleftrightarrow (3)\Longleftrightarrow (4)$.
$(1)\Longleftrightarrow (5)\Longleftrightarrow (6)$ will be proved in Subsection \ref{sec7-2}, and
$(1)\Longleftrightarrow (7)$  in Subsection \ref{sec7-3}.

\subsection{$\PHI(\phi) \Longleftrightarrow \PHI^+(\phi) \Longleftrightarrow  \UHK(\phi)+\NDL(\phi)+\UJS \Longleftrightarrow \NDL(\phi)+\UJS$}\label{suffphi71}

In this subsection, we will establish  $(1)\Longleftrightarrow (2)\Longleftrightarrow (3)\Longleftrightarrow (4)$ in
Theorem \ref{T:PHI}.
Since $(1)\Longrightarrow (3)$ is already proved in Subsection \ref{phi-section-1}, and $(1)\Longrightarrow (2)$ and $(3)\Longrightarrow (4)$ hold trivially,
it remains to show that
prove $(4)\Longrightarrow(3)\Longrightarrow (2)$.

 \begin{lemma}\label{phi-l-1} Assume that $\VD$, \eqref{polycon}, $\UHK(\phi)$, $\NDL(\phi)$ and $\UJS$.
Let $\delta\le \eps $ $($where $\eps\in(0,1)$ is the constant in the
definition of $\NDL(\phi)$$)$, and $\theta\ge 1/2$. Let
$0<\delta_0<\delta$ and $0<\delta_1<\delta_2<\delta_3<\delta_4$ such
that $(\delta_3 -\delta_2)\phi(r)\ge \phi(\delta_0 r)$ and
$\delta_4\phi(r)\le \phi(\delta r)$ for all $r>0$. Set
 $$Q_1=(t_0,t_0+\delta_4\phi( r))\times B(x_0,\delta_0^2 r),\quad Q_2=(t_0,t_0+\delta_4\phi( r))\times B(x_0, r)$$ for $x_0\in M$, $t_0\ge0$ and $r>0.$  Define
$$
Q_3 =[ t_0+\delta _1\phi(r),t_0+ \delta_2 \phi(r) ]\times
B(x_0,\delta_0^2r/2 )\setminus\mathcal{N}
$$
and
$$ Q_4 =[ t_0+\delta_3\phi( r), t_0+\delta_4\phi(r)] \times B(x_0,\delta_0^2r/2)\setminus\mathcal{N}.
$$
 Let $f:(t_0,\infty)\times M\to \bR_+$ be bounded and supported in $(t_0,\infty)\times B(x_0,(1+\theta) r)^c.$ Then there is a constant $C_2>0$ such that the following holds:
$$
\bE^{(t_1,y_1)}f(Z_{\tau_{Q_1}})\le C_2 \bE^{(t_2,y_2)}f(Z_{\tau_{Q_2}})
\quad \hbox{for every }
(t_1,y_1)\in Q_3\textrm{ and  }(t_2,y_2)\in Q_4.
$$
\end{lemma}

\begin{proof}
The proof is  the same as that of \cite[Lemma 5.3]{CKK1}.
We present the proof here for the sake of completeness.

 Without loss of generality, we may and do assume that $t_0=0$. For $x_0\in M$ and $s>0$, set $B_{s}=B(x_0,s)$.
 By Lemma \ref{Levy-sys}, for any $(t_2,y_2)\in Q_4$,
\begin{align}\label{pr-lemm-01}
\bE^{(t_2,y_2)}f(Z_{\tau_{Q_2}})&=\bE^{(t_2,y_2)}f(t_2-(\tau_{B_{r}}\wedge t_2), X_{\tau_{B_{r}}\wedge t_2})\nonumber\\
&=\bE^{(t_2,y_2)} \left[\int_0^{t_2}{\bf 1}_{\{t\le \tau_{B_{r}}\}} \,dt\int_{B_{(1+\theta)r}^c} f(t_2-t,v)J(X_t,v)\,\mu(dv)\right]
  \nonumber\\
&=\int_0^{t_2} \,dt\int_{B_{(1+\theta)r}^c} f(t_2-t,v)\bE^{(t_2,y_2)} \left[{\bf1}_{\{t\le \tau_{B_{r}}\}} J(X_t,v)\right]\,\mu(dv)\nonumber\\
&= \int_0^{t_2}\,ds\int_{B_{(1+\theta)r}^c} f(s,v) \bE^{(t_2,y_2)} \left[{\bf1}_{\{t_2-s\le \tau_{B_{r}}\}} J(X_{t_2-s},v)\right]\,\mu(dv)
  \nonumber\\
&=\int_0^{t_2}\,ds\int_{B_{(1+\theta)r}^c} f(s,v)\,\mu(dv)\int_{B_{r}}p^{B_{r}} (t_2-s, y_2,z)J(z,v)\,\mu(dz)\\
&\ge \int_0^{t_1}\,ds\int_{B_{(1+\theta)r}^c} f(s,v)\,\mu(dv)\int_{B_{\delta_0^2r}}p^{B_{r}}( t_2-s, y_2,z)J(z,v)\,\mu(dz). \nonumber
\end{align}
Since for $s\in [0,t_1]$,
$\phi(\delta_0r)\le t_2-t_1\le t_2-s\le \phi(\delta r),$ by $\VD$, \eqref{polycon} and $\NDL(\phi)$, we know that the right hand side of the inequality above is greater than or equal to
$$\frac{c_1}{V(x_0,r)}\int_0^{t_1}\,ds\int_{B_{(1+\theta)r}^c} f(s,v)\,\mu(dv)\int_{B_{\delta_0^2r}}J(z,v)\,\mu(dz).$$ So the proof is complete, once we can obtain that for every $(t_1,y_1)\in Q_3$,
\begin{equation}\label{pr-lemm-02}
\bE^{(t_1,y_1)} f(Z_{\tau_{Q_1}})\le \frac{c_2}{V(x_0,r)}\int_0^{t_1}\,ds\int_{B_{(1+\theta)r}^c} f(s,v)\,\mu(dv)\int_{B_{\delta_0^2r}}J(z,v)\,\mu(dz).\end{equation}

Similar to the argument for \eqref{pr-lemm-01}, we have by using Lemma \ref{Levy-sys},
\begin{align*}
\bE^{(t_1,y_1)}f(Z_{\tau_{Q_1}})=&\int_0^{t_1}\,ds\int_{B_{(1+\theta)r}^c} f(s,v)\,\mu(dv)\int_{B_{\delta_0^2 r}}p^{B_{\delta_0^2 r}} (t_1-s, y_1,z)J(z,v)\,\mu(dz)\\
=&\int_0^{t_1}\,ds\int_{B_{\delta_0^2 r}}p^{B_{\delta_0^2 r}} (t_1-s, y_1,z)\,\mu(dz)\int_{B_{(1+\theta)r}^c} f(s,v)J(z,v)\,\mu(dv).
\end{align*}
Notice that
\begin{align*}
&\int_{B_{\delta_0^2 r}}p^{B_{\delta_0^2 r}} (t_1-s, y_1,z)\,\mu(dz)\int_{B_{(1+\theta)r}^c} f(s,v)J(z,v)\,\mu(dv)\\
&=\int_{B_{\delta_0^2 r}\backslash B_{3\delta_0^2 r/4}}p^{B_{\delta_0^2 r}} (t_1-s, y_1,z)\,\mu(dz)\int_{B_{(1+\theta)r}^c} f(s,v)J(z,v)\,\mu(dv)\\
&\quad + \int_{B_{3\delta_0^2 r/4}}p^{B_{\delta_0^2 r}} (t_1-s, y_1,z)\,\mu(dz)\int_{B_{(1+\theta)r}^c}f(s,v)J(z,v)\,\mu(dv)\\
&=:I_1+I_2.
\end{align*}
 On the one hand, when $z\in (B_{\delta_0^2 r}\backslash B_{3\delta_0^2 r/4})\cap M_0$, we have $\delta_0^2 r/4\le d(y_1,z)\le 3\delta_0^2 r/2$, and so by $\UHK(\phi)$, $\VD$ and \eqref{polycon},
$$
p^{B_{\delta_0^2 r}} (t_1-s, y_1,z)\le \frac{c_3t_1}{V(y_1,d(y_1,z))\phi(d(y_1,z))}\le \frac{c_4}{V(x_0,r)}
$$
 for some constants $c_3,c_4>0$. Hence,  $\int_0^{t_1}I_1\,ds$ is less than or equal to the right hand side of
 \eqref{pr-lemm-02}. On the other hand, for $z\in B_{3\delta_0^2 r/4}$, by $\UJS$ and $\VD$,
\begin{align*}
\int_{B_{(1+\theta)r}^c}J(z,v)f(s,v)\,\mu(dv)\le & \frac{c_5}{V(x_0,r)}\int_{B(z,\delta_0^2 r/4)} J(w,v)\,\mu(dw) \int_{B_{(1+\theta)r}^c}f(s,v)\,\mu(dv)\\
\le& \frac{c_5}{V(x_0,r)}\int_{B_{\delta_0^2 r}}J(w,v)\,\mu(dw)\int_{B_{(1+\theta)r}^c} f(s,v)\,\mu(dv).
\end{align*}
Note that the right hand side of the above inequality does not
depend on $z$. Multiplying both sides by $p^{B_{\delta_0^2 r}}
(t_1-s, y_1,z)$ and integrating over $z\in B_{3\delta_0^2 r/4}$ and
then over $s\in[0,t_1]$, we obtain that $\int_0^{t_1}I_2\,ds$ is
also less than or equal to the right hand side of
\eqref{pr-lemm-02}. This proves the lemma. \qed
\end{proof}

Once again, in the following lemma we use the definition of
$\NDL(\phi)$ with $\eps\phi(r)$ and $\phi^{-1}(\eps t)$ replaced by
$ \phi(\eps r)$ and $\eps\phi^{-1}(t)$, respectively.

\begin{lemma} \label{phi-l-3} Suppose that $\VD$, \eqref{polycon}  and $\NDL(\phi)$ hold. Let
$0<\delta\le \eps /4$ such that $4\delta\phi( 2 r)\le \varepsilon \phi(r)$ for all $r>0$, where $\eps\in(0,1)$ is the constant in the definition of $\NDL(\phi)$. Then
there exists a constant $C_3>0$ such that for every $R>0$, $r\in (0,\phi^{-1}(\varepsilon \delta \phi( R)/2)/2 ]$, $x_0\in M$, $\delta \phi(R)/2 \le t-s\le4\delta\phi( 2 R)$,  $x\in B(x_0,\phi^{-1}(\varepsilon \delta
\phi( R)/2)/2)\setminus\mathcal{N}$, and $z\in
B(x_0,$ $\phi^{-1}(\varepsilon \delta \phi( R)/2)) \setminus\mathcal{N}$,
$$\bP^{(t,z)}(\sigma_{U(s,x,r)}\le \tau_{[s,t]\times B(x_0,R)})\ge C_3\frac{V(x,r)}{V(x,R)},$$ where $U(s,x,r)=\{s\}\times B(x,r).$
\end{lemma}

\begin{proof}
The left hand side of the desired estimate is equal to
\be\label{eq:no3930}
\bP^z(X_{t-s}^{B(x_0,R)}\in B(x,r))=\int_{B(x,r)} p^{B(x_0,R)} (t-s, z,y)\,\mu(dy).
\ee
By $\VD$, \eqref{polycon}, $\NDL(\phi)$, and the facts that
$\delta \phi(R)/2
\le t-s\le 4\delta \phi(2R)$ and $B(x,r)\subset B(x_0,\phi^{-1}(\varepsilon \delta \phi( R)/2))$,
\eqref{eq:no3930} is greater than or equal to
$$c_1\frac{V(x,r)}{V(z,R)}\ge c_2 \frac{V(x,r)}{V(x,R)}.$$ This proves the desired assertion.
\qed\end{proof}

Having these two lemmas as well as Lemma \ref{phi-l-2} at hand, one
can obtain the following form of $\PHI^+(\phi)$.

\begin{theorem}\label{phi-probab}
Suppose that $\VD$ and \eqref{polycon} hold.
Under $\UHK(\phi)$, $\NDL(\phi)$ and $\UJS$, the following
$\PHI^+(\phi)$ holds:
there exist constants $\delta >0$, $C>1$ and $K\geq 1$
  such that for every $x_0\in M\setminus \cal{N}$,
$t_0\ge0$, $R>0$ and every non-negative function $u$ on
$[0,\infty)\times M$ that is
caloric  on $Q:=(t_0, t_0+ 4 \delta
\phi(C R))\times B(x_0, C R)$, we have
\begin{equation}\label{e:PHI}
\esssup_{(t_1,y_1)\in Q_-}u(t_1,y_1)\le K \,\,\essinf_{(t_2,y_2)\in Q_+}u(t_2,y_2),
\end{equation}
where $Q_-=[t_0+\delta \phi(C R), t_0+2 \delta \phi(C R)]\times
B(x_0,R)$ and $Q_+=[t_0+3 \delta \phi( C R), t_0+ 4\delta \phi( C R))\times
B(x_0,R)$.
\end{theorem}

\begin{proof}
 Let $\eps\in (0, 1)$ be the constant in $\NDL (\phi)$.
 Take and fix some $\delta \in (0, \eps /4 ]$ so that $\delta \phi (2r) \leq \phi (\eps r)/4$ for all $r>0$
 and take $\delta_0 \in (0, \delta)$ so that $\phi (\delta_0 r)\leq \delta \phi (r)$
 for all $r>0$. The existence of such $\delta$ and $\delta_0$ is guaranteed by the assumption
 \eqref{polycon}.
 We choose $\delta$ and $\delta_0$ in such a way so that Lemma \ref{phi-l-1} holds
  by taking $\delta_j$ to be $j\delta$ for $j=1,2,3,4$ there.
 Condition \eqref{polycon} ensures that there is a constant $c_0 \in (0, 1/2)$ so that
 $\phi^{-1}(\delta \eps \phi (r)/2) \geq c_0r$ for every $r>0.$
 Take
\begin{equation}\label{e:4.6a}
C=(2/c_0) +2  \quad \hbox{and} \quad  C_0=C-2=2/c_0.
\end{equation}
The reason of defining such $C_0$ is that the conclusion of Lemma \ref{phi-l-3}
holds for any $x, z\in B(x_0, R/C_0)$.

Let $u$ be a  non-negative function on
$[0,\infty)\times M$ that is
caloric on $Q:=(t_0, t_0+ 4 \delta
\phi(C R))\times B(x_0, C R)$. We will show \eqref{e:PHI} holds.

The proof below is mainly based on that of \cite[Theorem 5.2]{CKK1}
with some non-trivial modifications; see also the proof of
\cite[Proposition 4.3]{CK1} or of \cite[Theorem 4.12]{CK2},
whose idea is originally due to \cite[Theorem 3.1]{BL2}.
Truncating $u$ by $n$ outside $Q$ and then passing $n\to \infty$ if
needed, without loss of generality, we may and do assume that
$t_0=0$, and that the function $u$ is bounded on $Q$, see Step 3 in
the proof of  \cite[Theorem 5.2]{CKK1} (e.g.\ page 1085 in
\cite{CKK1}).
Furthermore, by looking at $au+b$ for suitable constants $a$ and
$b$, we may and do assume that $\inf_{(t,y)\in Q_+}u(t,y)=1/2.$ Let
$(t_*,y_*)\in Q_+$ be such that $u(t_*, y_*)\le 1$. It is enough to
show that $u(t,x)$ is bounded from above in $Q_-$ by a constant that
is independent of the function $u$.

For any $t\ge \delta \phi(r)$, set $Q^\downarrow(t,\delta,x,r)=[t-  \delta\phi(r),t]\times B(x, r)$. Note that
$$
m\otimes \mu (Q^\downarrow(t,\delta,x,r))= \delta\phi(r) V(x,  r).
$$
By Lemma \ref{phi-l-2},
there exists a constant $c_1\in(0,1/2)$
so that for any $r\le R/2$ and any compact set $D$ satisfying that
$$
D\subset \left[t-  \delta\phi(r),t-\frac{1}{2}\delta\phi(r)\right]\times B(x,c_0 r) \subset Q^\downarrow(t,\delta,x,r)
$$
and
$$
m\otimes \mu (D)/ m\otimes \mu (Q^\downarrow(t,\delta,x,r))\ge \frac{c_0^{d_2}}{4\wt C_\mu},
$$
we have
$$
\bP^{(t,x)}(\sigma_D <\tau_{Q^\downarrow(t,\delta,x,r)})\ge c_1,
$$
 where $\widetilde C_\mu$ and $d_2$ are the constants in \eqref{eq:vdeno}.
Let $C_2$ be the constant $C_2$ in Lemma \ref{phi-l-1} with
$\delta_j=j\delta$ and $\theta=1/2$. Define $$\eta=\frac{c_1}{3},\quad \xi =\frac{1}{3}  \wedge (C_2^{-1} \eta).$$
We claim that there is a universal constant $K \geq 2$ to be determined later, which is independent
of $R$ and the function $u$, such that $u\le  K$ on $Q_-$. We are going to prove this by contradiction.

Suppose this is not true. Then there is some point $(t_1, x_1)\in Q_-$ such that $u(t_1, x_1)\ge K$.
We will
show that there are a constant $\beta>0$ and a sequence of
points $\{(t_k, x_k)\}$ in $ [t_0+\delta \phi (CR)/2, t_0+2\delta \phi (C R))\times B(x_0, 2R)
\subset Q$ so that $u(t_k, x_k) \geq (1+\beta)^{k-1} K$,
which contradicts to the assumption that $u$ is bounded on $Q$.

Recall that $\beta_1, \beta_2, c_3$ and $ c_4$ are the constants in \eqref{polycon}.
Then, by \eqref{eq:vdeno} and \eqref{polycon}, we have for all $x\in M$ and
all $0<r_1< r_2 \wedge r_3 <\infty$:
\begin{equation}\label{e:ff}
\frac{V(x,r_1 )\phi(r_1)}{V(x,r_2)\phi(r_3)}\ge \frac{1}{c_4 \widetilde
C_\mu} \left(\frac{r_1}{r_2 }\right)^{d_2} \left(\frac{r_1}{r_3 }\right)^{\beta_2} .
\end{equation}
Let $C_3$ be the constant in Lemma \ref{phi-l-3},
and set $r:=RK^{-1/(2(d_2+\beta_2))} $.
We take $K\geq 2$ large enough
so that $K\geq (2\wt C_\mu/ (C_3 \xi \delta_0^{2d_2}))^2 (2C_0)^{d_2}$ and that, in view of \eqref{polycon},
$$
r< R/8 \ \hbox{ and } \ \phi(r)< \frac{1}{8}\phi(R) \qquad \hbox{for all } R>0
\hbox{ and } r=RK^{-1/(2(d_2+\beta_2))}.
$$

 With such $r$, we have by \eqref{e:ff}
\begin{equation}\label{e:4.5}
 \frac{m\otimes \mu (Q^\downarrow(t,\delta,x,r))}{\phi (R) V(x, C_0 R)}=
 \frac{\delta \phi (r) V(x,  r)}{\phi ( R) V(x, C_0 R)}
 \geq \frac{\delta}{c_4 \wt C_\mu C_0^{d_2} \sqrt{K}}.
\end{equation}

Take $\tilde{t}=t_1+(5/2) \delta\phi(r)$ and define $\tilde{U}=\{\tilde{t}\}\times B(x_1, \delta_0^2 r/2)$.
Observe that $t_*-\tilde t\geq \frac12 \delta\phi (CR)$ since $t_*-t_1\geq \delta\phi (CR)$.
 If the
 caloric  function $u\ge \xi K$ on $\tilde{U}$, we would have by \eqref{e:4.6a} and
 Lemma \ref{phi-l-3}  that
\begin{align*}
1\ge u(t_*, y_*)&=\bE^{(t_*,y_*)} u(Z_{\sigma_{\tilde{U} }\wedge \tau_{Q_*}})\ge \xi K \bP^{(t_*,y_*)}(\sigma_{\tilde{U}}\le \tau_{Q_*})\ge \xi K \frac{C_3V(x_1,\delta_0^2 r/2 ) }{V(x_1,C_0R) } \\
&\geq  \frac{C_3 \xi K} {\wt C_\mu }  (\delta_0^2 r/(2C_0 R))^{d_2}\geq \frac{C_3 \xi \delta_0^{2d_2} \sqrt{K}}
{(2C_0)^{d_2} \wt C_\mu} \geq 2,
\end{align*}
where $Q_*= [t_1-\delta \phi (r) , t_*] \times B(x_0,C_0 R)$.
 This contradiction yields that
$$
\hbox{there is some } y_1\in B(x_1, \delta_0^2 r/2) \hbox{ so that }
u(\tilde{t}, y_1)< \xi K.
$$
We next show that
\begin{equation}\label{e:4.10a}
\bE^{(t_1,x_1)}\left[u(Z_{\tau_r}):X_{\tau_r}\notin B(x_1, 3r/2)\right]\le \eta K,
\end{equation}
 where $\tau_r:=\tau_{(t_1-\delta \phi (r), t_1+3\delta \phi (r))\times B(x_1, \delta_0^2 r)}$.
 If it is not true, then we would have by Lemma \ref{phi-l-1} with $\delta_j=j\delta$ $(j=1,2,3,4)$ and
 $\theta=1/2$ that
 \begin{align*}
\xi K \, > \, u(\tilde{t}, y_1) \ge& \bE^{(\tilde{t},y_1)}\left[u(Z_{\tau_{[t_1-\delta\phi(r), t_1+3\delta\phi(r)]\times B(x_1,r)}}):X_{\tau_{[t_1-\delta\phi(r),t_1+3\delta\phi(r)]\times B(x_1,r)}}\notin B(x_1,3r/2)\right]\\
\ge&  C_2^{-1} \bE^{(t_1,x_1)}\left[ u(Z_{\tau_{r}}): X_{\tau_r}\notin B(x_1,3r/2))\right] \\
>& C_2^{-1} \eta K\ge \xi K,
\end{align*}
which is a contradiction. This establishes \eqref{e:4.10a}.

Let $A$ be any compact subset of
$$\tilde{A}:=\left\{(s,y)\in  \Big[t_1-\delta\phi(r),t_1-\frac{1}{2}\delta\phi(r)\Big]\times B(x_1, c_0r ): u(s,y)\ge \xi K\right\},$$ and
define $U_1=\{t_1\}\times B(x_1, \delta_0^2 r)$.  By Lemmas \ref{phi-l-2} and
\ref{phi-l-3} and the strong Markov property,
 \begin{equation}\label{e:4.6}\begin{split}
1& \ge u(t_*,y_*)\ge \bE^{(t_*,y_*)} [u(Z_{\sigma_A}): \sigma_A\le \tau_{Q_*} ] \\
&\ge \bE^{(t_*,y_*)} [u(Z_{\sigma_A}): \sigma_{U_1}<\tau_{Q_*},
 \sigma_A< \tau_{[t_1-\delta \phi (r), t_*] \times B(x_1, 2r)} ] \\
&\ge  \bP^{(t_*,y_*)} (\sigma_{U_1}<\tau_{Q_*}) \inf_{z\in B(x_1,r/2)}\bE^{(t_1,z)} [u(Z_{\sigma_A}):  \sigma_A<\tau_{ [t_1-\delta \phi (r), t_*]\times B(z, r)}]\\
&\ge  C_3\frac{V(x_1,\delta_0^2 r)}{V(x_1,C_0 R)} \cdot \xi K C_1
\inf_{z\in B(x_1,r/2)}
\frac{m\otimes \mu(A)}{V(z, r)\phi(r)}
\\
&\ge \frac{C_1 C_3  \xi K}{c_4 \wt C_\mu} \Big(\frac{\delta_0^2}{2}\Big)^{d_2}\frac{m\otimes \mu(A)}{V(x_1,C_0R)\phi(R)},
\end{split}
\end{equation}
where in the third inequality we used the fact that
$\tau_{[t_1-\delta \phi (r), t_*] \times B(x_1, 2r)}   \leq \tau_{ Q_*}$.
Since $A$ is an arbitrary compact subset of $\wt A$, we have by
 \eqref{e:4.6} that
$$
\frac{m\otimes \mu(\wt A)}{V(x_1,C_0 R)\phi(R)} \leq
\frac{ c_4 \wt C_\mu}{C_1 C_3  \xi K} \left( \frac{2}{\delta_0^2}\right)^{d_2}.
$$
Thus by  \eqref{e:4.5},
$$
\frac{m\otimes \mu(\wt A)}{m\otimes \mu(Q^\downarrow(t_1,\delta,x_1,r))  } \le
\frac{ c_4^2 \wt C_\mu^2 C_0^{d_2}}{\delta C_1 C_3  \xi \sqrt{K}} \left( \frac{2}{\delta_0^2}\right)^{d_2},
$$
which is no larger than $\frac{c_0^{d_2}}{4\wt C_\mu}$ by taking $K$ sufficiently large.
Let
$$D= \Big[t_1-\delta\phi(r),t_1-\frac{1}{2}\delta\phi(r)\Big]\times B(x_1, c_0 r)\setminus \tilde{A}
$$
 and $M=\sup_{(s,y)\in Q^\downarrow(t_1,\delta,x_1,3r/2)}u(s,y)$.
Note that
$$
\frac{m\otimes \mu(\wt D)}{m\otimes \mu(Q^\downarrow(t_1,\delta,x_1,r))  }
= \frac{\delta \phi (r) V(x_1, c_0r)}{2 \delta \phi (r) V(x_1, r)}
-\frac{m\otimes \mu(\wt A)}{m\otimes \mu(Q^\downarrow(t_1,\delta,x_1,r))}
\geq \frac{c_0^{d_2}}{4\wt C_\mu}.
$$
  We have by \eqref{e:4.10a},
  \begin{align*}
K &\le  u(t_1,x_1)=\bE^{(t_1,x_1)} [u(Z_{\sigma_D\wedge \tau_r})]\\
&= \bE^{(t_1,x_1)} [u(Z_{\sigma_D\wedge \tau_r}): \sigma_D<\tau_r]+ \bE^{(t_1,x_1)} [u(Z_{\sigma_D\wedge \tau_r}): \sigma_D\ge \tau_r, X_{\tau_r}\notin B(x_1,3r/2)]\\
& \quad +\bE^{(t_1,x_1)} [u(Z_{\sigma_D\wedge \tau_r}): \sigma_D\ge \tau_r, X_{\tau_r}\in B(x_1,3r/2)]\\
&\le  \xi K \bP^{(t_1,x_1)}(\sigma_D<\tau_r)+ \eta K+ M \bP^{(t_1,x_1)}( \sigma_D\ge \tau_r).
\end{align*}
Therefore,
$$
M/K\ge \frac{1-\eta-\xi\bP^{(t_1,x_1)}(\sigma_D<\tau_r) }{\bP^{(t_1,x_1)}(\sigma_D\ge\tau_r)}\ge \frac{1-\eta-\xi c_1}{1-c_1}\ge \frac{1-(2c_1)/3}{1-c_1}=:1+2\beta,
$$
where $\beta=c_1/(6(1-c_1))$.
Consequently, there exists a point $(t_2,x_2)\in Q^\downarrow(t_1,\delta,x_1,2r)\subset Q$ such that $u(t_2,x_2)\ge (1+\beta)K=:K_2.$

Iterating the procedure above, we can find a sequence of points $\{(t_k,x_k)\}_{k=1}^\infty$ in
 $ [t_0+\delta \phi (C R)/2, t_0+2\delta \phi (CR))\times B(x_0, 2R)$ in the following way. Following the above argument with $(t_2,x_2)$ and $K_2$ in place of $(t_1,x_1)$ and $K$ respectively, we obtain that there exists a point $(t_3,x_3)\in  Q^\downarrow(t_2,\delta,x_2,2r_2)$ such that
$$r_2= RK_2^{-1/(d_2+\beta_2)}=(1+\beta)^{-1/(d_2+\beta_2)}RK^{-1/(d_2+\beta_2)}$$ and $$u(t_3,x_3)\ge (1+\beta)K_2=(1+\beta)^2K=:K_3.$$ We continue this procedure to obtain a sequence of points $\{(t_k,x_k)\}$ such that $(t_{k+1},x_{k+1})\in  Q^\downarrow(t_k,\delta,x_k,2r_k)$ with
$$r_k:=RK_k^{-1/(d_2+\beta_2)}= (1+\beta)^{-(k-1)/(d_2+\beta_2)}  RK^{-1/(d_2+\beta_2)},$$
 and
$$
u(t_{k+1}, x_{k+1})\ge (1+\beta)^k K=:K_{k+1}.
$$
As $0\le t_k-t_{k+1}\le \delta\phi(2r_k)$ and $d(x_k,x_{k+1})\le 2r_k$, we can take $K$ large enough (independent of $R$ and $u$) so that $(t_k,x_k)\in
 [t_0+\delta \phi (CR)/2, t_0+2\delta \phi (C R))\times B(x_0, 2R)$ for all $k$.
This is a contradiction because $u(t_k,x_k)\ge (1+\beta)^{k-1}K$ goes to infinity as $k\to \infty$, while $u$ is bounded on $Q$. We conclude that $u$ is bounded by $K$ in $Q_-$.
The proof is complete. \qed
\end{proof}

\begin{remark}\rm
In \cite[Proposition 3.3]{BBK2}, another proof of parabolic Harnack inequalities is given
by using the Balayage formula. We point out that there is a minor error in its proof there.
Indeed,
in  (3.7) and (3.8)  of \cite{BBK2} and   in lines  5 and 6 from the bottom of p.\ 307  in \cite{BBK2},
 the summations
should be taken over $G-B'$ instead of $B-B'$.
 With these corrections, the proof of \cite[Proposition 3.3]{BBK2} goes through.
\end{remark}

Finally, we prove that under $\NDL(\phi)$, $\J_{\phi, \le}$ is equivalent to $\UHK(\phi)$, which immediately yields that $\NDL(\phi)+ \UJS \Longleftrightarrow \PHI^+(\phi).$

\begin{proposition}\label{phi-phi}  Assume that $\VD$, \eqref{polycon} and $\RVD$ hold. Then, \begin{equation}\label{phi-1}\NDL(\phi)+ \J_{\phi,\le} \Longleftrightarrow \NDL(\phi)+\UHK(\phi)\end{equation} and so
\begin{equation}\label{e:4.13}
\NDL(\phi)+ \UJS \Longleftrightarrow \PHI^+ (\phi)  \Longleftrightarrow \PHI(\phi).
\end{equation}
\end{proposition}

\begin{proof}
First, note that the process $\{X_t\}$ is conservative due to
$\NDL(\phi)$ (see  Proposition  \ref{P:3.1}).
On the one hand, by  Theorem \ref{T:main-1},   $\UHK(\phi)$ implies $\J_{\phi, \le}$.
On the other hand, according to
Proposition \ref{pi-e}, under $\VD$, \eqref{polycon} and $\RVD$,
$\NDL(\phi)$ implies $\FK(\phi)$ and $\E_\phi$. In particular, the process $\{X_t\}$ possesses a heat kernel.
 Thus we have by \cite[Theorem 4.25] {CKW1}  that $\NDL(\phi)+\J_{\phi, \le}$ imply $\UHKD(\phi)$.
Furthermore, by Theorem \ref{T:main-1}, $\NDL(\phi)+ \J_{\phi,\le}$ imply $\UHK(\phi)$. This proves \eqref{phi-1}.

By Corollary \ref{ujs-ndl}, $\NDL(\phi)+ \UJS \Longrightarrow \J_{\phi,\le}$, which along with \eqref{phi-1} gives us
$$
\NDL(\phi)+ \UJS \Longleftrightarrow\UHK(\phi)+\NDL(\phi)+\UJS.
$$
It now
follows from Propositions \ref{ndlb} and \ref{ujs-jle}, and  Theorem \ref{phi-probab} that
$$ \PHI (\phi)\Longrightarrow   \NDL(\phi)+ \UJS \Longrightarrow \PHI^+ (\phi) .
$$
This establishes assertion \eqref{e:4.13} as $\PHI^+ (\phi) \Longrightarrow \PHI (\phi)$.
\qed \end{proof}

\subsection{$\PHI(\phi)  \Longleftrightarrow {\rm PHR}(\phi)+
\E_\phi +\UJS
\Longleftrightarrow {\rm EHR}+ \E_\phi+\UJS$ }\label{sec7-2}

The main contribution of this subsection is the
following relations among $\PHI(\phi)$, $\PHR(\phi)$  and $\EHR$,
which establish the equivalences
among (1), (5) and (6) of Theorem \ref{T:PHI}.

\begin{theorem}\label{prop-ehi}
 Assume that $\mu$ and $\phi$ satisfy $\VD$, $\RVD$ and \eqref{polycon} respectively. Then
$$
\PHI(\phi) \Longleftrightarrow {\rm PHR}(\phi)+
\E_\phi
+\UJS \Longleftrightarrow {\rm EHR}+ \E_\phi+\UJS.
$$
\end{theorem}

We start with the following key lemma.

\begin{lemma}\label{ehi-fk}  Under $\VD$ and \eqref{polycon},
{\rm EHR} and $\E_{\phi, \le}$ imply $\FK(\phi)$.\end{lemma}

\begin{proof} According to Remark \ref{R:phrehr}(ii), throughout this subsection we may and do assume that the constant $\eps=1/2$ in the definition of $\EHR$.

For any open subset ${D}$ of $M$, let $G^{D}$ be the associated Green operator. Recall that for any open set ${D}$,
it holds that
\begin{equation}\label{ppff-1}
\lambda_1({D})^{-1}\le\sup_{x\in D\cap M_0} \bE^x \tau_D = \sup_{x\in D\cap M_0} G^D{\bf 1}(x).
\end{equation}
For any ball $B=B(x,R)\subset M$ with $x\in M$ and $R>0$, and any open set ${D}\subset B$, we will verify that
\begin{equation}\label{ppff}
\sup_{x\in D\cap M_0}
\bE^x \tau_D \le c\phi(R)\left( \frac{\mu({D})}{V(x,R)}\right)^{\nu},
\end{equation}
where $c>0$ and $\nu \in (0, 1)$ are two constants independent of ${D}$ and $B$.   Once this is proved, $\FK(\phi)$ immediately follows from \eqref{ppff-1} and \eqref{ppff}.

 Fix an arbitrary $x_0\in {D}\cap M_0$.
  Let $R_k=2 \delta^kR$ for $k\ge0$, where $\delta\in(0,1/2]$ is a constant to be determined later.
  Set $B_k=B(x_0, R_k)$ for $k\ge 0$. Clearly $D\subset B_0=B(x_0, 2R)$.
Since $(G^{B_k} -G^{B_{k+1}}) {\bf 1}_D$ is a bounded non-negative function that is harmonic in $B_{k+1}$,
   we have by EHR and the $\mu$-symmetry of the Green operator $G^{B_k}$
   that  for any positive integers $n>k\geq 0$,
  \begin{equation}\label{e:fk1} \begin{split}
 & \sup_{y\in B_{n+1} \cap M_0}
(G^{B_k} -G^{B_{k+1}}) {\bf 1}_D(y) \\
   &\leq  \inf_{y\in B_{n+1} \cap M_0} (G^{B_k} -G^{B_{k+1}}) {\bf 1}_D(y) +
 c_1  \, \delta^{(n-k)\theta} \,
 \sup_{y\in B_{n+1} \cap M_0} |(G^{B_k} -G^{B_{k+1}}) {\bf 1}_D (y)| \\
 &\leq \frac{1}{\mu (B_{n+1}) }\int_{B_{n+1}} (G^{B_k} -G^{B_{k+1}}) {\bf 1}_D (y) \,\mu (dy)
    + c_1  \, \delta^{(n-k)\theta} \,
 \sup_{y\in B_{k} \cap M_0} | G^{B_k} {\bf 1}_D (y)| \\
  &\leq \frac{1}{\mu (B_{n+1}) }\int {\bf 1}_{B_k\cap D} (y) G^{B_k}{\bf 1}_{B_{n+1} } (y)   \, \mu (dy)
    + c_1  \, \delta^{(n-k)\theta} \,
 \sup_{y\in B_{k} \cap M_0} | G^{B_k} {\bf 1} (y)| \\
   &\leq \frac{1}{V(x_0, R_{n+1}) }\mu (D)  \| G^{B_k}{\bf1}\|_\infty
    + c_1  \, \delta^{(n-k)\theta} \,
  \sup_{y\in B_{k} \cap M_0} | G^{B_k} {\bf 1} (y)| ,
   \end{split}
\end{equation}
where $c_1=2^\theta c>0$, and $c$ and $\theta\in(0,1]$ are the constants in EHR. On the other hand, we have by $\E_{\phi, \le}$   that
\begin{equation}\label{e:fk2}
\sup_{y\in B_{k} \cap M_0} | G^{B_k} {\bf 1} (y)|   \leq
\sup_{y\in B_{k}\cap M_0} \bE^y \tau_{B(y, 2R_k)}
\leq c_2 \phi (2R_k).
\end{equation}
Taking $k=0$ and $n=1$ in \eqref{e:fk1} and $k=1$ in \eqref{e:fk2}, we find by \eqref{e:vd2} from VD and \eqref{polycon} that
 \begin{align}\label{e:7.7}
 \bE^{x_0} \tau_D
 &\leq \sup_{y\in B_{2} \cap M_0} G^{B_0}{\bf 1}_D (y)  \nonumber \\
& \leq   \sup_{y\in B_{2} \cap M_0} (G^{B_0} -G^{B_{1}}) {\bf 1}_D(y)  +
 \sup_{y\in B_{2} \cap M_0} G^{B_1} {\bf 1} (y) \nonumber \\
 &\leq  c_3   \left( \frac{\mu (D) }{V(x_0, R_2) }  + \delta^{\theta}\right)
 \phi (2R_0) + c_2 \phi (2R_1) \nonumber \\
 &\leq  c_4  \left( \frac{\mu (D) }{V(x_0, 2R) } \delta^{-2d_2} + \delta^{ \theta}\right)
\phi(R)      + c_4 \phi ( R ) \delta^{ \beta_1} \nonumber \\
&\leq  c_5 \phi (R) \left( \frac{\mu (D) }{V(x , R) } \delta^{-2d_2} + \delta^{ \theta \wedge \beta_1}  \right) .
 \end{align}
 Define $\nu= \frac{\theta\wedge \beta_1}{2d_2+  \theta\wedge \beta_1}$.
 If  $\frac{\mu (D) }{V(x, R) }\leq (1/2)^{2d_2+   \theta\wedge \beta_1}$,
 we take   $\delta = \left( \frac{\mu (D) }{V(x, R) }\right)^{1/(2d_2+ \theta \wedge \beta_1)}$, which is no larger than
 $1/2$,   in \eqref{e:7.7} to deduce
 $$  \bE^{x_0} \tau_D  \leq 2c_5  \phi (R) \left( \frac{\mu (D) }{V(x, R) }\right)^\nu.
 $$
 If $\frac{\mu (D) }{V(x, R) }> (1/2)^{2d_2+  \theta \wedge \beta_1}$, we get from $\E_{\phi,\le}$ that
 $$ \bE^{x_0} \tau_D  \leq c_6  \phi (R) \left( \frac{\mu (D) }{V(x, R) }\right)^\nu.
 $$
 Since $x_0\in D\cap M_0$ is arbitrary, this establishes \eqref{ppff} and hence completes the proof.     \qed
 \end{proof}

By $\VD$, \eqref{polycon} and
\cite[Proposition 7.3]{CKW1},
 $\FK(\phi)$ implies the existence of the Dirichlet
heat kernel $p^{D}(t,\cdot,\cdot)$ for any bounded open subset $D\subset M$, and that there is a constant $C_\nu>0$ such that
for every $x_0\in D$ and $t >0$
\begin{equation}\label{k-upper}
\esssup_{x,y\in D}p^{D}(t, x,y )\le \frac{C_\nu}{V(x_0,r)} \left( \frac{\phi(r)}{t}\right)^{1/\nu},
\end{equation}
where $r = {\rm diam}(D)$, the diameter of $D$.

\medskip

In the following, we will deduce $\NDL(\phi)$ from $\EHR$ and $\E_{\phi,\le}$, through establishing the space regularity of Dirichlet heat kernel. This basic approach is due to \cite[Lemma 3.8, Lemma 3.9 and Proposition 3.5]{HS}. The arguments below is also motivated by these in \cite[Subsections 5.3 and 5.4]{GT}.

\begin{lemma} \label{com-osc-sup} Assume that \eqref{polycon}, $\EHR$  and $\E_{\phi,\le}$ are satisfied. Let ${D}$ be a bounded
 open subset of $M$. Let $t>0$, $x\in {D}\setminus \mathcal{N}$ and $0<{r_1}< \phi^{-1}(t)$ such that
  $0<r_1\le r/2$ and  $B(x,r)\subset {D}$, where $r=\left( \phi^{-1}(t)^{\beta_1} r_1^\theta \right)^{1/(\beta_1+\theta)}$, $\beta_1$ is the constant in \eqref{polycon} and $\theta$ is the H\"older exponent in $\EHR$.
 Then,
$$
\essosc_{y\in B(x,{r_1})} p^{D}(t, x,y)\le C\left(\frac{{r_1}}{\phi^{-1}(t)}\right)^\kappa \esssup_{y\in D}\,  p^{D}(t/2, y,y),
$$
 where $\kappa=\beta_1\theta/(\beta_1+\theta)$, and $C$ is a constant depending on
the constants in \eqref{polycon} and
$\E_{\phi, \leq}$.
\end{lemma}

\begin{proof}
The proof uses some ideas from but is more direct than that of \cite[Lemma 5.10]{GT}.
For fixed $x\in {D}\setminus \mathcal{N}$ and $s>0$, set $u(s, y)=p^{D}(s, x,y)$.
According to Lemma \ref{ehi-fk} and \eqref{k-upper},  $$\int_D u(s, y)^2 \,\mu (dy)= p^D(2s, x, x)<\infty.$$ Since, by the symmetry of $p^D(t/2,z,x)=p^D(t/2, x,z)$,
$$
u(t, y)=\int_D p^{D} (t/2, y, z) p^D (t/2, z, x)\,\mu(dz)=P^D_{t/2} u(t/2, \cdot ) (y),
$$
we have $u(t, \cdot ) \in {\rm Dom}({\cal L}^D)\subset \sF^D$ for every $t>0$.
Thus for $\mu$-a.e.  $y\in D$,
\begin{align*}
\partial_t u(t, y) &= {\cal L}^D P^D_{t/2}  u(t/2, \cdot ) (y)= P^D_{t/2} {\cal L}^D u (t/2,  \cdot ) (y) \\
&= \int_D p^D (t/2, y, z) {\cal L}^D u (t/2,  \cdot ) (z)\, \mu (dz) = - \sE ( p^D(t/2, y, \cdot ),  u (t/2,   \cdot )).
\end{align*}
Hence, by the  Cauchy-Schwarz
inequality and the spectral representation,
\begin{align*}
| \partial_t u(t, y) |  & \leq    \sqrt{\sE ( p^D(t/2, y, \cdot ),  p^D(t/2, y, \cdot )) } \,
\sqrt{ \sE (u (t/2,   \cdot ),  u (t/2,   \cdot ) ) }  \\
&= \sqrt{\sE ( P^D_{t/4} p^D(t/4, y, \cdot ),   P^D_{t/4} p^D(t/4, y, \cdot ) ) } \,
\sqrt{\sE( P^D_{t/4} u (t/4,  \cdot ),   P^D_{t/4} u (t/4,   \cdot )) }  \\
&\leq \sqrt{ (2/t) \,  \| p^D(t/4, y, \cdot )\|^2_{L^2(D; \mu) }} \,
\sqrt{ (2/t)   \, \| u (t/4,  \cdot )\|^2_{L^2(D; \mu) }} \\
&=\frac{2}{t} \sqrt{p^D(t/2, y, y) p^D(t/2, x, x) } \leq \frac2t \esssup_{D\setminus {\cal N}} \,  p^D(t/2, y, y).
\end{align*}
In particular, by \eqref{k-upper}, $f(t, y):= \partial_t u(t, y)$ is a bounded function on $D$ for every $t>0$.
Note that $\lim_{s\to \infty} p^D(s, x, y) =0$ for every $y\in D\setminus {\cal N}$, also thanks to \eqref{k-upper}.
Then we have
\begin{align*}
u(t, y) &= - \int_t^\infty \partial_s p^D(s,x,  y) \,ds =-\int_0^\infty \partial_t p^D(t+r, x, y)\, dr \\
&= -\int_0^\infty  \int_D p^D(r, y, z) \partial_t p^D(t, x, z) \,\mu (dz) \, dr  =-G^D f(t, \cdot) (y).
\end{align*}
    Hence, by  EHR, Lemma \ref{green-inequality}   and $\E_{\phi,\le}$, for any $0<{r_1}\le r/2$,
\begin{align*}
\essosc_{B(x,{r_1})}u (t, \cdot) \le &2  \sup_{y\in B(x,r)\setminus {\cal N}} | f(t, y)| \sup_{y\in B(x, r)\setminus {\cal N}}
\bE^y\tau_{B(x,r)}   + c_1 \left(\frac{{r_1}}{r} \right)^\theta \sup_{y\in D\setminus {\cal N}} |u(t, y)|\\
\le &c_2 \left[ \phi(r) \frac{A}{t}+\left(\frac{{r_1}}{r} \right)^\theta A \right],
\end{align*}
where
$ A= \sup_{z\in D\backslash \mathcal{N}}p^{D}({t/2},z,z)$.
In the last inequality above, we also used the facts that $\sup_{y,z\in D\setminus {\cal N}}p^D(t,y,z)= \sup_{z\in D\setminus {\cal N}}p^D(t,z,z)$ and  $t\mapsto \sup_{z\in D\backslash \mathcal{N}}p^{D}({t},z,z)$ is a decreasing function, see e.g., the proof of Lemma \cite[Lemma 7.9]{CKW1}.

For any $0<r< \phi^{-1}(t)$, by \eqref{polycon},
$$
\frac{\phi(r)}{t} \le c_3\left( \frac{r}{\phi^{-1}(t)}\right)^{\beta_1}  ,
$$
 whence it follows that for any $0<{r_1}\le r/2$ and $0<r<\phi^{-1}(t)$,
 $$
  \essosc_{B(x,{r_1})}u\le C\left[\left( \frac{r}{\phi^{-1}(t)}\right)^{\beta_1}+ \left(\frac{{r_1}}{r} \right)^\theta\right]A.
 $$
  By choosing $r=\left( \phi^{-1}(t)^{\beta_1} r_1^\theta \right)^{1/(\beta_1+\theta)}$
  in the inequality above, we proved the desired assertion. \qed
\end{proof}
\vskip-0.2mm
\begin{lemma}\label{continuity}
Suppose that $\VD$, \eqref{polycon}, {\rm EHR}  and $\E_{\phi,\le}$ hold. Then  for any $x\in M_0$, $t>0$ and $0<{r}\le 2^{-(\beta_1+\theta)/\beta_1} \phi^{-1}(t)$ the following estimate holds
$$
|p^{B(x,\phi^{-1}(t))}(t, x,x)-p^{B(x,\phi^{-1}(t))}(t, x,y)|
\le \left( \frac{{r}}{\phi^{-1}(t) }\right)^\kappa \frac{C}{V(x,\phi^{-1}(t))},\quad y\in B(x,{r})\backslash \mathcal{N},
$$
where $\beta_1$ is the constant in \eqref{polycon}, $\theta$ is the H\"older exponent in {\rm EHR}, and $\kappa$ is the constant in Lemma $\ref{com-osc-sup}$.
\end{lemma}

\begin{proof}
 Fix $x\in M_0$ and $t,r_1>0$ with $0<{r}_1\le 2^{-(\beta_1+\theta)/\beta_1}  \phi^{-1}(t).$ We choose $r=\left( \phi^{-1}(t)^{\beta_1} {r}_1^\theta \right)^{1/(\beta_1+\theta)}$ as in Lemma \ref{com-osc-sup}. Then, $0<r_1\le r/2$. By applying Lemma \ref{com-osc-sup} with ${D}=B(x,\phi^{-1}(t))$, we get
$$
\essosc_{y\in B(x,{r}_1)}p^{B(x,\phi^{-1}(t))}(t, x,y)\le C\left( \frac{{r}_1}{\phi^{-1}(t) }\right)^\kappa\esssup_{y\in B(x,\phi^{-1}(t))}
p^{B(x,\phi^{-1}(t))}(t/2, y,y).
$$ This along with \eqref{k-upper} yields the desired assertion. \qed
 \end{proof}

Having all the lemmas at hand, we can obtain the following result.

\begin{proposition} \label{phi-ephi-2}
Let $\VD$, \eqref{polycon}, {\rm EHR} and  $\E_\phi$ be satisfied. Then for any open subset ${D}\subset M$, the semigroup $\{P_t^{D}\}$ possesses the heat kernel $p^{D}(t, x,y)$, and moreover $\NDL(\phi)$ holds true. \end{proposition}

\begin{proof}
 The existence of heat kernel $p^{D}(t,x,y)$ associated with the semigroup $\{P_t^{D}\}$ for any open subset ${D}\subset M$ has been stated in the remark below Lemma \ref{ehi-fk}, and so we only need to verify $\NDL(\phi)$.

According to $\E_\phi$ and \cite[Lemma 4.17]{CKW1},
there are constants $\varepsilon\in(0,1)$ and $\delta\in (0,1/2)$ such that for all $x\in M_0$ and for any $t,r>0$ with $t\le \delta \phi(r)$, $\bP^x(\tau_{B(x,r)}\le t)\le \varepsilon.$
In the following, let $B=B(x,r)$ and $0<t\le \delta \phi(r)$. Then for any $x\in B\backslash \mathcal{N}$,
since the process $\{X_t\}$ has no killings inside $M$,
 $$\int_{B}p^B(t, x,y)\,\mu(dy)=\bP^x(\tau_B>t)\ge 1- \varepsilon.$$
Therefore,
$$
p^B(2t, x,x)=\int_{B}p^B(t, x,y)^2\,\mu(dy)\ge \frac{1}{\mu(B)}\left(\int_B p^B(t, x,y)\,\mu(dy)\right)^2\ge \frac{c_1}{V(x,r)}.
$$
In particular, taking $r=\phi^{-1}(t/\delta) >0$ in the inequality above, we arrive at
$$
p^{B(x, \phi^{-1}(t/(2\delta)))}(t, x,x)\ge \frac{c_2}{V(x,\phi^{-1}(t))}.
$$
Furthermore, according to Lemma \ref{continuity}, $\VD$ and \eqref{polycon}, there exists a constant $c_3>0$ such that for any $0<{r}\le 2^{-(\beta_1+\theta)/\beta_1}\phi^{-1}(t/(2\delta))$, we have
$$
|p^{B(x,\phi^{-1}(t/(2\delta)))}(t, x,x)-p^{B(x,\phi^{-1}(t/(2\delta)))}(t, x,y)|
\le \left( \frac{{r}}{\phi^{-1}(t) }\right)^\kappa \frac{c_3}{V(x,\phi^{-1}(t))},\quad y\in B(x,{r})\backslash \mathcal{N},
$$
 where $\beta_1$ is the constant in \eqref{polycon}, $\theta$ is the H\"older exponent in {\rm EHR}, and $\kappa$ is the constant in Lemma $\ref{com-osc-sup}$.

Combining with both inequalities above and choosing $\eta\in (0,1)$ small enough such that $\eta^\kappa c_3\le \frac{1}{2}c_2$ and $\eta \phi^{-1}(t) \le 2^{-(\beta_1+\theta)/\beta_1}\phi^{-1}(t/(2\delta))$ for all $t>0$, one can get that for any $x\in M_0$ and $y\in B(x,\eta \phi^{-1}(t))\backslash \mathcal{N}$,
 \begin{align*}
 &p^{B(x,\phi^{-1}(t/(2\delta)))}(t, x,y) \\
 &\ge  p^{B(x,\phi^{-1}(t/(2\delta)))}(t, x,x)-|p^{B(x,\phi^{-1}(t/(2\delta)))}(t, x,x)-p^{B(x,\phi^{-1}(t/(2\delta)))}(t, x,y)|\\
& \ge \frac{c_2}{2V(x,\phi^{-1}(t))}.
\end{align*}
That is, thanks to $\VD$ and \eqref{polycon} again, there are constants $c_i>0$ $(i=4,5,6)$ such that $0<2c_4\le c_5$ and for any $x\in M_0$ and $y\in B(x,2c_4\phi^{-1}(t))\backslash \mathcal{N}$,
$$ p^{B(x,c_5\phi^{-1}(t))}(t, x,y) \ge  \frac{c_6}{V(x,\phi^{-1}(t))}.$$

Now, for any $x_0\in M$ and $r,t>0$ such that $(c_4+c_5)\phi^{-1}(t)\le r$, we have
$B(x,c_5\phi^{-1}(t))\subset B(x_0,r)$ for all $x\in B(x_0,c_4\phi^{-1}(t))$, and so
$$ p^{B(x_0,r)}(t, x,y) \ge p^{B(x,c_5\phi^{-1}(t))}(t, x,y)\ge \frac{c_6}{V(x,\phi^{-1}(t))},\quad x,y\in B(x_0,c_4 \phi^{-1}(t))\backslash \mathcal{N}.$$
This proves that $\NDL(\phi)$ holds true with $\eps=c_4\wedge \frac{1}{c_4+c_5}$.
 \qed\end{proof}

Note that by
Proposition \ref{phi-ephi-2} and Proposition \ref{P:3.1},
 ${\rm EHR} +\E_\phi$ imply the conservativeness of the process
$X=\{X_t; t\geq 0\}$
(see Proposition  \ref{P:3.1}).

\medskip

Next, we present the proof of Theorem \ref{prop-ehi}.

{\medskip\noindent {\bf Proof of Theorem \ref{prop-ehi}. }}
  That
$\PHI(\phi)\Longrightarrow \NDL(\phi)+ \E_\phi+\UJS+\J_{\phi,\le}$
has been established in Subsection \ref{phi-section-1}, where $\RVD$ is used.
Since $\NDL +
 \E_{\phi, \leq}
+\J_{\phi,\le} \Longrightarrow {\rm PHR}(\phi)$ by Proposition \ref{ndl-holder}, we have
$\PHI(\phi)$ implies ${\rm PHR}(\phi)+ \E_\phi+\UJS$.

On the other hand, by  Proposition \ref{phi-ephi-2} and \eqref{e:4.13}
 (where $\RVD$ is used too),
we have
$$ {\rm EHR}+ \E_\phi + \UJS\Longrightarrow \NDL(\phi)+ \UJS\Longleftrightarrow \PHI(\phi).
$$
This completes the proof of the theorem.
\qed

\subsection{$\PI(\phi)+\J_{\phi,\le}+ \CSJ(\phi)+\UJS\Longleftrightarrow \PHI(\phi)$}\label{sec7-3}
In this subsection, we will prove the above mentioned equivalence in
Theorem \ref{T:PHI}.
Note that, under $\VD$, \eqref{polycon} and $\RVD$,
$\PHI(\phi)$
$\Longrightarrow\PI(\phi)+\J_{\phi,\le}+ \CSJ(\phi)+\UJS$
has already been proved
by combining
the results in Subsection \ref{phi-section-1},
Propositions \ref{pi-e} and Theorem \ref{T:main-1}.
So all we need is to prove the following theorem.

\begin{theorem}\label{P:gg2}
Assume that $\mu$ and $\phi$ satisfy $\VD$, $\RVD$  and \eqref{polycon} respectively. Then
$$\PI(\phi)+\J_{\phi,\le}+ \CSJ(\phi)+\UJS\Longrightarrow \PHI(\phi).$$
 \end{theorem}

 First of all, note that $\PI(\phi)+\J_{\phi,\le}+ \CSJ(\phi)$ imply the conservativeness of the process. Indeed,
$\PI(\phi)+\RVD $ imply $\FK(\phi)$  by Proposition \ref{pi-e-pre},  and
$\FK(\phi)+\J_{\phi,\le}+ \CSJ(\phi)$ imply $\E_\phi$
by Proposition \ref{P:exit}. Furthermore,
$\J_{\phi,\le}+\E_\phi$ imply the conservativeness of the process
(see \cite[Lemma 4.21]{CKW1}).

To prove the theorem, we begin with the following logarithmic lemma, which plays the key role in the proof of H\"{o}lder continuity of harmonic functions. The proof below is motivated by that of \cite[Lemma 1.3]{CKP1}.

\begin{proposition}\label{L:log-l} Let $B_r=B(x_0,r)$ for some $x_0\in M$ and $r>0$.
Assume that $u\in \sF_{B_R}^{loc}$ is a bounded and superharmonic function in a ball
$B_R$ such that $u\ge 0$ on $B_R$. If $\VD$, \eqref{polycon}, $\CSJ(\phi)$ and $\J_{\phi,
\le}$ hold, then for any $l>0$ and $0<2r\le R$,
$$\int_{B_r\times B_r}\left[\log  \Big(\frac{u(x)+l}{u(y)+l} \Big)\right]^2\, J(dx,dy)\le \frac{c_1V(x_0,r)}{\phi(r)}\bigg(1+\frac{\phi(r)}{\phi(R)}\frac{\T\,(u_-; x_0,R)}{l}\bigg),$$
 where $\T\,(u_-;x_0,R)$ is the nonlocal tail of $u_-$ in $B(x_0,R)$ defined by \eqref{def-T}, and  $c_1$ is a constant independent of $u$, $x_0$, $r$, $R$ and $l$.\end{proposition}

\begin{proof}
 According to $\CSJ(\phi)$, $\J_{\phi, \le}$ and  \cite[Proposition 2.3(5)]{CKW1},
we can choose $\varphi\in \sF_{B_{3r/2}}$ related
to $\mbox{Cap} (B_r,B_{3r/2})$ such that
\begin{equation}\label{cap-estimate-1}
\sE(\varphi,\varphi)\le 2  \mbox{Cap} (B_r,B_{3r/2}) \le
\frac{c_1V(x_0,r)}{\phi(r)}.
 \end{equation}
Since $u$ is  a bounded and superharmonic function in a ball $B_R$
and $\frac{\varphi^2}{u+l}\in \sF_{B_{3r/2}}$ for any $l>0$,
we have by Theorem \ref{equ-har} that
\begin{align*}0\le &\sE\big(u,\frac{\varphi^2}{u+l}\big)\\
=&\int_{B_{2r}\times B_{2r}} (u(x)-u(y))\Big(
\frac{\varphi^2(x)}{u(x)+l}-
\frac{\varphi^2(y)}{u(y)+l}\Big)\,J(dx,dy)\\
&+2\int_{B_{2r}\times B_{2r}^c} (u(x)-u(y))
 \frac{\varphi^2(x)}{u(x)+l}\,J(dx,dy)\\
=&\int_{B_{2r}\times B_{2r}} \Big((u(x)+l)-(u(y)+l)\Big)\Big(
\frac{\varphi^2(x)}{u(x)+l}-
\frac{\varphi^2(y)}{u(y)+l}\Big)\,J(dx,dy)\\
&+2\int_{B_{2r}\times B_{2r}^c} (u(x)-u(y))
 \frac{\varphi^2(x)}{u(x)+l}\,J(dx,dy)\\
=&\int_{B_{2r}\times B_{2r}}
\varphi(x)\varphi(y)\bigg(\frac{\varphi(y)}{\varphi(x)}+\frac{\varphi(x)}{\varphi(y)}
-\frac{\varphi(x)(u(y)+l)}{\varphi(y)(u(x)+l)}- \frac{\varphi(y)(u(x)+l)}{\varphi(x)(u(y)+l)}\bigg)\,J(dx,dy)\\
&+2\int_{B_{2r}\times B_{2r}^c} (u(x)-u(y))
 \frac{\varphi^2(x)}{u(x)+l}\,J(dx,dy)\\
=&:I_1+I_2.\end{align*}

Applying the inequality
$$\frac{a}{b}+ \frac{b}{a} -2 = (a-b)(b^{-1}-a^{-1})\ge (\log a- \log b)^2,\quad a,b>0$$ with $a=\frac{u(y)+l}{\varphi(y)}$ and $b=\frac{u(x)+l}{\varphi(x)}$, we find that
 \begin{align*}
 &\frac{\varphi(x)(u(y)+l)}{\varphi(y)(u(x)+l)} + \frac{\varphi(y)(u(x)+l)}{\varphi(x)(u(y)+l)}- \frac{\varphi(y)}{\varphi(x)} - \frac{\varphi(x)}{\varphi(y)}\\
 &\ge \left(   \log\frac{u(y)+l}{\varphi(y)}- \log \frac{u(x)+l}{\varphi(x)} \right)^2- \bigg(\frac{\varphi(y)}{\varphi(x)}+\frac{\varphi(x)}{\varphi(y)}-2    \bigg),
\end{align*} and so
\begin{align*}
 I_1\le &-\int_{B_{2r}\times B_{2r}} \varphi(x)\varphi(y) \left(   \log\frac{u(y)+l}{\varphi(y)}- \log \frac{u(x)+l}{\varphi(x)} \right)^2 \,J(dx,dy)\\
 &+\int_{B_{2r}\times B_{2r}} (\varphi(x)-\varphi(y))^2\,J(dx,dy).
\end{align*}

On the other hand, due to the fact that $u\ge0$ on $B_R$, for all
$x\in B_{2r}$ and $y\in B_R\setminus B_{2r}$,
$$\frac{u(x)-u(y)}{u(x)+l}\le 1;$$ while for all $x\in B_{2r}$ and
$y\in B_R^c$,
$$\frac{u(x)-u(y)}{u(x)+l}\le \frac{(u(x)-u(y))_+}{u(x)+l}\le
\frac{u(x)+u_-(y)}{u(x)+l}\le 1+l^{-1} u_-(y).$$ Therefore,
$$I_2\le 2\int_{B_{2r}\times B_{2r}^c}\varphi^2(x)\,J(dx,dy)+2l^{-1}\int_{B_{2r}\times
 B_R^c}u_-(y)\varphi^2(x)\,J(dx,dy).$$

Combining all the estimates above and the fact that $\varphi=1$ on $B_r$, we obtain
\begin{align*}&\int_{B_r\times B_r} \left[\log  \bigg(\frac{u(x)+l}{u(y)+l}\bigg) \right]^2
\,J(dx,dy)\\
 &\le \int_{B_{2r}\times B_{2r}} \varphi(x)\varphi(y) \left(
\log\frac{u(y)+l}{\varphi(y)}- \log \frac{u(x)+l}{\varphi(x)}
\right)^2
\,J(dx,dy)\\
&\le \int_{B_{2r}\times B_{2r}}
(\varphi(x)-\varphi(y))^2\,J(dx,dy)+
2\int_{B_{2r}\times B_{2r}^c}\varphi^2(x)\,J(dx,dy)\\
&\quad +2l^{-1}\int_{B_{2r}\times B_{R}^c}
 u_-(y)\varphi^2(x)\,J(dx,dy)\\
&\le \sE(\varphi,\varphi)+
\frac{c_2V(x_0,r)}{\phi(R)l}\T\, (u_-; x_0,R),\end{align*}
where the last inequality follows from $\J_{\phi,\le}$ and the fact that for any $x\in B_{3r/2}$ and $y\in B_R^c$ with $R\ge 2r$,
$$\frac{V(x_0,d(x_0,y))\phi(d(x_0,y))}{V(x,d(x,y))\phi(d(x,y))}\le c'\left(1+\frac{d(x_0,x)}{d(x,y)}\right)^{\beta_2+\alpha_2}\le c'\left(1+\frac{3r/2}{R-3r/2}\right)^{\beta_2+\alpha_2}\le c'',$$
thanks to $\VD$ and \eqref{polycon}. Hence, the desired
assertion follows from the inequality and \eqref{cap-estimate-1}. \qed
\end{proof}

For the diffusion case, Proposition \ref{L:log-l} was originally due to Moser.
In that case, one can use the Leibniz rule, but for the jump case some more care is required.
See \cite[Corollary 7.7]{KZ} for a related inequality.
In the following we give another proof that is more robust.

\begin{proof}{\bf (Another proof of Proposition \ref{L:log-l})}
For a function $v$ on $M$ and for fixed $x,y\in M$, write
\[
\bar v(t)=\bar v_{xy}(t):=tv(x)+(1-t)v(y),\quad~~~~t\in [0,1].
\]
Take $\varphi\in \sF_{B_{3r/2}}$ as in
\eqref{cap-estimate-1} in the previous proof.
For any $x,y\in M$ and $l>0$, it holds that
\begin{align*}
(u(x)&-u(y))\left[\vp(x)^2/(u(x)+l)-\vp(y)^2/(u(y)+l)\right]\\=&
\int_0^1\left[\frac d{dt}\frac{\bar \vp^2}{(\bar u+l)}(s)\right]\frac d{dt}(\bar u(s)+l)\,ds\\
=&\int_0^1\frac{2\bar \vp(s)\frac d{dt}\bar \vp(s)}{(\bar u(s)+l)}\frac d{dt}(\bar u(s)+l)\,ds
-\int_0^1\left[\frac{\bar \vp}{(\bar u+l)}(s)\right]^2\left[\frac d{dt}(\bar u(s)+l)\right]^2\,ds\\
=&\int_0^12\left[\bar\vp(s)\frac d{dt}\bar \vp(s)\right]\left[\frac d{dt}\log (\bar u(s)+l)\right]\,ds
-\int_0^1\bar\vp(s)^2\left[\frac d{dt}\log (\bar u(s)+l)\right]^2\,ds.
\end{align*}
Multiplying $J(x,y)$
and integrating over $B_{2r}\times B_{2r}$ w.r.t. $\mu\times \mu$ in both
sides of the equality above, we have
\begin{equation}\begin{split}
&\int_{B_{2r}\times B_{2r}}\int_0^1\bar\vp(s)^2\left[\frac d{dt}\log (\bar
u(s)+l)\right]^2\,ds\,J(dx,dy)\\
&\qquad+{\sE}(u,\vp^2/(u+l))-2\int_{B_{2r}\times B_{2r}^c} (u(x)-u(y))
 \frac{\varphi^2(x)}{u(x)+l}\,J(dx,dy)\\
&= 2\int_{B_{2r}\times B_{2r}}\int_0^1\left[\bar\vp(s)\frac d{dt}\bar
\vp(s)\right]\left[\frac d{dt}\log (\bar u(s)+l)\right]\,ds\,
J(dx,dy)\\
&\le 2\left[\int_{B_{2r}\times B_{2r}}\int_0^1\bar\vp(s)^2\Big(\frac d{dt}\log
(\bar u(s)+l)\Big)^2\,ds\,J(dx,dy)
\right]^{1/2}\\
&\quad\times  \left[\int_{B_{2r}\times B_{2r}}\int_0^1\Big(\frac d{dt}\bar \vp(s)\Big)^2\,ds\,
J(dx,dy)\right]^{1/2}\\
&\le2\left[\int_{B_{2r}\times B_{2r}}\int_0^1\bar\vp(s)^2\Big(\frac d{dt}\log
(\bar u(s)+l)\Big)^2\,ds\,J(dx,dy)
\right]^{1/2}\sE(\vp,\vp)^{1/2}.\label{eq:niirviw}
\end{split}\end{equation}

In the following, we set $$K:=\int_{B_{2r}\times B_{2r}}\int_0^1\bar\vp(s)^2\Big(\frac d{dt}\log
(\bar u(s)+l)\Big)^2\,ds\,J(dx,dy).$$
Now, as in the previous proof,
\[2\,\bigg|\int_{B_{2r}\times B_{2r}^c} (u(x)-u(y))
 \frac{\varphi^2(x)}{u(x)+l}\,J(dx,dy)\bigg|=|I_2|\le \sE(\vp,\vp)+
\frac{c_2V(x_0,r)}{\phi(R)l}\T\, (u_-; x_0,R).
\]
Further, noting that $u+l$ is bounded and  superharmonic  on $2B$, we have by Theorem \ref{equ-har} that
$\sE(u,\varphi^2/(u+l))\ge 0$. Plugging these and \eqref{cap-estimate-1} into \eqref{eq:niirviw}, we have
\[
K-\frac{c_1V(x_0,r)}{\phi(r)}
-\frac{c_2V(x_0,r)}{\phi(R)l}\T\, (u_-; x_0,R)\le 2K^{1/2}\left(\frac{c_1V(x_0,r)}{\phi(r)}\right)^{1/2}.
\]
We thus obtain
\[
K \le \frac{c_3V(x_0,r)}{\phi(r)}\bigg[1+\frac{\phi(r)}{\phi(R)}\frac{\T\,(u_-; x_0,R)}{l}\bigg].
\]

On the other hand, since $\vp=1$ on $B_r$, using the Cauchy-Schwarz inequality we have
\begin{align*}
K &\ge \int_{B_{2r}\times B_{2r}}(\vp(x)^2\wedge \vp(y)^2)
\int_0^1\left[\frac d{dt}\log (\bar u(s)+l)\right]^2\,ds\,J(dx,dy)\\
&\ge \int_{B_{2r}\times B_{2r}}(\vp(x)^2\wedge \vp(y)^2)
\left[\int_0^1\frac d{dt}\log (\bar u(s)+l)\,ds\right]^2\,J(dx,dy)\\
&\ge \int_{B_r\times B_r}\left[\log (u(y)+l)-\log (u(x)+l)\right]^2\, J(dx,dy).
\end{align*}
We therefore prove the desired inequality, by combining all the inequalities above.
\qed\end{proof}

As a consequence of Proposition \ref{L:log-l}, we have the following
statement.

\begin{corollary}\label{C:PIandlog} Let $B_r=B(x_0,r)$ for some $x_0\in M$ and $r>0$.
Assume that $u\in \sF_{B_R}^{loc}$ is a bounded and
superharmonic function in a ball
$B_R$ such that $u\ge 0$ on $B_R$.  For any $a,l>0$ and $b>1$,
define
$$v= \Big[\log \Big(\frac{a+l}{u+l}\Big)\Big]_+ \wedge \log b.$$ If $\VD$, \eqref{polycon}, $\CSJ(\phi)$, $\J_{\phi, \le}$ and $\PI(\phi)$ hold, then for
any $l>0$ and $0<2 \kappa r\le R$,
$$\frac{1}{V(x_0,r)}\int_{B_r} (v-\ol v_{B_r})^2 d\mu\le  c_1\bigg(1+ \frac{\phi(r)}{\phi(R)} \frac{\T\,(u_-; x_0,R)}{l}\bigg),$$
 where $\kappa\ge 1$ is the constant in $\PI(\phi)$, $\ol v_{B_r}= \frac{1}{\mu(B_r)}\int_{B_r}v\,d\mu$ and $c_1$ is a constant independent of $u$, $x_0$, $r$, $R$ and $l$.
\end{corollary}
\begin{proof}By $\PI(\phi)$ and \eqref{polycon}, we have
$$\int_{B_r} (v-\ol v_{B_r})^2 d\mu\le c_2\phi(r)\int_{B_{\kappa r}\times B_{\kappa r}} (v(x)-v(y))^2\,J(dx,dy).$$ Observing that $v$ is a truncation of the sum of a constant and $\log(u+l)$,
$$\int_{B_{\kappa r}\times B_{\kappa r}} (v(x)-v(y))^2\,J(dx,dy)\le\int_{B_{\kappa r}\times B_{\kappa r}}\left(\log  \Big(\frac{u(x)+l}{u(y)+l} \Big)\right)^2\, J(dx,dy).$$ Hence, it suffices to apply Proposition \ref{L:log-l}
to conclude the assertion.\qed \end{proof}

\begin{proposition}\label{P:ho} Let $B_r=B(x_0,r)$ for some $x_0\in M$ and $r>0$.
Assume that $u\in \sF_{B_R}^{loc}$ is a bounded and
harmonic function in a ball
$B_R$. If $\VD$, $\RVD$, \eqref{polycon}, $\CSJ(\phi)$, $\J_{\phi, \le}$ and $\PI(\phi)$ hold, there are constants $\gamma\in(0,\beta_1)$ and $c>0$ such that
\begin{equation}\label{e:hocon}\essosc_{B_{r'}}u\le c\Big(\frac{r'}{r}\Big)^\gamma\left[\left(\frac{1}{V(x_0,2r)}\int_{B(x_0,{2r})} u^2\,d\mu \right)^{1/2} +\T\, (u; x_0,r) \right],\end{equation} where $0<r'\le r< R/2.$ In particular, suppose that $\VD$, $\RVD$ and \eqref{polycon} hold, then we have
$$ \PI(\phi)+\J_{\phi,\le}+ \CSJ(\phi)\Longrightarrow  {\rm EHR}.$$
\end{proposition}

\begin{proof} (i) First, by $\J_{\phi,\le}$ and Lemma \ref{intelem}, it is easy to see that
$$\T\, (u; x_0,r)\le c'\|u\|_\infty,\quad r>0.$$ Thus,
assuming \eqref{e:hocon},
there is a constant $c''>0$ such that for all $0<r<R/2,$
$$\essosc_{B_{r}}u\le c''\Big(\frac{r}{R}\Big)^\gamma \|u\|_\infty.$$ From this, we can easily see that,
 once \eqref{e:hocon} is proved, EHR is yielded.

(ii) In the following, we mainly prove \eqref{e:hocon}. We begin with the argument of \cite[Theorem 1.2]{CKP1}. Before starting, let us fix some notations. For any $j\ge0$ and $0<2r<R$, let
$r_j=r\sigma^j$ and $B_j=B_{r_j}$, where
$\sigma\in(0,1/(4\kappa)]$ and $\kappa\ge1$ is the constant in $\PI(\phi)$. Let us define $$
w(r_0)=w(r)=2C_0\Bigg[\left(\frac{1}{V(x_0,2r)}\int_{B(x_0,{2r})}
u^2\,d\mu \right)^{1/2}+ \T\, (u; x_0, r)\Bigg]$$ with the constant
$C_0$ given in \eqref{e:mvi2g-1} of Proposition \ref{P:mvi2g}, and
$$w(r_j)=\Big(\frac{r_j}{r_0}\Big)^\gamma w(r_0)$$ for some
$\gamma\in(0,\beta_1).$ In order to prove the required assertion, it will suffice to verify that
\begin{equation}\label{e:ho}\essosc_{B_j}u\le w(r_j),\quad j\ge 0.\end{equation}
Indeed, for any $0<r'\le r$, we can choose $j\ge 0$ such that $r_{j+1}<r'\le r_j$. Then, by \eqref{e:ho}, we have
$$\essosc_{B_{r'}}u\le \essosc_{B_j} u\le w(r_j)\le \sigma^\gamma \Big(\frac{r_{j+1}}{r}\Big)^\gamma w(r)\le \sigma^\gamma \Big(\frac{r'}{r}\Big)^\gamma w(r).$$ Thus, the required assertion holds with $c=2C_0\sigma^\gamma.$

(iii) We will prove \eqref{e:ho} by induction. For this,  note that $\PI(\phi)+\RVD $ imply $\FK(\phi)$  by Proposition \ref{pi-e-pre}. Then, according to the definition
of $w(r_0)$ and Proposition \ref{P:mvi2g}, \eqref{e:ho}
holds for $j=0$, since both the functions $u_+$ and $u_-$
bounded subharmonic in $B_R$.

Now, we make an induction assumption and assume that \eqref{e:ho} is valid
for all $0\le i\le j$ for some $j\ge0$, and then we prove it holds
also for $j+1$. We have that either
\begin{equation}\label{e:ho01}\frac{\mu(2B_{j+1}\cap \{u\ge \essinf_{B_j}
u+w(r_j)/2\})}{\mu(2B_{j+1})}
\ge \frac{1}{2},\end{equation} or
\begin{equation}\label{e:ho02}\frac{\mu(2B_{j+1}\cap \{u\le \essinf_{B_j}
u+w(r_j)/2\})}{\mu(2B_{j+1})}
\ge \frac{1}{2}\end{equation} must
hold. If \eqref{e:ho01} holds, we set $u_j:=u-\essinf_{B_j}u$, and if
\eqref{e:ho02} holds, we set $u_j:=w(r_j)-(u-\essinf_{B_j}u)$. In both
cases we have $u_j\ge 0$ on $B_j$ and \begin{equation}\label{e:ho021}\frac{\mu(2B_{j+1}\cap
\{u_j\ge w(r_j)/2\})}{\mu(2B_{j+1})}
\ge \frac{1}{2}\end{equation} holds.
Clearly,  $u_j$ is bounded and harmonic in $B_R$ satisfying that
\begin{equation}\label{e:ho022}\begin{split}\esssup_{B_i}|u_j|\le& w(r_j)+\esssup_{B_i}|u-\essinf_{B_j}u|\\
\le & w(r_i)+\esssup_{B_i}|u-\essinf_{B_i}u|+ |\essinf_{B_i}u-\essinf_{B_j}u|\\
\le& 2 w(r_i)+\esssup_{B_i}u-\essinf_{B_i}u\\
\le& 3w(r_i),\quad 0\le i\le j.\end{split}\end{equation}

We now claim that under the induction assumption we have
\begin{equation}\label{e:ho03} \T\, (u_j; x_0, r_j)\le c_0 \sigma^{-\gamma} w(r_j),\end{equation}
where $c_0>0$ is independent of $u$, $x_0$, $r$ and $\sigma$. Indeed,
we have
\begin{align*} \T\, (u_j ; x_0, r_j)=&\phi(r_j)\sum_{i=1}^j\int_{B_{i-1}\setminus
B_i}\frac{|u_j(x)|}{V(x_0,d(x_0,x))\phi(d(x_0,x))}\,\mu(dx)\\
&+\phi(r_j)\int_{B_0^c}\frac{|u_j(x)|}{V(x_0,d(x_0,x))\phi(d(x_0,x))}\,\mu(dx)\\
\le &
\phi(r_j)\sum_{i=1}^j\esssup_{B_{i-1}}|u_j|\int_{B_i^c}\frac{1}{V(x_0,d(x_0,x))\phi(d(x_0,x))}\,\mu(dx)\\
&+ \phi(r_j)\int_{B_0^c}\frac{|u_j(x)|}{V(x_0,d(x_0,x))\phi(d(x_0,x))}\,\mu(dx)\\
\le& c_1\sum_{i=1}^j \frac{\phi(r_j)}{\phi(r_i)}w(r_{i-1}),
\end{align*}where in the last inequality we have used \eqref{e:ho022}, Lemma \ref{intelem}, $$|u_j|\le w(r_0)+\esssup_{B_0} |u|+|u|,\quad \quad j\ge0$$ and
\begin{align*}&\int_{B_0^c}\frac{|u_j(x)|}{V(x_0,d(x_0,x))\phi(d(x_0,x))}\,\mu(dx)\\
&\le
c'\bigg[\frac{1}{\phi(r_0)}\big(\esssup_{B_0}|u|+w(r_0)\big)+\int_{B_0^c}\frac{|u(x)|}{V(x_0,d(x_0,x))\phi(d(x_0,x))}\,\mu(dx)\bigg]\\
&\le
c''\frac{w(r_0)}{\phi(r_0)}\le c''\frac{w(r_0)}{\phi({r_1})} .\end{align*} Note that, in the second inequality above we used the fact that $$\esssup_{B_0}|u|\le \esssup_{B_0}u^++\esssup_{B_0}u^-\le w(r_0)$$ deduced from Proposition \ref{P:mvi2g}. Estimating further, we have
\begin{align*} \sum_{i=1}^j \frac{\phi(r_j)}{\phi(r_i)}w(r_{i-1})&=
w(r_0)\Big(\frac{r_j}{r_0}\Big)^\gamma\sum_{i=1}^j\frac{\phi(r_j)}{\phi(r_i)}\Big(\frac{r_{i-1}}{r_j}\Big)^\gamma\\
&\le
c_2w(r_0)\Big(\frac{r_j}{r_0}\Big)^\gamma\sum_{i=1}^j\Big(\frac{r_j}{r_i}\Big)^{\beta_1}\Big(\frac{r_{i-1}}{r_j}\Big)^\gamma\\
&=c_2w(r_0)\Big(\frac{r_j}{r_0}\Big)^\gamma\sum_{i=1}^j\Big(\frac{r_{i-1}}{r_i}\Big)^\gamma\Big(\frac{r_j}{r_i}\Big)^{\beta_1-\gamma}\\
&\le \frac{c_2\sigma^{-\gamma}}{1-\sigma^{\beta_1-\gamma}} w(r_j)\le
c_3\sigma^{-\gamma} w(r_j),
\end{align*} where we used \eqref{polycon} in the first inequality, and used $\sigma\in(0,1/(4\kappa)]$ and $\beta_1>\gamma$ in the second inequality. Hence, \eqref{e:ho03} is proved with $c_0$
independent of $\sigma$.

Next, consider the function $v$ defined as follows
$$
v:= \Big[\log\Big(\frac{w(r_j)/2+l}{u_j+l}\Big)\Big]_+
\wedge k,\quad k,l>0.
$$
Using the fact $\sigma\in(0,1/(4\kappa)]$ again and applying Corollary \ref{C:PIandlog}, we get
$$\frac{1}{\mu(2B_{j+1})}\int_{2B_{j+1}} (v-\ol v_{2B_{j+1}})^2\,d\mu\le c_4\left(1+l^{-1}\frac{\phi(r_{j+1})}{\phi(r_j)} \T\, (u_j; x_0,
r_j)\right).$$This, along with \eqref{e:ho03} and \eqref{polycon}, yields that
$$\frac{1}{\mu(2B_{j+1})}\int_{2B_{j+1}} (v-\ol
v_{2B_{j+1}})^2\,d\mu\le c_5 \left(1+l^{-1}
\sigma^{\beta_1-\gamma}w(r_j)\right).$$Hence, choosing
$l=\eps w(r_j)$ with $\eps=\sigma^{\beta_1-\gamma}$, we get that
\begin{equation}\label{e:ho04}\frac{1}{\mu(2B_{j+1})}\int_{2B_{j+1}} (v-\ol
v_{2B_{j+1}})^2\,d\mu\le c_6.\end{equation}

To continue, denote in short $\tilde{B}=2B_{j+1}$. We obtain from \eqref{e:ho021} that
\begin{align*} k&=\frac{1}{\mu(\tilde{B}\cap\{u_j\ge
w(r_j)/2\})}\int_{\tilde{B}\cap\{u_j\ge w(r_j)/2\}}k\,d\mu\\
&=\frac{1}{\mu(\tilde{B}\cap\{u_j\ge
w(r_j)/2\})}\int_{\tilde{B}\cap\{v=0\}}k\,d\mu\\
&\le \frac{2}{\mu(\tilde{B})}\int_{\tilde{B}}(k-v)\,d\mu=2(k-\ol
v_{\tilde{B}}).
\end{align*}By integrating the preceding inequality over the set
$\tilde{B}\cap \{v=k\}$, we further obtain
\[
\frac{\mu(\tilde{B}\cap \{v=k\})}{\mu(\tilde{B})}k \le
\frac{2}{\mu(\tilde{B})} \int_{\tilde{B}\cap\{v=k\}}(k-\ol
v_{\tilde{B}})\,d\mu\le \frac{2}{\mu(\tilde{B})}
\int_{\tilde{B}}|v-\ol v_{\tilde{B}}|\,d\mu\le c_7,\]
where \eqref{e:ho04} and the Cauchy-Schwarz inequality are used in the last inequality. Let us take
$$k=\log\left(\frac{w(r_j)/2+\eps w(r_j)}{3\eps
w(r_j)}\right)=\log\left(\frac{
\frac{1}{2}+\eps}{3\eps}\right)\approx\log\Big(\frac{1}{\eps}\Big),$$
and so we have
\be\label{e:ho05}\frac{\mu(\tilde{B}\cap \{u_j\le 2\eps
w(r_j)\})}{\mu(\tilde{B})}\le \frac{c_7}{k}\le \frac{c_8}{-\log
\sigma}.\ee

(iv) We are now in a position to start a suitable iteration to deduce the
desired oscillation reduction. From here we make essential changes of the argument in the proof of \cite[Theorem 1.2]{CKP1}. Note that, in the setting of \cite{CKP1} the proof is heavily based on the fractional Poincar\'{e} inequalities (see \cite[(5.11)]{CKP1}), which however are not available in the present situation. To deal with this difficulty, we apply Lemma \ref{oppo} instead. In the following, we fix $j\ge0$. First, for any $i\ge0$, we define
$$\varrho_i=(1+2^{-i})r_{j+1},
\quad B^i=B_{\varrho_i}$$ and set
$$k_i=(1+2^{-i})\eps w(r_j),\quad w_i=(k_i-u_j)_+,\quad A_i=\frac{\mu(B^i\cap \{u_j\le k_i\})}{\mu(B^i)}.$$
Then, we have by $\VD$ and Lemma \ref{oppo} that
\begin{align*} A_{i+2}(k_{i+1}-k_{i+2})^2&=\frac{1}{\mu(B^{i+2})}\int_{B^{i+2}\cap \{u_j\le k_{i+2}\}} (k_{i+1}-k_{i+2})^2\,d\mu\\
&\le \frac{1}{\mu(B^{i+2})}\int_{B^{i+2}} w_{i+1}^2\,d\mu\\
&\le \frac{c_8}{(k_i-k_{i+1})^{2\nu}}\left( \frac{1}{\mu(B^{i+1})}\int_{B^{i+1}} w_{i}^2\,d\mu\right)^{1+\nu}\left(\frac{\varrho_{i+2}}{\varrho_{i+1}-\varrho_{i+2}}\right)^{\beta_2}\\
&\quad \times \left[1+\frac{1}{k_i-k_{i+1}}\left(\frac{\varrho_{i+2}}{\varrho_{i+1}-\varrho_{i+2}}\right)^{d_2+
\beta_2-\beta_1} \T\,(w_i; x_0, \varrho_{i+1})\right]\\
&\le \frac{c_9}{[(2^{-i}-2^{-i-1})\eps w(r_j)]^{2\nu}}\left[(\eps w(r_j))^2 A_i\right]^{1+\nu}\left(\frac{1}{2^{-i}-2^{-i-1}}\right)^{\beta_2}\\
&\quad \times \left[1+\frac{1}{(2^{-i}-2^{-i-1})\eps w(r_j)}\left(\frac{1}{2^{-i}-2^{-i-1}}\right)^{d_2+
\beta_2-\beta_1} \T\,(w_i; x_0, r_{j+1})\right]\\
&\le c_{10} [\eps w(r_j)]^2 A_i^{1+\nu} 2^{(1+2\nu+d_2+2\beta_2-\beta_1)i}\left(1+\frac{ 1}{\eps w(r_j)}\T\,(w_i; x_0, r_{j+1})\right),
\end{align*} where $\nu$ is the constant in $\FK(\phi)$, and in the third inequality we have used the facts that $u_j\ge 0$ on $B^{i+1}\subset B_j$ and
$$\int_{B^{i+1}}w_i^2\,d\mu\le k_i^2 \mu(B^{i+1}\cap \{w_i\ge0\})\le c'(\eps w(r_j))^2 \mu(B^i\cap\{u_j\le k_i\}).$$
Hence,
$$A_{i+2}\le c_{11} A_i^{1+\nu} 2^{(3+2\nu+d_2+2\beta_2-\beta_1)i}\left(1+\frac{ 1}{\eps w(r_j)}\T\,(w_i; x_0, r_{j+1})\right).$$ Note that, by the facts that $u_j\ge0$, $w_i\le 2\eps w(r_j)$ on $B_j$ and $|w_i|\le |u_j|+2\eps w(r_j)$ on $M$,
\begin{align*}\T\,(w_i; x_0, r_{j+1})&=
\phi(r_{j+1})\int_{B_j\setminus B_{j+1}} \frac{|w_i(x)|}{V(x_0,d(x_0,x))\phi(d(x_0,x))}\,\mu(dx)\\
&\quad +\frac{\phi(r_{j+1})}{\phi(r_j)}\T\,(w_i;x_0,r_j)\\
&\le c_{12}\bigg(2\eps w(r_j)\phi(r_{j+1})\int_{B_{j+1}^c}
\frac{\mu(dx)}{V(x_0,d(x_0,x))\phi(d(x_0,x))}\\
&\qquad\quad  + 2\eps w(r_j)\phi(r_{j+1})\int_{B_{j}^c}
\frac{\mu(dx)}{V(x_0,d(x_0,x))\phi(d(x_0,x))}\\
&\qquad\quad + \frac{\phi(r_{j+1})}{\phi(r_j)}\T\,(u_j;x_0,r_j)\bigg)\\
&\le c_{12}\left(\eps w(r_{j})+ \sigma^{\beta_1} \T\,(u_j; x_0, r_j)\right)\\
&\le c_{13}\left(1+ \frac{\sigma^{\beta_1-\gamma}}{\eps}\right)\eps w(r_{j})\le 2c_{13}\eps w(r_{j}),
\end{align*} where the second and the third inequalities follow from Lemma \ref{intelem} and \eqref{e:ho03},
respectively, and the last inequality is due to $\eps=\sigma^{\beta_1-\gamma}$.
Combining with all the conclusions above, we arrive at
$$A_{i+2}\le c_{14} A_i^{1+\nu} 2^{(3+2\nu+d_2+2\beta_2-\beta_1)i}.$$
Let $c^*=c_{14}^{-1/\nu}2^{-(3+2\nu+d_2+2\beta_2-\beta_1)/\nu^2}$ and choose the constant $\sigma\in\Big(0,\frac{1}{4} \wedge \exp^{-\big(\frac{c_8}{c^*}\big)}\Big).$ Then, by \eqref{e:ho05},  $$A_0\le c^*= c_{14}^{-1/\nu}2^{-(3+2\nu+d_2+2\beta_2-\beta_1)/\nu^2}.$$ According to Lemma \ref{L:it}, we can deduce that $\lim_{i\to\infty}A_i=0$. Therefore, $u_j\ge \eps w(r_j)$ on $B_{{j+1}}$, and then we can find that
$$ \essosc_{B_{j+1}}u=\esssup_{B_{j+1}}u_j-\essinf_{B_{j+1}}u_j\le (1-\eps)w(r_j)=(1-\eps)\sigma^{-\gamma} w(r_{j+1}),$$ where the inequality above follows from the fact that $\esssup_{B_{j+1}}u_j\le w(r_j)$, since under \eqref{e:ho01} $$\esssup_{B_{j+1}}u_j=\esssup_{B_{j+1}}u-\essinf_{B_j}u\le \esssup_{B_j}u-\essinf_{B_j}u\le w(r_j),$$ or under \eqref{e:ho02} $$\esssup_{B_{j+1}}u_j=w(r_j)-\essinf_{B_{j+1}}(u-\essinf_{B_j}u)\le w(r_j).$$
Taking finally $\gamma\in(0,\beta_1)$ small enough such that
$\sigma^\gamma\ge 1-\eps=1-\sigma^{\beta_1-\gamma},$ we obtain that
$$\essosc_{B_{j+1}} u\le w(r_{j+1})$$ holds, proving the induction step and finishing
the proof of \eqref{e:ho}. \qed
\end{proof}

We are now in a position to present the proof of the main theorem in this subsection.

\medskip

 \noindent {\bf Proof of Theorem \ref{P:gg2}.}
By Theorem \ref{prop-ehi},
it suffices to prove that
$$
 \PI(\phi)+\J_{\phi,\le}+ \CSJ(\phi)+\UJS\Longrightarrow {\rm EHR}+\E_\phi+\UJS.
$$
As
mentioned in the remark below Theorem \ref{P:gg2}, under $\VD$, $\RVD$ and \eqref{polycon},
$$ \PI(\phi)+\J_{\phi,\le}+ \CSJ(\phi)\Longrightarrow   \E_\phi.$$
On the other hand, according
to Proposition \ref{P:ho} (where $\RVD$ is used again),
$$ \PI(\phi)+\J_{\phi,\le}+ \CSJ(\phi)\Longrightarrow  {\rm EHR}.$$  The proof is complete.
\qed

\begin{remark}\rm By the proof above and Propositions \ref{phi-ephi-2} and \ref{pi-e}, under $\VD$, $\RVD$ and \eqref{polycon}, we have the following relations without using $\UJS$:
$$ \PI(\phi)+\J_{\phi,\le}+ \CSJ(\phi)\Longrightarrow  {\rm EHR} +\E_\phi\Longrightarrow \NDL(\phi)\Longrightarrow\PI(\phi)+\E_\phi.$$
\end{remark}

\medskip \noindent
{\bf Proof of Corollary \ref{C:1.25}.}   Assume $\PHI (\phi)$ and
$\J_{\phi, \geq }$ are satisfied. Then by Theorem
\ref{T:PHI}(4), $\J_\phi$ and $\CSJ (\phi)$ hold. So by
Theorem \ref{T:main}(4), $\HK (\phi)$ also holds.

Conversely, assume $\HK (\phi )$ holds. By Theorem \ref{T:main},
$\J_\phi$ and $\CSJ (\phi )$ are satisfied. Note that $\UJS$ holds
trivially because of $\J_\phi$. Thus by Theorem \ref{T:PHI}
again, $\PHI(\phi)$ holds. \qed

\section{Applications and Examples}\label{Sectin-ex}

The stability results in Theorem \ref{T:PHI}  allow us to obtain PHI for a large class of symmetric
jump processes using ``transferring method"; that is, by first establishing PHI for
a particular symmetric jump process with jumping kernel $J(x, y)$, we can then use
Theorem \ref{T:PHI} to obtain PHI for other symmetric jump processes whose jumping kernels
are comparable to $J(x, y)$. Examples are given in \cite[Section 6.1]{CKW1} on fractals
that support anomalous diffusions with two-sided heat kernel estimates.
The subordination of these diffusion processes enjoy $\HK (\phi)$ and hence $\PHI (\phi)$  by Corollary \ref{C:1.25}, and so can
be served as the base examples.
For readers' convenience,
we give one concrete example here on the Sierpinski gasket.

\begin{example}\label{E:5-1SG}{\bf (Subordinations of diffusions on fractal-like manifolds.)}\quad \rm
We first define the $2$-dimensional Sierpinski gasket and Brownian motion on it.
Let $a_1=(0,0), a_2=(1,0), a_3=(1/2,\sqrt 3/2)$, and set
$F_i(x)=(x-a_i)/2+a_i$ for $i=1,2,3$.
Then, there exists unique non-void compact set such that
$K=\cup_{i=1}^3F_i(K)$; we call $K$ the 2-dimensional
Sierpinski gasket. Let $V_0=\{a_1,a_2,a_3\}$ and set
\[V_k:=\bigcup_{1\le i_1,\cdots,i_k\le 3}F_{i_1}\circ\cdots\circ F_{i_k}
(V_0),~~~\hat K_{{\rm pre}}:=\bigcup_{k\ge 0}2^kV_k~~\mbox{and}~~\hat K
:=\bigcup_{k\ge 0}2^kK.\]
$\hat K_{{\rm pre}}$ is called a pre-gasket, and $\hat K$ is called an unbounded gasket.
Let $d(\cdot,\cdot)$ be the geodesic distance on $\hat K$ (which is comparable to the Euclidean metric) and let $\mu$ be the
(normalized) Hausdorff measure on $\hat K$ with respect to $d$.
Brownian motion has been constructed on $\hat K$ and it has been proved in \cite{BP} that
its heat kernel $\{q(t,x,y): t>0, x,y\in \hat K\}$ enjoys the following estimates for
all $t>0, x,y\in \hat K$:
\begin{align}
c_1t^{-{d_f}/{d_w}}\exp \Big(-c_2 \Big(\frac {d(x,y)^{d_w}}{t}\Big)^{\frac 1{d_w-1}}\Big)
\le q(t,x,y)
\le c_3t^{-{d_f}/{d_w}}\exp \Big(-c_4 \Big(\frac {d(x,y)^{d_w}}{t}\Big)^{\frac 1{d_w-1}}\Big),
\label{eq:aron}\end{align}
where $d_f=\log 3/\log 2$ is the Hausdorff dimension and $d_w=\log 5/\log 2$ is called the walk dimension.

We next consider a $2$-dimensional Riemannian manifold (a fractal-like manifold) $M$, whose global structure is
like that of the fractal. It can be constructed from $\hat K_{{\rm pre}}$ by changing each bond
to a cylinder and smoothing the connection to make it a manifold.
One can naturally construct a Brownian motion on the surfaces of cylinders. Using the stability of sub-Gaussian heat kernel estimates (see for instance \cite{BBK1} for details), one can show that any divergence operator ${\mathcal L}=\sum_{i,j=1}^2\frac{\partial}
{\partial x_i}(a_{ij}(x)\frac{\partial}{\partial x_j})$
in local coordinates on such manifolds that satisfies
the uniform elliptic condition obeys the following heat kernel estimates for
all $t>0, x,y\in M$:

\begin{align}
\frac{c_1}{V(x,\Psi^{-1}(t))}\exp \Big(-c_{2}\Big(\frac{\Psi(d(x,y))}t\Big)^{\gamma_1}\Big)
&\le  q(t,x,y)\label{frachk2}\\
&\le
\frac{c_3}{V(x,\Psi^{-1}(t))}\exp \Big(-c_{4}\Big(\frac{\Psi(d(x,y))}t\Big)^{\gamma_2}\Big),
\nonumber\end{align}
where $V(x,r)\asymp r^2\wedge r^{d_f}$ for all $x\in M$, $\Psi(s)=s^2\vee s^{d_w}$,
and $\gamma_1,\gamma_2>0$ are some constants.

We now subordinate the diffusion $\{Z_t\}$ whose heat kernel enjoys
\eqref{frachk2}. Let $\{\xi_t\}$ be a subordinator that is
independent of $\{Z_t\}$; namely, it is an increasing L\'evy process
on $\bR_+$. Let $\bar \phi$ be the Laplace exponent of the
subordinator, i.e. $$\bE[\exp (-\lambda \xi_t)]=\exp
(-t\bar\phi(\lambda)),\quad \lambda, t>0.$$
In this example, for simplicity we consider the case
$\bar\phi(t)=t^{\alpha_1/2}+t^{\alpha_2/2}$ for some $0<\alpha_1\le
\alpha_2<2$, in which case, $\{\xi_t\}$ is a sum of independent
$\alpha_1/2$- and $\alpha_2/2$-subordinators.
The process $\{X_t\}$ defined by $X_t=Z_{\xi_t}$ for any $t\ge0$ is called a subordinate process. Define
\begin{equation}\label{bieibfis}
\phi(r)=\frac 1{\bar\phi(1/\Psi(r))}.
\end{equation}
It is easy to see $\phi$ satisfies \eqref{polycon}.
As discussed in \cite[Section 6.1]{CKW1},
the heat kernel for $\{X_t\}$ satisfies $\HK(\phi)$ (hence $\PHI (\phi)$ as well) with
\begin{equation} \label{e:6.9}
\phi(r)=r^{ \alpha_2}1_{\{r\le 1\}}+r^{ \alpha_1 d_w/2}1_{\{r>1\}},
\end{equation}
which is (up to constant multiplicative) the same as
\eqref{bieibfis}. Note that $ \alpha_1 d_w/2>2$ when $\alpha_1$ is close to $2$.

It follows from our stability theorem for heat kernels, Theorem \ref{T:main}, that
 for any symmetric pure jump process on the above
mentioned space whose jumping kernel enjoys $\J_\phi$ with $\phi$  given by
\eqref{e:6.9},  its heat kernel enjoys the estimates
$\HK(\phi)$, hence $\PHI (\phi)$
holds for these processes.
\end{example}

\medskip

The following example is taken from \cite{CKi},
which shows that $\PHI$ holds for the trace of Brownian motion on Sierpinski gasket on one side of the big triangle, by using the characterization of
$\PHI$ from the main result of this paper, Theorem \ref{T:PHI}.

\begin{example}\label{E:5.2}{\bf (Trace of Brownian motion on
the Sierpinski gasket.)} \rm
\quad
Let $K$ be the two-dimensional
Sierpinski
gasket obtained from the unit triangle with
vertices $a_1=(0, 0)$, $a_2=(1, 0)$ and $a_3=(1/2, \sqrt{3}/2)$
as in Example \ref{E:5-1SG}. It is known that there is a Brownian motion $X$ on $K$.
Let $Y=\{Y_t; t\geq 0\}$ be
the trace process of $X$ on the line segment $I$ connecting $a_1$ and $a_2$.  That is, $Y_t=X_{\tau_t}$ for $t\geq 0$,
where $\tau_t=\inf\{s>0: A_s>t\}$ and $A_t$ is the
 positive continuous additive functional  of $X$ whose Revuz measure $\mu$ is the one-dimensional Lebesgue measure restricted to $I$.
 The trace process $Y$ is an $\mu$-symmetric pure jump
process on $I$  with jumping measure $J(dx, dy)=J(x, y) \,\mu (dx) \,\mu (dy)$ (cf. \cite{CF, FOT}).
For convenience, we identify
the line segment $I$ with  the unit interval $[0, 1]$.
Denote the Dirichlet form of $Y$ on $L^2(I; \mu)$ by $(\sE, \sF)$. Then the domain of the Dirichlet form $\sF$ is the same as that of the symmetric $\alpha$-stable process
on $I$, where $\alpha = \log(10/3)/\log 2 \in (1, 2)$,
and their corresponding $\sE_1$-energies (i.e. $\sE_1(u, u)$) are comparable.
Kigami \cite{Kig2}   computed the jumping kernel $J(x, y)$.
From which, it is easy to deduce that there
is a constant $c_1 >0$ so that
\begin{equation}\label{e:5.5}
J(x, y) \leq c_1 |x-y|^{-(1+\alpha)} \quad \hbox{for all } x, y\in [0, 1].
\end{equation}
However, $J(x, y)$   vanishes on some open subset of $[0, 1]^2$
and is  comparable to $|x-y|^{-(1+\alpha)}$
only on a proper subset $U$ of $[0, 1]^2$. Let $g_1 (x, y)=(x/2, y/2)$ and $g_2 (x, y) = ((1+x)/2, (1+y)/2)$.
Define
  \begin{equation}\label{e:5.6}
  D_1=\{(x, y) \in [0, 1]^2:  |x-y|\geq 1/2\}, \ \
    D_{k+1} = g_1 (D_k) \cup g_2 (D_k) \hbox{ for }
    k\geq 1 \mbox{ and}  \ \
  U= \cup_{k\geq 1} D_k.
  \end{equation}
Then there is a constant $c_2>0$ so that
\begin{equation}\label{e:5.7}
J(x, y) \geq c_2 |x-y|^{-(1+\alpha)} \quad \hbox{for all } (x, y) \in U.
\end{equation}
Define $\phi (r)=r^\alpha$. Then $\J_{\phi, \leq}$ holds for  the trace process $Y$ in view of \eqref{e:5.5}.

For $x\in I=[0, 1]$, define $U_x=\{y\in I: (x, y)\in U\}$.
Then from the definition of $U$ in \eqref{e:5.6}
(details are given in \cite{CKi}),
 it is easy to see that
there exist constants $c_3, c_4 \in (0, 1)$ such that
for every $x\in I$ and $0<r\leq 1$,
\begin{equation}\label{e:5.8}
\mu(U_x \cap A(x, c_3r, r)) \ge c_4\mu(A(x, c_3r, r)).
\end{equation}
Here for $0<r_1<r_2$,
$ A(x, r_1, r_2) := B(x, r_2) \setminus B(x, r_1)$.
Thus by \eqref{e:5.7}, for every  $B_r=B(x, r)$ with $x\in I$ and $r\in (0, 1]$ and every
$f\in \sF_{B_r}$,
\begin{align*}
\int_{B_r} (f-\bar f_{B_r})^2 \,d\mu
 &\leq C_1 \phi (r) \int_{B_r\times B_r } (f(x)-f(y))^2 \frac{1}{|x-y|^{1+\alpha}} \,dx \,dy \\
&\leq   C_2 \phi (r) \int_{B_r\times B_r } (f(x)-f(y))^2 J(x, y)\, dx \,dy,
\end{align*}
where the first inequality is due to the Poincar\'e inequality for symmetric $\alpha$-stable process on $I$, and the second inequality
is due to \eqref{e:5.7}, \eqref{e:5.8} and an argument similar to that of
\cite[Theorem 1.1]{BKS}. Hence the finite range version of $\PI (\phi)$ holds for $Y$.
By Remark \ref{R:csjrem}, $\SCSJ(\phi)$ and hence $\CSJ (\phi)$ automatically holds since $\alpha <2$.
It is easy to see that $\UJS$  holds as well in view of \eqref{e:5.5} and \eqref{e:5.7}.
Therefore by Theorem  \ref{T:PHI}
and Remark \ref{rem:boun}, the finite range version of $\PHI (\phi)$ holds for $Y$; that is,
$\PHI (\phi)$ holds
for non-negative caloric functions of $Y$ in any cylinder
 with $r\leq 1$.

Since the jumping kernel $J(x, y)$ vanishes on some open subsets
of $[0, 1]^2$,
it does not satisfy $\J_\phi$ condition. Hence $Y$ does not have two-sided heat kernel estimates
$\HK(\phi)$. However, by \eqref{e:5.5} and
\cite[the proof of Theorem 1.2 and Remark 4.4]{CK2},
we can show that the transition density function $p(t, x, y)$ of $Y$ with respect to the Lebesgue measure $\mu$ on $I$ has the following upper bound estimate: there is a constant $c_5>0$
so that
$$
p(t, x, y) \leq c_5 \left( t^{-1/d} \wedge \frac{t}{|x-y|^{1+\alpha}} \right)
\quad \hbox{for all } (t, x, y)\in (0, 1]\times I \times I.
$$
That is, $\UHK (\phi)$ holds for $Y$ over any bounded time interval.
Although the corresponding lower bound estimate fails for $Y$, a main result of
\cite{CKi}
asserts that the corresponding lower bound holds for $Y$ over the subset $U$ of $I^2$;
that is, there is a constant $c_6>0$ so that
$$
p(t, x, y) \geq c_6 \left( t^{-1/d} \wedge \frac{t}{|x-y|^{1+\alpha}} \right)
\quad \hbox{for all } (t, x, y)\in (0, 1]\times U.
$$
\end{example}

\medskip

In the remainder of this section, we give some more
details for Examples \ref{E:1.2}-\ref{E:1.3}, and
show that some conditions
in the equivalence statements of Theorem \ref{T:PHI} are necessary through
two more examples.

 \medskip

\noindent {\bf Example \ref{E:1.2}} (continued):
Here we provide some more details for this example.
Clearly $\J_{\phi,\le}$ and $\UJS$ hold.
Since $\PI (\phi)$ holds for rotationally symmetric stable process on $\bR^d$ with
$\phi (r)=r^\alpha$,
it follows from \cite[Example 3]{DK} or
\cite[Theorem 1.1]{BKS}
that $\PI (\phi)$ holds for the symmetric non-local Dirichlet form with jumping kernel $J(x, y)$.
By Remark \ref{R:csjrem}, $\SCSJ(\phi)$ and hence $\CSJ (\phi)$ holds.
So we have $\PHI (\phi)$ by Theorem \ref{T:PHI}.
However, since $\J_{\phi,\ge}$ does not hold,
$\HK(\phi)$ does not hold either in view of Theorem \ref{T:main}.

\bigskip

\noindent {\bf Example \ref{E:1.3}} (continued):
We now provide some more details for this example.
Clearly $\J_{\phi,\le}$ holds.
We show below in Proposition \ref{usj-cone}  that $\UJS$ holds
when the domain parameter $r_\theta$ in the increment condition of $\xi (x)$ is sufficiently
small.
On the other hand, by \cite[Theorem 1.1]{BKS},
there is a constant $C_0 >0$ so that for every ball
$B\subset \bR^d$, every $f\in L^2(B; dx)$ and every $\beta \in [\alpha_1, 2)$
$$
  \int_{B\times B} \frac{(f(x)-f(y))^2}{|x-y|^{d+\beta}}\,   \left( {\bf 1}_{\Gamma_\theta (x)}(y)+ {\bf 1}_{\Gamma_\theta (y)}(x)\right)  \,dx\, dy
\geq  C_0 \int_{B\times B} \frac{(f(x)-f(y))^2}{|x-y|^{d+\beta}}\, dx \,dy.
$$
Integrating in $\beta$ with respect to the probability measure $\nu$ over $[\alpha_1, \alpha_2]$  yields
\begin{align*}
 &\int_{B\times B} (f(x)-f(y))^2 J(x, y) \,dx \,dy \\
 &\geq C^{-1}  \int_{\alpha_1}^{\alpha_2}
 \int_{B\times B} \frac{(f(x)-f(y))^2}{|x-y|^{d+\beta}}\,   \left( {\bf 1}_{\Gamma_\theta (x)}(y)+ {\bf 1}_{\Gamma_\theta (y)}(x)\right)  \,dx\, dy\, \nu (d \beta) \\
 &\geq  C^{-1} C_0 \int_{B\times B} \frac{(f(x)-f(y))^2} {|x-y|^d \phi (|x-y|) }\, dx \,dy,
\end{align*}
where $C\geq 1$ is the constant in \eqref{e:1.5}. The other direction of the inequality
$$
\int_{B\times B} (f(x)-f(y))^2 J(x, y)\, dx \,dy \leq C  \int_{B\times B} \frac{(f(x)-f(y))^2} {|x-y|^d \phi (|x-y|) }\, dx \,dy
$$
follows directly from \eqref{e:1.5}. It has been established in \cite{CK2}  that the symmetric Markov process on $\bR^d$ with jumping kernel $\frac1{|x-y|^d \phi (|x-y|)}$
has two-sided heat kernel estimates $\HK (\phi)$ and so $\PI (\phi)$ holds for this process in view of  Corollary \ref{C:1.25} and Theorem \ref{T:PHI}.
Hence we deduce from the above inequalities  that $\PI (\phi)$ holds for the symmetric non-local Dirichlet form with jumping kernel
$J(x,y)$.
By Remark \ref{R:csjrem},
$\SCSJ(\phi)$ and hence $\CSJ (\phi)$ hold.
Therefore $\PHI(\phi)$ holds by Theorem \ref{T:PHI}.
Moreover, by Theorem \ref{T:PHI}, $\PHR (\phi)$, $\EHR$ as well as $\E_\phi$ hold.

\begin{prop}\label{usj-cone}
$\UJS$ holds for Example $\ref{E:1.3}$ when the parameter $r_\theta  > 0$ in \eqref{eq:xixisin} is sufficiently small.
\end{prop}
\begin{proof}
Case 1, $x\in \Xi(y)$: In this case, either $x\in \Gamma_\theta (y)$ or
$x\in \ol{B(y, 1)}$. In each case, it is easy to verify that for $0<r\leq d(x, y)/2$,
\[
|\{z\in B(x,r): z\in \Xi(y)\}|\ge c_{\theta,1}r^2
\]
with some $c_{\theta,1}>0$, hence $\UJS$ holds.

Case 2, $y\in \Xi (x)$ and $d(x,y)\le 1$: In this case, $y\in \ol{B(x, 1)}$,
and it is once again easy to verify  that for $0<r\leq d(x, y)/2$,
\begin{equation}\label{eq:GPQs}
|\{z\in B(x,r): y\in \Xi(z)\}|\ge c_{\theta,2}r^2
\end{equation}
 with some $c_{\theta,2}>0$, hence $\UJS$ holds.

Case 3, $y\in \Xi (x)$ and $d(x,y)> 1$: In this case, $y\in \Gamma_\theta (x)$.
 We further divide it into two cases. Recall that $c_\theta >0$ is the constant so that
$\xi (x+c_\theta)
 = \xi (x) + 2\pi$
for $x\in \bR$.

 i) When $r \in (0,  c_\theta]$:
Let $s\in (0,r\wedge r_\theta]$, and $y'$ be either
     $y-(s,0)$ or $y+(s,0)$.
   In  each case the angle $\angle yxy'$ is at most $\sin^{-1}s$. Hence  by  assumption
   \eqref{eq:xixisin} and translation,
 we have either $y\in \Gamma_\theta (x+(s,0))$ for all $s\le r\wedge r_\theta$ or
 $y\in \Gamma_\theta (x-(s,0))$ for all $s\le r\wedge r_\theta$.
 Suppose the former holds. (One can argue similarly
if the latter holds.) Then, because
$v(x)$ depends only on the first coordinate of $x$, we have either
$y\in \Gamma_\theta (x+(s,s'))$ for all $s'\le r\wedge r_\theta$ or $y\in \Gamma_\theta (x+(s,-s'))$ for all $s'\le r\wedge r_\theta$. Suppose the former holds. (Again, we can argue similarly if the latter holds.) Then we have
\begin{equation}\label{e:5.10}
y\in \Gamma_\theta (z)\quad \mbox{ for all }~z\in D:=\{x+(s,s'): s,s'\in (0,r\wedge r_\theta]\}.
\end{equation}
Hence $|D|\asymp (r\wedge r_\theta)^2$ so that \eqref{eq:GPQs} holds, which implies $\UJS$.

ii) When $r> c_\theta$:
Let
\[
H_j:=\{z=(z_1,z_2):
z_1=x_1+jc_\theta,\, d(z,y)\ge d(x,y)\}\cap \Gamma_\theta (x;y^c),\quad j\in \bZ,
\]
where $\Gamma_\theta (x;y^c)$ is the connected component of $\Gamma_\theta (x)\setminus \{x\}$
that does not contain $y$.
Depending on whether $\Gamma_\theta (x)$ contains $\{z=(z_1,z_2): z_1=x_1\}$ or not
and depending on the angle of the cone, it holds that either the length of $H_j\cap B(x,r)$ is
of order $([r/c_\theta]-|j|)$ or $|j|$ for either
$j=0,1,\cdots, [r/c_\theta]$ or $j=0,-1,\cdots, -[r/c_\theta]$. Among the four cases, let us discuss
the first case (the other cases can be discussed similarly).
Since $\Gamma_\theta (z)$ is a translation of $\Gamma_\theta (x)$
(because of the periodicity of $v(\cdot)$), it holds that
$y\in \Gamma_\theta (z)$ for all $z\in H_j$.
By the conclusion \eqref{e:5.10}
of Case 3 i) with $r=c_\theta$,
for any $j=0,1,\cdots, [r/c_\theta]$,
there exists
a rectangle $G_j$ with width $c_\theta \wedge r_\theta$ from Case 3 i) and
having $ H_j\cap B(x,r)$ as one of its  vertical side
such that $y\in \Gamma_\theta (z')$ for all $z'\in G_j$.
Hence, $|G_j| \geq (c_\theta \wedge r_\theta)([r/c_\theta]-j)$, and so
$|\cup_{j=1}^{[r/c_\theta]}G_j|\ge
\sum_{j=1}^{[r/c_\theta]} (c_\theta \wedge r_\theta)([r/c_\theta]-j)
\geq c_{\theta,3} r^2$ for some $c_{\theta,3}>0$.
Thus  \eqref{eq:GPQs} holds with $c_{\theta,2}>0$ being replaced by a different constant $c_{\theta,4}>0$, which implies $\UJS$.
  \qed
\end{proof}

\medskip

\begin{example}\label{E:82-3}{\bf ($\EHI$ and $\E_\phi$ do not imply
$\PHI(\phi)$.)}\quad \rm Let $M=\bR^2$ and $1<\alpha <2$. Consider a
symmetric L\'evy process $X=\{X_t\}$ on $\bR^2$ with the L\'evy
measure  of the form
$$ \nu(dx) =h(x)\,dx:= |x|^{-2-\alpha} f(x/|x|)\,dx , $$
where $f: \sS^{1} \to \bR_+$ is bounded and symmetric.
Then, it is proved
in \cite[Corollary 13]{BS} that $\EHI$ holds for non-negative harmonic functions.
In fact, \cite[Theorem 1]{BS} gives more general fact in $\bR^d$ setting with $d\ge 1$ that $\EHI$ holds for non-negative harmonic functions on $B(0,1)$ if
and only if there is a constant $C>0$ such that the following holds
$$
\int_{B(y,1/2)}|y-v|^{\alpha-d}h(v)\,dv\le C\int_{B(y,1/2)}h(v)\,dv,\quad |y|>1.$$

Let us take a particular choice of $f$ given as follows. For $i\in \bN$, let $\theta_i=(3\pi/8)4^{-i}$ and
$\theta_i'=(3\pi/8)2^{-i}$. Note that $\sum_{i=1}^\infty(\theta_i+\theta_i')=\pi/2$. Define
$$H=\Big\{e^{\theta\sqrt {-1}}, e^{-\theta\sqrt {-1}}, -e^{\theta\sqrt {-1}}, -e^{-\theta\sqrt {-1}}: \theta\in A\Big\},$$ where \[
A=[0,\theta_1)\cup\Big(\bigcup_{n=1}^\infty\Big[\sum_{i=1}^n(\theta_i+\theta_i'),\sum_{i=1}^n(\theta_i+\theta_i')+\theta_{n+1}\Big)\Big).
\] Set $f(x)={\bf1}_H(x)$. Then, writing $\xi_n=\sum_{i=1}^n(\theta_i+\theta_i')+\theta_{n+1}/2$ and $J(x,y)=h(x-y)$,
we see that $$J(e^{\xi_n\sqrt {-1}},0)=1.$$
Setting $H_n=\{e^{\theta\sqrt{-1}}:\theta\in [\xi_n-\theta_{n+1}/2,\xi_n+\theta_{n+1}/2)\}$, we have for large $n$,
\begin{align*}
V(e^{\xi_n\sqrt
{-1}},2^{-n-1})^{-1}\int_{B(e^{\xi_n\sqrt
{-1}},2^{-n-1})}J(z,0)\,dz&\le c(2^{n+1})^2
\int_{B(e^{\xi_n\sqrt{-1}},2^{-n-1})}{\bf1}_{H_n}(z/|z|)\,dz\\
&\le c'4^n2^{-n-1}4^{-n-1}\le c_02^{-n},
\end{align*}
so $\UJS$ does not hold. Therefore, by
Theorem \ref{T:PHI}, $\PHI(\phi)$
can not
hold in this case.

We will briefly explain why $\E_\phi$ holds with $\phi(r)=r^\alpha$. Note that the corresponding generator can be written as follows
\[
{\cal L}u(x)=\int_{\bR^2}\big(u(x+z)-u(x)-\nabla u(x)\cdot z{\bf1}_{\{|z|<1\}}\big)\,\nu(dz).
\]For $g\in C_b^2(\bR^2)$ with
$0\le g\le 1$, let $g_r(y)=g(y/r)$ for $r>0$. Then, by similar computations as in
\cite[Lemma 13.4.1]{KSV}, we have $|{\cal L}g_r|\le c_1r^{-\alpha}$, and so
$\bP^0(\tau_{B(0,r)}\le t)\le c_2t/r^\alpha$ for all $t,r>0$. This implies
\[
\bE^0[\tau_{B(0,r)}]\ge \frac{r^\alpha}{2c_2}\bP^0(\tau_{B(0,r)}\ge r^\alpha/(2c_2))\ge
\frac{r^\alpha}{4c_2},
\]
so that (since the process is the L\'evy process) $\E_{\phi,\ge}$ holds.
Next we have by the L\'{e}vy system formula,
\begin{align*}
\bP^0(\tau_{B(0, r)}\leq r^\alpha)
&\geq
\bP^0(X \mbox{ hits }  B(0, 6r)\setminus B(0, 3r) \hbox{ by time } r^\alpha) \\
&\ge \bP^0(X_{r^\alpha\wedge \tau_{B(0,r)}}\in B(0, 6r)\setminus B(0, 3r))\\
&= \bE^0\left[\int_0^{r^\alpha\wedge \tau_{B(0,r)}}\nu ((B(0, 6r)\setminus B(0, 3r))-X_s)\,ds\right]\\
&\ge  \nu (B(0, 5r)\setminus B(0, 4r)) \, \bE^0[r^\alpha\wedge \tau_{B(0,r)}] \\
&\geq   \frac{c_3}{r^\alpha} \bE^0[r^\alpha\wedge \tau_{B(0,r)}] \\
&\ge \frac{c_3}{r^{\alpha}}\cdot \frac{r^\alpha}{2c_2} \,
 \bP^0 \left(\tau_{B(0,r)}\ge {r^\alpha}/ (2c_2) \right)
\ge \frac{c_3}{4c_2}=:c_4.
\end{align*}
It follows that $\bP^0(\tau_{B(0,r)}> r^\alpha)\le 1-c_4$.
Iterating this as in the proof of
Proposition \ref{pi-e}(ii), we obtain $\E_{\phi,\le}$.
\end{example}

Though the following example is not
in the framework of our paper since the L\'evy measure is singular to the Lebesgue measure on $\bR^d$,
it illustrates that  in the context of symmetric jump processes,
$\EHI$ in general does not follow from $\EHR$ and $\E_\phi$ alone.

 \begin{example}\label{E:82-3-0}
{\bf ($\EHR$ and $\E_\phi$ do not imply $\EHI$
 nor
$\PHI(\phi)$.)}\rm \quad
Let $M=\bR^3$ and $0<\alpha <2$. Consider a symmetric process
$X_t=(X^{(1)}_t, X^{(2)}_t, X^{(3)}_t)$, where $X^{(i)}_t$,
 $i=1,2,3$, are independent
$1$-dimensional symmetric $\alpha$-stable processes.
In \cite{BC}, it is proved that $X=\{X_t; t\geq 0\}$ satisfies $\EHR$ and $\E_\phi$ with $\phi(r)=r^\alpha$, but
$\EHI$ and, consequently $\PHI (\phi)$, fails too.
In addition,
in this case one can easily see that $\UJS$ does not hold.  (We note that
in \cite{BC}, the authors discussed more general processes on $\bR^d$ that are expressed by
a system of stochastic differential equations $dX_t=A(X_{t-})\,dZ_t$, where
$Z^{(i)}_t$, $1\le i\le d$, are independent $1$-dimensional symmetric $\alpha$-stable processes
and $A$ is a matrix-valued function which is bounded, continuous and non-degenerate.) We also note that for this example, $\PI(\phi)$ and $\SCSJ(\phi)$ are satisfied by \cite[Example 4]{DK} and Remark \ref{R:csjrem}, respectively.
\end{example}

\medskip

\noindent{\bf Acknowledgements}. We thank Dr. M. Murugan for spotting
the inconsistent use of time direction in Definition \ref{PER}\,(iii) of
 an earlier version of this paper.
We also thank Prof. M. Kassmann and Dr. T. Schulze for pointing out an error in  Example \ref{E:1.3}  of
an earlier version of this paper. We are  grateful to
the referees for reading the paper carefully and helpful comments.

\vskip 0.2truein

\noindent {\bf Zhen-Qing Chen}

\smallskip \noindent
Department of Mathematics, University of Washington, Seattle,
WA 98195, USA

\noindent
E-mail: \texttt{zqchen@uw.edu}

\medskip

\noindent {\bf Takashi Kumagai}:

\smallskip \noindent
 Research Institute for Mathematical Sciences,
Kyoto University, Kyoto 606-8502, Japan

\noindent Email: \texttt{kumagai@kurims.kyoto-u.ac.jp}

\medskip

\noindent {\bf Jian Wang}:

\smallskip \noindent
College of Mathematics and Informatics \& Fujian Key Laboratory of Mathematical
Analysis and Applications (FJKLMAA) \& Center for Applied Mathematics of Fujian Province (FJNU), Fujian Normal University, 350007 Fuzhou,
P.R. China.

\noindent Email: \texttt{jianwang@fjnu.edu.cn}

 \end{document}